\documentclass[10pt]{article}

\usepackage{url}
\usepackage{mathtools}
\usepackage{amssymb}
\usepackage{amsthm}
\usepackage{empheq}
\usepackage{latexsym}
\usepackage{enumitem}
\usepackage{eurosym}
\usepackage{dsfont}
\usepackage{yfonts}
\usepackage{appendix}
\usepackage{color}
\usepackage{aliascnt}
\usepackage[unicode]{hyperref}
\usepackage{frcursive}
\usepackage[utf8]{inputenc}
\usepackage[T1]{fontenc}
\usepackage{geometry}
\usepackage{multirow}
\usepackage[colorinlistoftodos]{todonotes}
\usepackage{lmodern}
\usepackage{anyfontsize}
\usepackage{natbib}
\usepackage{cleveref}

\usepackage{amsbsy}

\usepackage{comment}

\setcitestyle{numbers,open={[},close={]}}

\definecolor{red}{rgb}{0.7,0.15,0.15}
\definecolor{green}{rgb}{0,0.5,0}
\definecolor{blue}{rgb}{0,0,0.7}
\hypersetup{colorlinks, linkcolor={red},citecolor={green}, urlcolor={blue}}

\makeatletter \@addtoreset{equation}{section}

\newtheorem{theorem}{Theorem}[section]
\crefname{theorem}{theorem}{theorems}
\Crefname{theorem}{Theorem}{Theorems}

\newaliascnt{assumption}{theorem}
\newtheorem{assumption}[assumption]{Assumption}
\aliascntresetthe{assumption}
\crefname{assumption}{assumption}{assumptions}
\Crefname{assumption}{Assumption}{Assumptions}

\newaliascnt{corollary}{theorem}
\newtheorem{corollary}[corollary]{Corollary}
\aliascntresetthe{corollary}
\crefname{corollary}{corollary}{corollaries}
\Crefname{corollary}{Corollary}{Corollaries}

\newaliascnt{example}{theorem}

\aliascntresetthe{example}
\crefname{example}{example}{examples}
\Crefname{example}{Example}{Examples}

\newaliascnt{exercise}{theorem}

\aliascntresetthe{exercise}
\crefname{exercise}{exercise}{exercises}
\Crefname{exercise}{Exercise}{Exercises}

\newaliascnt{lemma}{theorem}

\aliascntresetthe{lemma}
\crefname{lemma}{lemma}{lemmas}
\Crefname{lemma}{Lemma}{Lemmas}

\newaliascnt{proposition}{theorem}
\newtheorem{proposition}[proposition]{Proposition}
\aliascntresetthe{proposition}
\crefname{proposition}{proposition}{propositions}
\Crefname{proposition}{Proposition}{Propositions}

\newaliascnt{condition}{theorem}

\aliascntresetthe{condition}
\crefname{condition}{condition}{conditions}
\Crefname{condition}{Condition}{Conditions}

\newaliascnt{definition}{theorem}
\newtheorem{definition}[definition]{Definition}
\aliascntresetthe{definition}
\crefname{definition}{definition}{definitions}
\Crefname{definition}{Definition}{Definitions}

\newaliascnt{remark}{theorem}
\newtheorem{remark}[remark]{Remark}
\aliascntresetthe{remark}
\crefname{remark}{remark}{remarks}
\Crefname{remark}{Remark}{Remarks}

\def\no{\noindent}
\def\beq{\begin{eqnarray}}
\def\eeq{\end{eqnarray}}
\def\be*{\begin{eqnarray*}}
\def\ee*{\end{eqnarray*}}

\def \E{\mathbb{E}}

\def \G{\mathbb{G}}
\def \H{\mathbb{H}}

\def \P{\mathbb{P}}
\def \Q{\mathbb{Q}}
\def \R{\mathbb{R}}

\def \Pr{\mathrm{P}}

\def\Ac{{\cal A}}

\def\Ec{{\cal E}}
\def\Fc{{\cal F}}
\def\Gc{{\cal G}}
\def\Hc{{\cal H}}

\def\Jc{{\cal J}}

\def\Lc{{\cal L}}

\def\Pc{{\cal P}}

\def\Rc{{\cal R}}

\def\Uc{{\cal U}}
\def\Vc{{\cal V}}
\def\Wc{{\cal W}}

\def\Yc{{\cal Y}}
\def\Zc{{\cal Z}}

\def\x{\times}

\def\Om{\Omega}

\def\om{\omega}

\def\0{\mathbf{0}}
\def \Ec{\mathcal{E}}

\def\normeL2#1{\left\|{#1}\right\|_{L^2}}

\def\Nh{\widehat N}

\def\Yh{\widehat Y}

\def\Wh{\widehat W}
\def\Zh{\widehat Z}

\def \Tr{{\rm Tr}}

\def \Zbb{\mathbf{Z}}

\def \1{\mathds{1}}

\def \xbb {\boldsymbol{x}}

\def \Zbb{\mathbf{Z}}

\def\Ir{{\rm I}}

\def\Rr{{\rm R}}
\def\Kr{{\rm K}}

\def \Rr{\mathrm{R}}

\DeclareUnicodeCharacter{014D}{\=o}
 \title{

 W-BSDEs with Common Noise: Well-posedness, Characterization, and Consequences
 }

\author{
 Mao Fabrice Djete\footnote{\'Ecole Polytechnique Paris, Centre de Math\'ematiques Appliqu\'ees, mao-fabrice.djete@polytechnique.edu. The author benefits from the financial support of the Chairs {\it Financial Risk} and {\it Finance and Sustainable Development}}
 \and
 Mattia Martini \footnote{\'Ecole Polytechnique Paris, Centre de Math\'ematiques Appliqu\'ees, mattia.martini@polytechnique.edu. The author thanks the support of the French ANR PEPR Math-Vives project MIRTE ANR-23-EXMA-0011}
 }

\begin{document}

\maketitle

\begin{abstract}
We study a notion of backward stochastic differential equation on the Wasserstein
space, called a W-BSDE. This work extends to the common noise setting recent
results of the first author. The equation is formulated in Markovian form along
admissible flows of conditional laws generated by McKean-Vlasov dynamics with
common noise. A solution is a pair of fields $(Y,Z)$, where $Y$ is a functional
on the space of probability measures and $Z$ is the corresponding intrinsic
Wasserstein gradient.

Our main result shows that the well--posedness of the W-BSDE can be reduced to a
characteristic equation for a field $p$. Whenever this equation admits a suitable
solution, we construct a W-BSDE solution and identify $Z$ with both the
characteristic field $p$ and the intrinsic gradient of $Y$. This gives an
existence result for the W-BSDE and, at the same time, a probabilistic
construction of a functional on the Wasserstein space together with its intrinsic
gradient.

We then develop several consequences of the theory. We obtain verification
principles for McKean-Vlasov control problems and for mean field games with common
noise, without relying on convexity or separability assumptions. We also show that the
characteristic field admits a dynamic representation through a McKean-Vlasov
forward-backward SDE. The first component $Y$ solves the associated semilinear PDE on the space of
probability measures in the viscosity sense, including the second-order term
induced by common noise. Under additional smoothness, $Y$ is a classical
solution. Finally, we
prove that W-BSDEs arise as intrinsic mean--field limits of classical
finite--dimensional BSDEs written on empirical measures, with convergence of both
the value component and the rescaled martingale integrands.

\end{abstract}

\vspace{3mm}

\vspace{3mm}
\no{\bf MSC2010.} 60K35, 60H30, 91A13, 91A23, 91B30.

\section{Introduction}\label{sec:intro}
Since the seminal works of \citeauthor*{bismut1973conjugate}
\cite{bismut1973conjugate} and \citeauthor*{pardoux1990adapted}
\cite{pardoux1990adapted}, backward stochastic differential equations
(BSDEs) have become a central tool in stochastic analysis, stochastic control,
mathematical finance, and in the probabilistic study of nonlinear partial
differential equations. We refer, among many others, to
\citeauthor*{peng1991probabilistic}~\cite{peng1991probabilistic},
\citeauthor*{pardoux1992backward}~\cite{pardoux1992backward}, and
\citeauthor*{el1997backward}~\cite{el1997backward} for early developments and
applications of the theory.
To recall the basic Markovian picture, let $T>0$, let $f:[0,T]\times\R^d\times\R^d\longrightarrow \R$
be a generator, and let $g:\R^d\longrightarrow \R$
be a terminal condition. Given a Brownian motion $W$ and the forward process $X_t=X_0+W_t,\quad t\in[0,T],$
the associated BSDE consists in finding an adapted pair $(Y,Z)$ such that
\begin{align} \label{eq:def}
    Y_t
    =
    g(X_T)
    +
    \int_t^T f(s,X_s,Z_s)\,\mathrm{d}s
    -
    \int_t^T Z_s\cdot\mathrm{d}W_s,
    \qquad t\in[0,T].
\end{align}
Here $Y$ is real-valued, while $Z$ is $\R^d$-valued and plays the role of
the martingale integrand.
The initial framework of \eqref{eq:def} has since been considerably extended.
Important developments include BSDEs with generators of quadratic or
superquadratic growth, see for instance \citeauthor*{kobylanski2000backward}
\cite{kobylanski2000backward}; reflected BSDEs with obstacle constraints, as in
\citeauthor*{el1997reflected}~\cite{el1997reflected}; BSDEs driven by Lévy
processes or, more generally, by martingales with jumps; and second-order BSDEs,
which are designed to incorporate volatility uncertainty and nonlinear
expectations, see \citeauthor*{cheridito2007second}
\cite{cheridito2007second} and \citeauthor*{soner2012wellposedness}
\cite{soner2012wellposedness}. Systems of BSDEs and multidimensional BSDEs have
also been widely studied, notably in relation with stochastic games, equilibrium
problems, and risk-sharing models. These developments illustrate the flexibility
of the BSDE methodology and its ability to encode both analytic and probabilistic
structures.

\smallskip
In the Markovian finite-dimensional setting considered above, the coefficients
depend on the past only through the current state $X_t$. This allows us to reformulate the problem equivalently in feedback form. More precisely, one seeks Borel maps
\[
    u:[0,T]\times\R^d\longrightarrow \R,
    \qquad
    v:[0,T)\times\R^d\longrightarrow \R^d,
\]
satisfying suitable integrability conditions and such that
\begin{align} \label{eq:markovian_def}
    u(t,X_t)
    =
    g(X_T)
    +
    \int_t^T f\big(s,X_s,v(s,X_s)\big)\,\mathrm{d}s
    -
    \int_t^T v(s,X_s)\cdot\mathrm{d}W_s.
\end{align}
Under standard regularity assumptions, the solution $(u,v)$ satisfies
\[
    Y_t=u(t,X_t),
    \qquad
    Z_t=v(t,X_t).
\]
Moreover, the map $u$ admits a spatial derivative and the second component of
the BSDE is given by
\[
    Z_t=\nabla_x u(t,X_t)=v(t,X_t),
\]
see for instance \citeauthor*{ma2002representation}~\cite[Theorem 4.2.]{ma2002representation}.
Thus, in the Markovian case, the BSDE does not only represent the value field
$u$; it also identifies its gradient through the martingale integrand.

\smallskip
This relation between BSDEs, gradients, and semilinear parabolic PDEs is one of
the main reasons why BSDE techniques are powerful in stochastic control and
related fields. Indeed, when $u$ is sufficiently smooth, applying Itô's formula
to $u(t,X_t)$ shows that $u$ solves the semilinear PDE
\[
    \partial_t u(t,x)
    +
    \frac12 \Delta u(t,x)
    +
    f\big(t,x,\nabla_x u(t,x)\big)
    =
    0,
    \qquad
    u(T,x)=g(x).
\]
Thus, among other things, BSDEs give a probabilistic construction of the PDE solution, and the
process $Z$ identifies its spatial gradient. A key advantage of this approach is
that BSDE well-posedness often requires weaker regularity assumptions than a
classical PDE formulation.

\medskip
{\bf \large Our contribution}.
The goal of this paper is to develop an analogue of this picture when the relevant
state variable is no longer a point in a finite-dimensional space, but a probability
measure. This situation is natural in McKean-Vlasov models. In particular, in the
presence of common noise, the object which evolves in a Markovian way is the
conditional law of the state given the common noise. Thus the natural state space is
the Wasserstein space of probability measures. We introduce and study a notion of
backward stochastic differential equation on this space, which we call a
Wasserstein BSDE, or W-BSDE. The approach is motivated by the \emph{Markovian} definition \Cref{eq:markovian_def} rather than the usual formulation \Cref{eq:def}. A preliminary version, without common noise, has been introduced in \cite{djete2025notionbsdewassersteinspace}.
The basic idea is the following (see \Cref{def:W-BSDE} for more details). Given a generator
\[
        \Fc:[0,T]\times \mathbb R^d\times \mathcal P(\mathbb R^d)\times \mathbb R^d
        \longrightarrow \mathbb R
\]
and a terminal condition
\[
        G:\mathcal P(\mathbb R^d)\longrightarrow \mathbb R,
\]
we look for two deterministic fields
\[
        Y:[0,T]\times \mathcal P(\mathbb R^d)\longrightarrow \mathbb R,
        \qquad
        Z:[0,T)\times \mathbb R^d\times \mathcal P(\mathbb R^d)
        \longrightarrow \mathbb R^d,
\]
such that, with $(\Gc_t)_{t \in [0,T]}$ the common noise filtration, along every admissible flow of conditional laws
\[
        \mu_t=\mathcal L(X_t\mid \mathcal G_t),
\]
the process $Y(t,\mu_t)$ satisfies a backward identity. The field $Z$ plays the role
of the intrinsic gradient of $Y$ with respect to the measure variable. More
precisely, under suitable assumptions, we prove the identification
\[
        Z(t,x,m)=\nabla_x\delta_m Y(t,m)(x),
\]
where $\delta_mY$ denotes the linear functional derivative of $Y$ with respect to
the measure argument. In this sense, the W-BSDE is the measure--dependent counterpart
of a Markovian BSDE: the Markovian state is the conditional law, and the analogue of
the martingale integrand is the Wasserstein gradient field.

\smallskip
A central point of the paper is that the well--posedness of the W-BSDE can be
reduced to a characteristic equation for the gradient field (see \Cref{thm:existence_WBSDE_gradient} for more details ). If $p(t,x,m)$ denotes
the candidate value of $Z(t,x,m)$, we associate with $p$ a McKean-Vlasov flow
\[
        \ell^{p,t,m}_s
        =
        \mathcal L(S^{t,m}_s\mid \mathcal G_s),
\]
whose drift is determined by $\nabla_z\Fc$ evaluated at the field $p$. The
characteristic equation then expresses $p(t,x,m)$ through a formula
involving the derivative of the terminal condition, the generator, and the linear
functional derivative of the generator with respect to the measure variable. Our
main result states that, whenever this characteristic equation admits a solution, one can construct a W-BSDE solution $(Y,Z)$ and identify
\[
        Z=p=\nabla_x\delta_mY.
\]
Moreover, uniqueness of the characteristic equation implies uniqueness of the
corresponding W-BSDE solution.

\smallskip
Let us stress that this formulation is intrinsically probabilistic. The W-BSDE is
not imposed as a pointwise identity on the Wasserstein space. Rather, it is required
to hold along every admissible flow of conditional laws. This is reminiscent of the
finite-dimensional Markovian BSDE framework, where the equation is tested along
forward state processes.

\smallskip
There is, however, an important difference. In the finite-dimensional setting, it is
often enough to test the BSDE along one fixed non-degenerate forward diffusion:
other admissible drift perturbations can then be recovered by a Girsanov change of
measure. In the present Wasserstein setting, such a reduction is no longer
available in general. The state variable is the conditional law itself, and changing
the drift changes the induced flow of conditional distributions in a way that cannot
be encoded simply by a change of probability measure on a single reference flow.
For this reason, the definition must be imposed along all admissible conditional
law flows $(\mu_t)_{t\in[0,T]}$.

\smallskip
Thus, the definition is genuinely dynamic: it characterizes the evolution of
$Y(t,\mu_t)$ along the family of measure-valued trajectories generated by
McKean-Vlasov dynamics with common noise.

\medskip
{\bf \large Consequences}.
The first consequence of the theory, developed in detail in
\Cref{sec:optimal_control}, is a verification principle for McKean-Vlasov control
problems with common noise. In this interpretation, the component $Y$ plays the
role of the value function on the Wasserstein space, while $Z$ is identified
with its intrinsic gradient. The optimal feedback is then obtained by a pointwise
maximization of the Hamiltonian
\[
        a \mapsto L(t,x,m,a)+b(t,x,m,a)\cdot Z(t,x,m).
\]
More precisely, once the W-BSDE is solved, the candidate optimal control is given
by a measurable maximizer of this expression, evaluated at the current state and
conditional law. The verification argument is then very close to the classical
BSDE proof: the W-BSDE identity gives an upper bound for the payoff of any
admissible control, and equality is obtained for the feedback maximizing the
Hamiltonian.
This provides a direct and probabilistic route to the McKean-Vlasov control
problem with common noise. Compared with maximum principle approaches, see for
instance \citeauthor*{andersson2011maximum}~\cite{andersson2011maximum},
\citeauthor*{buckdahn2011general}~\cite{buckdahn2011general}, and
\citeauthor*{Carmona2013ForwardBackwardSD}~\cite{Carmona2013ForwardBackwardSD},
the present verification method is more direct once the W-BSDE has been
constructed. In particular, it does not require convexity assumptions on the
control problem. The optimality is obtained from the pointwise Hamiltonian
maximization encoded in the W-BSDE, rather than from a variational argument or
from adjoint equations combined with convexity conditions.

\smallskip
The second consequence concerns mean field games with common noise; see
\Cref{sec:mfg} for a detailed discussion. In this case, the flow of measures is no
longer chosen by a controller, but must be consistent with the conditional law of
a representative player optimally responding to this same flow. We show that this
fixed-point structure can be encoded in a W-BSDE defined on an enlarged
Wasserstein space: one component describes the representative player's control
problem, while the other carries the candidate equilibrium flow. Once the
characteristic equation for the equilibrium gradient is solved, the feedback
obtained by maximizing the Hamiltonian generates a mean field game equilibrium,
and the W-BSDE provides the corresponding value representation. As in the control
case, no convexity condition is imposed, and the equilibrium feedback is obtained
directly from the characteristic field.

\smallskip
The third consequence, developed in detail in \Cref{sec:FBSDEs}, is a connection
with McKean-Vlasov forward-backward SDEs. The characteristic field $Z$ can be
recovered dynamically as the backward component of a McKean-Vlasov FBSDE: the
forward equation is driven by $\nabla_z\Fc$, while the backward equation involves
the linear functional derivatives of $G$ and $\Fc$. This gives a
characteristic representation of the W-BSDE, in direct analogy with the classical
link between a Markovian BSDE and the FBSDE associated with the gradient of its
decoupling field. More precisely, if
$(X^{t,m},\Yc^{t,m},\Zc^{t,m})$ denotes the FBSDE solution starting from
$\Lc(X^{t,m}_t \mid \Gc_t)=m$, then, along the characteristic flow,
\[
        \Zc^{t,m}_s
        =
        Z(s,X^{t,m}_s,\mu^{t,m}_s),
        \qquad
        \mu^{t,m}_s=\mathcal L(X^{t,m}_s\mid \mathcal G_s).
\]

The fourth consequence is the connection with PDEs on the Wasserstein space; see
\Cref{sec:PDEs}. The W-BSDE provides first a viscosity solution, in the sense
introduced there, to the semilinear equation on $\Pc(\R^d)$ associated with the
generator $\Fc$ and terminal condition $G$. More precisely, under the general
assumptions of \Cref{thm:existence_WBSDE_gradient}, the first component
\[
        \Uc(t,m):=Y(t,m)
\]
is a viscosity solution of the corresponding Wasserstein PDE. If, in addition, the
characteristic field $p$ enjoys enough smoothness in the space and measure
variables, then this viscosity solution is in fact classical and satisfies
\[
\begin{aligned}
0
&=
\partial_t \Uc(t,m)
+
\int_{\R^d}
\Fc\bigl(t,x,m,\nabla_x\delta_m\Uc(t,m)(x)\bigr)\,m(\mathrm{d}x)
\\
&\quad
+
\frac12
\int_{\R^d}
\Tr\Big[
\bigl(\Ir_d+\sigma_\circ\sigma_\circ^\top\bigr)
\nabla_x^2\delta_m\Uc(t,m)(x)
\Big]\,m(\mathrm{d}x)
\\
&\quad
+
\frac12
\int_{\R^d}\int_{\R^d}
\Tr\Big[
\sigma_\circ\sigma_\circ^\top
\nabla_y\nabla_x\delta_m^2\Uc(t,m)(x,y)
\Big]\,m(\mathrm{d}x)m(\mathrm{d}y),
\end{aligned}
\]
with terminal condition
\[
        \Uc(T,m)=G(m).
\]
The last term is the second-order contribution in the measure variable generated
by the common noise.
Thus the W-BSDE gives a probabilistic construction of solutions to semilinear
PDEs on the Wasserstein space. Under the assumptions of the main theorem, this
construction yields a viscosity solution; under the additional regularity stated in
\Cref{prop:WBSDE_PDE}, it yields a classical solution. This distinction is
important. The viscosity property only concerns the value function $Y$, while the W-BSDE
also identifies its intrinsic gradient,
\[
        Z(t,x,m)=\nabla_x\delta_mY(t,m)(x).
\]
This explains why our assumptions are stronger than those required for notions
based solely on the value function (see for instance \citeauthor*{BERTUCCI2024}~\cite{BERTUCCI2024}, \citeauthor*{cardaliaguet2022splitting}~\cite{cardaliaguet2022splitting},  \citeauthor*{cosso2024master}~\cite{cosso2024master}, \citeauthor*{Daudin01022025}~\cite{Daudin01022025}, \citeauthor*{meteViscosity2024}~\cite{meteViscosity2024}, \citeauthor*{Bayraktar03042025}~\cite{Bayraktar03042025} \dots), but they provide at the same time the value
functional and its Wasserstein gradient.
Apart from the linear-quadratic framework of
\citeauthor*{pham2018bellman}~\cite{pham2018bellman} and the construction in
\citeauthor*{cardaliaguet2015master}~\cite[Theorem~3.12]{cardaliaguet2015master}
using mean field games on the torus with separated Hamiltonian, this appears to be
among the first general frameworks yielding classical solutions of such PDEs on
the space of probability measures. Importantly, common noise is allowed, and
$\Fc$ is not required to have a Hamiltonian structure. A set of sufficient
conditions ensuring the required smoothness is given in the companion paper \citeauthor*{fabricemattiaDecoupling}~\cite{fabricemattiaDecoupling},
under smallness assumptions on the horizon $T$, and/or on the size of
$(\Fc,G)$ and their derivatives.

\smallskip
The fifth consequence, discussed in \Cref{sec:n_player}, is a mean-field limit
result for classical finite-dimensional BSDEs. Given a particle system with
empirical measure
\[
        \mu_t^n=\frac1n\sum_{i=1}^n\delta_{N_t^i},
\]
we consider the associated $n$-dimensional BSDE with terminal condition
$G(\mu_T^n)$ and generator given by the empirical average of $\Fc$. Under
suitable regularity assumptions, we prove that its solution converges to the
W-BSDE evaluated along the empirical measure:
\[
        Y_t^n \simeq Y(t,\mu_t^n),
        \qquad
        nZ_t^{i,n}\simeq Z(t,N_t^i,\mu_t^n).
\]
This identifies the W-BSDE as the intrinsic mean-field limit of a family of
finite-dimensional BSDEs.
When $\Fc$ is a Hamiltonian, that is, when the equation comes from a control
problem, the convergence of the value component $Y^n$ has been extensively
studied. For qualitative convergence, we refer, among others, to
\citeauthor*{lacker2017limit}~\cite{lacker2017limit} in the case without common
noise, to \citeauthor*{djete2019general}~\cite{djete2019general} in the presence
of common noise, and to \citeauthor*{GangboFinite2021}~\cite{GangboFinite2021}
for a PDE approach. Quantitative rates have been obtained in several works,
including
\citeauthor*{CardaliaguetAnAlgebraic2023}~\cite{CardaliaguetAnAlgebraic2023},
\citeauthor*{DAUDIN2024110660}~\cite{DAUDIN2024110660},
\citeauthor*{cardaliaguet2023sharpconvergenceratesmean}
~\cite{cardaliaguet2023sharpconvergenceratesmean},
\citeauthor*{BayraktarConvergence2025}~\cite{BayraktarConvergence2025}, \citeauthor*{cecchindaudinjacksonmartini}~\cite{cecchindaudinjacksonmartini}.
In convex regimes, the natural rate for the value component is typically of order
$1/n$. Without convexity, \cite[Example~2]{DAUDIN2024110660} shows that a
vanishing viscosity approach may lead only to a $1/\sqrt n$ rate, while
\cite[Theorem~2.3]{cardaliaguet2023sharpconvergenceratesmean} shows that the
$1/n$ rate can still be recovered in regions of strong regularity. In the
present work, the $1/n$ rate follows from the existence of a sufficiently smooth
classical solution of the PDE on the space of measures. In the non-convex case,
such regularity is obtained by imposing a smallness condition on the maturity
$T$, and/or on the size of $(\Fc,G)$ and their derivatives; see our companion paper \cite{fabricemattiaDecoupling} for more details. This is consistent with the obstruction exhibited in
\cite[Example~2]{DAUDIN2024110660}.

\smallskip
An important aspect of the framework is that it is not restricted to generators
coming from an optimization problem. Although mean field control and mean field
games are among the main motivations, the W-BSDE and the characteristic equation
are formulated for a general nonlinear generator $\Fc$. In particular, the
construction does not rely on convexity, concavity, or separability assumptions
between the state variable and the measure variable. The central object is the
characteristic equation for the gradient field; once this equation is solved, the
W-BSDE, the Wasserstein PDE, the McKean-Vlasov FBSDE representation, and the
finite-particle approximation follow from the same principle.

\medskip
{\bf \large Organization on the paper}.
In \Cref{sec:W_BSDEs_def}, we introduce the notion of
W-BSDE and discuss its interpretation. In \Cref{sec:BSDE_wellposedness}, we state the main
well-posedness theorem and the characteristic equation for the gradient field.
\Cref{sec:conseq} presents the main consequences of the theory: verification for
McKean-Vlasov control, verification for mean field games, the connection with
McKean-Vlasov FBSDEs, the connection with PDEs on the Wasserstein space, and the
mean-field limit of classical finite-dimensional BSDEs. \Cref{sec:proofs} contains the
proofs.

\medskip
{\bf \large Notations}.
	$(i)$
    Given a Polish space $(E,\Delta)$ and $q \ge 1,$ we denote by $\Pc(E)$ the collection of all Borel probability measures on $E$,
	and by $\Pc_q(E)$ the subset of Borel probability measures $\mu$
	such that $\int_E \Delta(e, e_0)^q  \mu(de) < \infty$ for some $e_0 \in E$. For $q=0$, we write $\circ$ instead of $0$ and simply set $\Pc_\circ(E):=\Pc(E)$. We denote by $\Wc_q$ the corresponding $q$-Wasserstein distance:
\[
\Wc_q(\gamma,\gamma')^{1 \vee q}
:=
\inf_{\pi \in \Pi(\gamma, \gamma')}
\int_{E \x E} \Delta_q(e,e') ~\pi( \mathrm{d}e, \mathrm{d}e'),
~\Delta_q(e,e'):=\Delta(e,e')^q\1_{\{ q \ge 1\}}+1\wedge \Delta(e,e')\1_{\{q=0\}},
\]
with $\Pi(\gamma, \gamma')$ the collection of all coupling probability measures with marginals $\gamma$ and $\gamma'$. In particular, $\Wc_\circ$ is the Kantorovich-Rubinstein distance associated with the truncated metric $\Delta_\circ=1\wedge\Delta$: it metrizes the weak topology, $(\Pc(E),\Wc_\circ)$ is Polish, and $\Wc_\circ\le 1\wedge\Wc_1$.

\medskip
    \noindent $(ii)$
	Given a probability space $(\Om, \Hc, \P)$ supporting a sub-$\sigma$-algebra $\Gc \subset \Hc$ then for a Polish space $E$ and any random variable $\xi: \Om \longrightarrow E$, both the notations $\Lc^{\P}( \xi | \Gc)(\om)$ and $\P^{\Gc}_{\om} \circ (\xi)^{-1}$ are used to denote the conditional distribution of $\xi$ knowing $\Gc$ under $\P$.

\medskip
    $(iii)$ A map $F:\Pc(\R^d)\to\R$ admits a linear functional derivative if there
exists a Borel map $\delta_mF:\Pc(\R^d)\times\R^d\to\R$ such that,
writing $m_\theta=(1-\theta)m+\theta m'$,
\[
    F(m')-F(m)
    =
    \int_0^1\int_{\R^d}
    \delta_mF(m_\theta)(x)\,(m'-m)(\mathrm dx)\,\mathrm d\theta
\]
for every $m,m'\in\Pc(\R^d)$ satisfying $\int_0^1\int_{\R^d}
    |\delta_mF(m_\theta)(x)|\,(m+m')(\mathrm dx)\,\mathrm d\theta<\infty.$
The derivative is defined up to an additive constant independent of $x$.
Its spatial gradient, when defined, is denoted by
$\nabla_x\delta_mF(m)(x)$ and called the Lions derivative.
For maps with additional arguments, $\delta_m$ acts only on the measure
variable, keeping the others fixed, and acts componentwise on vector-valued maps.
The evaluation variable is written separately, as in
$\delta_m\Vc(t,x,m)(y)$. Iteration defines the second derivative
$\delta_m^2\Vc(m)(x)(y)=\delta_m^2\Vc(m)(x,y)$.
Subscripts on $\nabla$ and $D$ identify the arguments with respect to which
differentiation is performed, before evaluation. Thus, in
$\delta_m\Vc(t,x,m)(y)$, $\nabla_x$ and $\nabla_y$ act on the state and
evaluation variables, respectively, regardless of the symbols used for those
arguments.

\section{The notion of W-BSDE} \label{sec:W_BSDEs_def}

We now introduce the main object of the paper. In the McKean-Vlasov setting with
common noise, the natural Markovian state variable is the conditional law of the
state given the common noise. The purpose of this section is to formulate a
backward equation on this space of conditional laws.

\smallskip
Throughout this paper, we fix a dimension $d\ge 1$, a time horizon $T>0$, and a common noise coefficient $\sigma_\circ\in \R^d$. We work on a filtered probability space $(\Om,\H:=(\Hc_t)_{t\in[0,T]},\Hc,\P)$\footnote{\label{footnote_enlarge}The probability space $(\Om,\H,\P)$ is assumed to be sufficiently rich in the following sense: whenever needed, it can support additional independent random variables, such as sequences of i.i.d.\ uniform random variables or Brownian motions, with prescribed laws under $\P$, without requiring us to explicitly enlarge the underlying space.
}
supporting an $\R^d\times \R$-valued $\H$-Brownian motion $(W,W^\circ)=(W_t,W^\circ_t)_{t\in[0,T]}$.
The process $W^\circ$ will represent the common noise. We denote by $\G=(\Gc_t)_{t\in[0,T]}$
the $\P$-completed filtration generated by the common noise, namely the completion of$\bigl(\,\sigma(\,\sigma_\circ\,W^\circ_s:\,0\le s\le t)\bigr)_{t\in[0,T]}.$

\medskip

We shall use the following class of admissible conditional law flows. A $\Pc(\R^d)$-valued, $\G$-adapted, continuous process
\[
    \mu=(\mu_t)_{t\in[0,T]}
\]
belongs to $\Lc$ if there exist an initial time $e \in [0,T]$, a bounded Borel map $B:[0,T] \x \R^d \x \Pc(\R^d) \to \R^d$
and a process $X=(X_t)_{t\in[0,T]}$ such that, for every $t\in[0,T]$,
\[
    \mu_t
    =
    \Lc(X_t\mid \Gc_t)
    =
    \Lc(X_t\mid \Gc_T),\quad \mbox{with}\quad \mu_t=\Lc(X_0)\quad\mbox{for}\quad t \in [0,e],
\]
and
\[
    \mathrm{d}X_t
    =
    B(t,X_t,\mu_t)\,\mathrm{d}t
    +
    \mathrm{d}W_t
    +
    \sigma_\circ\,\mathrm{d}W^\circ_t\quad\mbox{for}\quad t \in [e,T].
\]
Here and throughout, $\sigma_\circ\,\mathrm{d}W^\circ_t$ is understood as an $\R^d$-valued common-noise increment. The existence of an admissible conditional law flow $\mu$ is tied to the well-posedness of the McKean-Vlasov dynamics driven by the drift $B$. We shall
therefore call such a coefficient $B$ an \emph{admissible drift} whenever the equation
for $X$ is well posed. An example of condition is the following: for any $\overline{T}< T$ the map $[0,\overline{T}] \x \R^d \x \Pc(\R^d) \ni(t,x,m) \mapsto B(t,x,m) \in \R^d$ is Lipschitz in $(x,m)$ uniformly with respect to $t\in[0,\overline{T}]$ (see for instance \cite[Theorem A.3.]{djete2019mckean}).

\medskip

Let
\[
    \Fc:[0,T]\times \R^d \times \Pc(\R^d)\times \R^d \longrightarrow \R
\]
be a Borel map, which will play the role of the generator, and let
\[
    G:\Pc(\R^d)\longrightarrow \R
\]
be a Borel map, which will serve as the terminal condition. The precise assumptions imposed on $\Fc$ and $G$ will be specified below.

\medskip

We consider two Borel maps
\[
    Y:[0,T]\times \Pc(\R^d)\longrightarrow \R,
    \qquad
    Z:[0,T)\times \R^d\times \Pc(\R^d)\longrightarrow \R^d.
\]

\begin{definition} \label{def:W-BSDE}
We say that a pair $(Y,Z)$ is a {\rm BSDE solution on the Wasserstein space}
associated with $(\Fc,G)$ if the following property holds.
For every admissible conditional law flow $\mu=(\mu_t)_{t\in[e,T]}\in\Lc$,
associated with an initial time $e$, a drift $B$, and a state process $X$,
the process
\[
    t\longmapsto Y(t,\mu_t)
\]
admits a continuous modification on $[e,T]$, and the integrability condition
\begin{align} \label{eq:integral_cond}
    \E \left[
    \int_{e}^T
    \left(
        \left|\Fc\left(r,X_r,\mu_r,Z(r,X_r,\mu_r)\right)\right|
        +
        \left|Z(r,X_r,\mu_r)\right|
        +
        \left|\E\left[Z(r,X_r,\mu_r)\mid \Gc_T\right]\right|^2
    \right)\mathrm{d}r
    \right]
    <\infty
\end{align}
is satisfied. Moreover, for every $t\in[e,T]$, one has, $\P$-a.s.,
\begin{align} \label{eq:WBSDE_def}
    Y(t,\mu_t)
    =
    G(\mu_T)
    &+
    \E \left[
        \int_t^T
        \Fc\left(r,X_r,\mu_r,Z(r,X_r,\mu_r)\right)
        \,\mathrm{d}r
        \,\middle|\, \Gc_T
    \right]
    \nonumber
    \\
    &-
    \E \left[
        \int_t^T
        Z(r,X_r,\mu_r)\cdot B(r,X_r,\mu_r)
        \,\mathrm{d}r
        \,\middle|\, \Gc_T
    \right]
    -
    \int_t^T
    \E\left[Z(r,X_r,\mu_r)\mid \Gc_T\right]\cdot \sigma_\circ
    \,\mathrm{d}W^\circ_r .
\end{align}
\end{definition}

To simplify the terminology, we shall often refer to $(Y,Z)$ simply as a
solution to the Wasserstein BSDE, or {\rm W-BSDE}, associated with $(\Fc,G)$.

\begin{remark}
Let us make several comments on the meaning of the previous definition.

\smallskip

$(i)$ The terminology {\rm BSDE} comes essentially from the fact that the
identity {\rm \Cref{eq:WBSDE_def}} is written in a backward form: the quantity on
the left-hand side is $\Gc_t$-adapted, whereas the quantities appearing on
the right-hand side are expressed in terms of the terminal $\sigma$-field
$\Gc_T$ and of integrals over the time interval $[t,T]$.

It is important to stress that, although the equation contains a stochastic
integral with respect to the common noise $W^\circ$, the randomness of the
couple $(Y,Z)$ at time $t$ comes only through the random measure $\mu_t$.
In other words, once the conditional law $\mu_t$ is fixed, the fields
$Y(t,\mu_t)$ and $Z(t,\cdot,\mu_t)$ are evaluated in a Markovian way. Thus,
in comparison with the classical Markovian BSDE setting, the relevant
Markovian variable is no longer the state $X_t$ itself, but rather its
conditional law $\mu_t$. This is precisely the natural state variable in the
McKean-Vlasov framework, especially in the presence of common noise.

\medskip
$(ii)$ A reader familiar with classical BSDEs might find it less natural to
work directly in such a Markovian formulation. However, as will be seen in {\rm \Cref{thm:existence_WBSDE_gradient}}, even in this Markovian setting the well-posedness of the equation already
raises significant difficulties. At the same time, this formulation provides
a useful probabilistic object which, to the best of our knowledge, allows one
to obtain results that do not seem directly accessible by other existing
methods.

\smallskip
One could of course try to formulate a path-dependent version of the same
notion. We do not pursue this direction here and leave it for future work.
Indeed, such an extension raises additional questions. In the classical
setting, path-dependence concerns the trajectory of a finite-dimensional
state process $X_t$. In the present framework, the state variable is the
conditional law $\mu_t$, which takes values in an infinite-dimensional space
of probability measures. Consequently, making the equation path-dependent
would require a careful treatment of paths of measures, rather than paths of
points in $\R^d$.

\medskip

$(iii)$ The backward identity {\rm\Cref{eq:WBSDE_def}} is imposed only along
admissible flows of conditional laws. More precisely, it is tested on every
flow $\mu$ generated by an admissible dynamics starting from an initial time
$e$ and driven by a drift $B$. Thus, the definition is not merely pointwise on
the Wasserstein space: it is a dynamic identity, required to hold along all
admissible trajectories in that space.

\smallskip
The integrability condition {\rm\Cref{eq:integral_cond}} also indicates the class
of probability measures on which the equation is meaningful. For instance, if
the growth of $\Fc$ or $Z$ requires $q$-polynomial moment estimates along the
admissible flows, then the relevant state space is effectively
$\Pc_q(\R^d)$, even if the maps are formally defined on $\Pc(\R^d)$. In the
present paper we mostly work with bounded coefficients and bounded data, so
this distinction will not play an essential role, and it will be enough to
work on $\Pc(\R^d)$.

\medskip

$(iv)$ If
\[
    \E\left[\int_e^T
    \left|Z(r,X_r,\mu_r)\right|^2\,\mathrm{d}r\right]<\infty .
\]
Then \Cref{eq:WBSDE_def} can be written in the more compact form
\begin{align*}
    Y(t,\mu_t)
    =
    G(\mu_T)
    +
    \E \left[
        \int_t^T
        \Fc\left(r,X_r,\mu_r,Z(r,X_r,\mu_r)\right)
        \,\mathrm{d}r
        \,\middle|\, \Gc_T
    \right]
    -
    \E \left[
        \int_t^T
        Z(r,X_r,\mu_r)\cdot
        \,\mathrm{d}X_r
        \,\middle|\, \Gc_T
    \right].
\end{align*}
Indeed, along the admissible dynamics,
\[
    \mathrm{d}X_r
    =
    B(r,X_r,\mu_r)\,\mathrm{d}r
    +
    \mathrm{d}W_r
    +
    \sigma_\circ\,\mathrm{d}W^\circ_r .
\]
The idiosyncratic martingale part disappears after conditioning with respect
to $\Gc_T$, while the common-noise part produces precisely the stochastic
integral appearing in {\rm\Cref{eq:WBSDE_def}}. This compact form emphasizes the
analogy with classical BSDEs, where the $Z$ component is paired with the
forward dynamics of the state process.
\end{remark}

\medskip
Before starting to discuss the well-posedness of the notion of W-BSDE we have introduced, let us mention what we will call uniqueness.
We say that the {\rm W-BSDE} associated with $(\Fc,G)$ has a unique solution if, whenever $(Y^1,Z^1)$ and $(Y^2,Z^2)$ are two solutions, then for every $\mu\in\Lc$ satisfying the integrability condition \Cref{eq:integral_cond},
\[
    Y^1(t,\mu_t)=Y^2(t,\mu_t),
    \qquad t\in[e,T],
\]
up to indistinguishability, and
\[
    Z^1(t,x,\mu_t)=Z^2(t,x,\mu_t),
    \qquad
    \text{for a.e. }(t,x,\omega)\in [e,T]\times \R^d\times\Om.
\]

We end this section with a simple comparison principle, which in particular gives
uniqueness of the first component $Y$. Let
$(\Fc^i,G^i)_{i=1,2}$ be two bounded generators and terminal conditions, and let
$(Y^i,Z^i)_{i=1,2}$ be corresponding W-BSDE solutions.
Assume that $\Fc^1$ is Lipschitz continuous in $(m,z)$,
uniformly in $(t,x)$, and that, for every $\overline T<T$, the maps
\begin{align} \label{eq:cond_comparison}
    m\longmapsto Z^i(t,x,m),
    \qquad i=1,2,
\end{align}
are Lipschitz continuous, uniformly in $(t,x)\in[0,\overline T]\times\R^d$.
Then the following comparison result holds.

\begin{proposition} \label{prop:comparison}
    Under the previous assumptions, if $G^1\le G^2$ and
    $\Fc^1\le \Fc^2$, then
    \[
        Y^1(e,m)\le Y^2(e,m),
        \qquad (e,m)\in[0,T]\times\Pc(\R^d).
    \]
\end{proposition}

\begin{remark}
The proposition is the analogue, in the present Wasserstein setting, of the
classical comparison principle for {\rm BSDEs}. The additional assumption
\eqref{eq:cond_comparison} is specific to the measure-dependent formulation. Indeed,
in the proof one compares the two equations along a conditional law flow generated
by a drift chosen in order to cancel the difference of the generators in the
$z$-variable. For this drift to be admissible, it is not enough to assume the
usual Lipschitz continuity of the generator in $z$; one also needs enough
regularity in the measure argument after evaluating the generator along the fields
$Z^i$. This is precisely the role of \eqref{eq:cond_comparison}.

In the classical finite-dimensional case, or in the present framework when the
generator and terminal condition do not depend on the measure variable, this extra
condition disappears. In that situation the characteristic field $Z$ no longer
depends on $m$, see the characterization below, and the comparison argument
reduces to the standard one for Markovian {\rm BSDEs}.
\end{remark}

\section{Well-posedness and characterization of W-BSDE} \label{sec:BSDE_wellposedness}

We now turn to the well-posedness of the W-BSDE introduced above. The main idea is
to characterize the second component $Z$ through a fixed-point equation for a
field $p$. Given such a candidate field $p$, we associate with it a McKean-Vlasov flow whose
drift is determined by $\nabla_z\Fc$ evaluated at $p$. The characteristic
equation below then gives a closed relation for $p$, in the spirit of the
classical representation formula for the gradient of a Markovian BSDE.

The main result of this section states that, once this characteristic equation
admits a solution, one can reconstruct a W-BSDE solution
$(Y,Z)$.
Thus the well-posedness of the W-BSDE is reduced to the well-posedness of the
characteristic equation for its intrinsic gradient.

\smallskip
Let $t\in[0,T]$, $m\in\Pc(\R^d)$, and let
\[
    p=(p^1,\dots,p^d):[0,T)\times \R^d\times \Pc(\R^d)\longrightarrow \R^d
\]
be a Borel map. We associate with $p$ the following conditional law flow. For $s\in[t,T]$, set
\[
    \ell^{p,t,m}_s
    :=
    \Lc(S^{t,m}_s\mid \Gc_T)
    =
    \Lc(S^{t,m}_s\mid \Gc_s),
\]
where $(S^{t,m}_s)_{s\in[t,T]}$ is a process satisfying $\Lc(S^{t,m}_t)=m$ and solving
\begin{align} \label{eq:max_dynamic}
    \mathrm{d}S^{t,m}_s
    =
    \nabla_z\Fc
    \left(
        s,S^{t,m}_s,\ell^{p,t,m}_s,
        p(s,S^{t,m}_s,\ell^{p,t,m}_s)
    \right)\mathrm{d}s
    + \mathrm{d}W_s
    + \sigma_\circ\,\mathrm{d}W^\circ_s,
    \qquad s\in[t,T].
\end{align}
Here $\nabla_z\Fc$ denotes the gradient of $\Fc$ with respect to its last variable. Note that, under the assumptions stated below, this McKean-Vlasov equation is well posed.

\smallskip

For $(t,x)\in[0,T]\times\R^d$, we introduce the shifted process
\begin{equation}\label{eq:defN}
    N^{t,x}_s
    :=
    x+W_{s\vee t}-W_t+\sigma_\circ\bigl(W^\circ_{s\vee t}-W^\circ_t\bigr),
    \qquad s\in[0,T].
\end{equation}
Let $(\Nh,\Wh)$ be an independent copy of $(N,W)$ given $W^\circ$. More explicitly, for each $(t,x)$, we write
\[
    \Nh^{t,x}_s
    :=
    x+\Wh_{s\vee t}-\Wh_t+\sigma_\circ\bigl(W^\circ_{s\vee t}-W^\circ_t\bigr),
\]
whenever such notation is needed. Finally, for each $m\in\Pc(\R^d)$, let $U$ be an $\Hc_0$-measurable random variable, independent of the Brownian noises under consideration, such that $\Lc(U)=m$.

\medskip

Let $(F,d_F)$ be a Polish space, and fix $x_F\in F$. For a Borel map
\[
    f:[0,T]\times\R^d\times\R^d\times\Pc(\R^d)\times\R^d \longrightarrow F,
\]
we introduce a uniform bound combining boundedness and Lipschitz continuity in
the measure and gradient variables:
\[
    {\rm Lip}_T(f)
    :=
    \sup_{(t,x,y,m,z)} d_F\bigl(f(t,x,y,m,z),x_F\bigr)
    +
    \sup_{(t,x,y)}\sup_{(m,z)\neq(m',z')}
    \frac{ d_F\bigl(f(t,x,y,m,z),f(t,x,y,m',z')\bigr) }{ \Wc_\circ(m,m')+|z-z'| }.
\]
When the time horizon is clear from the context, we simply write ${\rm Lip}(f)$.
For the coefficients $(\Fc,G)$, we use this notation for the collection of maps \[ (t,x,y,m,z) \longmapsto \Big( \delta_mG(m)(y), \Fc(t,x,m,z), \delta_m\Fc(t,x,m,z)(y), \nabla_z\Fc(t,x,m,z) \Big), \] with the convention that each component is tested only in the variables on which it depends. We then set \[ \underline{{\rm Lip}}_T(\Fc,G) := {\rm Lip}_T \bigl( \delta_mG,\Fc,\delta_m\Fc,\nabla_z\Fc \bigr) \in[0,\infty]. \] Thus $\underline{{\rm Lip}}_T(\Fc,G)$ records, uniformly in time and in the spatial variables, the boundedness and the Lipschitz dependence of the terminal derivative, the generator, its measure derivative, and its $z$-gradient with respect to the measure and gradient arguments.

\begin{assumption} \label{assum:gradient}
    The time horizon $T$ and the coefficients $(\Fc,G)$ are such that there
    exists a Borel map
    \[
        p:[0,T)\times \R^d\times \Pc(\R^d)\longrightarrow \R^d,
    \]
    unique in the class of maps satisfying the following properties.

    \smallskip
    First, for every $\overline T<T$, there exist constants $\Kr\in(0,\infty)$
    and $\alpha\in(0,1)$, depending only on $\overline T$, $T$ and
    $\underline{{\rm Lip}}_T(\Fc,G)$, and locally bounded as functions of these
    quantities, such that
    \begin{align}
        |p(t,x,m)-p(t',x',m')|
        \le
        \Kr
        \left(
            |t-t'|^\alpha
            +
            |x-x'|
            +
            \Wc_\circ(m,m')
        \right),
        \label{eq:regularity_p_assumption}
    \end{align}
    for all
    \[
        (t,t',x,x',m,m')
        \in
        [0,\overline T]^2
        \times(\R^d)^2
        \times\Pc(\R^d)^2.
    \]

    \smallskip
    Second, $p$ satisfies the following fixed-point relation. For every
    $i=1,\dots,d$ and every
    \[
        (t,x,m)\in[0,T)\times\R^d\times\Pc(\R^d),
    \]
    one has
    \begin{align} \label{eq:chac:z}
        p^i(t,x,m)
        &=
        \E \left[
            \delta_m G(\ell^{p,t,m}_T)(N^{t,x}_T)
            \frac{W^i_T-W^i_t}{T-t}
        \right]
        \nonumber
        \\
        &\quad+
        \E \left[
            \int_t^T
            \Fc
            \left(
                s,N^{t,x}_s,\ell^{p,t,m}_s,
                p(s,N^{t,x}_s,\ell^{p,t,m}_s)
            \right)
            \frac{W^i_s-W^i_t}{s-t}\,\mathrm{d}s
        \right]
        \nonumber
        \\
        &\quad+
        \E \left[
            \int_t^T
            \delta_m\Fc
            \left(
                s,S^{t,m}_s,\ell^{p,t,m}_s,
                p(s,S^{t,m}_s,\ell^{p,t,m}_s)
            \right)(\Nh^{t,x}_s)
            \frac{\Wh^i_s-\Wh^i_t}{s-t}\,\mathrm{d}s
        \right].
    \end{align}
\end{assumption}

\begin{theorem} \label{thm:existence_WBSDE_gradient}
Assume that {\rm \Cref{assum:gradient}} holds. Then there exists a solution
$(Y,Z)$ of the {\rm W-BSDE} associated with $(\Fc,G)$ in the sense of
{\rm \Cref{def:W-BSDE}}. Moreover, this solution enjoys the following
regularity and identification properties.
For every $t\in[0,T)$, the map
\[
    m\longmapsto Y(t,m)
\]
admits a differentiable linear functional derivative, denoted by
\[
    (x,m)\longmapsto \nabla_x\delta_m Y(t,m)(x).
\]
In addition, the field $Z$ is a solution of the characteristic equation
{\rm \Cref{eq:chac:z}}, and one has, for every $(t,m)\in[0,T)\times\Pc(\R^d)$,
\[
    \nabla_x\delta_m Y(t,m)(x)
    =
    Z(t,x,m) = p(t,x,m),
    \qquad m(\mathrm{d}x)\text{-a.e.}
\]

\end{theorem}

\begin{remark}
$(i)$ {\rm \Cref{thm:existence_WBSDE_gradient}} gives both a well-posedness result and
a characterization of the {\rm W-BSDE}, under {\rm \Cref{assum:gradient}}. This
assumption is formulated directly in terms of the characteristic equation for the
gradient field. In particular, it can be verified by a Banach fixed point argument
under natural smallness conditions on the horizon $T$, and/or on the size of
$(\Fc,G)$ and their derivatives; see {\rm \cite{fabricemattiaDecoupling}} for a detailed
discussion.
Moreover, the connection with McKean-Vlasov {\rm FBSDEs}
developed in {\rm \Cref{sec:FBSDEs}} suggests that other existence mechanisms may
also be available, for instance under convexity or monotonicity assumptions.

\smallskip
Because of the central role played by {\rm \Cref{eq:chac:z}}, we prefer to
formulate the condition as an assumption. Indeed, a look at the proof, see {\rm \Cref{sec:proofs}}, shows that the well-posedness of the pair $(Y,Z)$ is essentially reduced
to the existence of a solution to the characteristic equation
{\rm \Cref{eq:chac:z}}. It is also worth mentioning that the regularity
condition {\rm \Cref{eq:regularity_p_assumption}} can clearly be weakened.
Moreover, since there is no particular reason to believe that the Banach fixed
point method or the {\rm FBSDE} formulation provides optimal conditions, stating the result under
{\rm \Cref{assum:gradient}} seems more appropriate.

\medskip

$(ii)$ The regularity of $Y$ may be surprising or strong at first sight. As mentioned in the introduction, this property
is, however, quite natural in the theory of {\rm BSDEs}, especially in the
Markovian setting. Indeed, let
\[
    (\Yh,\Zh)=(\Yh,\Zh^1,\ldots,\Zh^d)
\]
be a solution of a classical {\rm BSDE} with generator
$f:[0,T]\times\R^d\times\R^d\to\R$ and terminal condition
$g:\R^d\to\R$, namely such that
\[
    \Yh_t,\Zh_t\in\sigma\{W_s:\,s\in[0,t]\},
    \qquad
    \E\left[
        \sup_{t\in[0,T]}|\Yh_t|
        +
        \int_0^T|\Zh_s|^2\,\mathrm{d}s
    \right]<\infty,
\]
and
\begin{align*}
    \Yh_t
    =
    g(W_T)
    +
    \int_t^T f(s,W_s,\Zh_s)\,\mathrm{d}s
    -
    \int_t^T \Zh_s\cdot\mathrm{d}W_s .
\end{align*}
Under bounded Lipschitz assumptions on $f$ and boundedness of $g$, it is
well known, see for instance {\rm \cite[Theorem 4.2.]{ma2002representation}}, that there exists a deterministic
map $u:[0,T]\times\R^d\to\R$ such that
\[
    \Yh_t=u(t,W_t),
    \qquad
    \Zh_t=\nabla_x u(t,W_t)=:v(t,W_t).
\]
Moreover, the field $v=(v^1,\cdots,v^d)$ can be characterized through an equation of the form
\begin{align} \label{eq:classical_chac:z_bismut}
    v^i(t,x)
    =
    \E \Bigg[
        g(x+W_T-W_t)\frac{W^i_T-W^i_t}{T-t}
        +
        \int_t^T
        f\big(s,x+W_s-W_t,v(s,x+W_s-W_t)\big)
        \frac{W^i_s-W^i_t}{s-t}
        \,\mathrm{d}s
    \Bigg], \quad i=1,\ldots,d.
\end{align}

{\rm \Cref{eq:classical_chac:z_bismut}} can be viewed as the
exact analogue of {\rm \Cref{eq:chac:z}} in the present setting, where the
state variable is no longer a point $x\in\R^d$ but a probability measure
$m\in\Pc(\R^d)$. In the classical Markovian case, the well-posedness of the
{\rm BSDE} is closely related to the well-posedness of the equation satisfied by
$v=\nabla_x u$. In the same way, for the {\rm W-BSDE}, the well-posedness of
the pair $(Y,Z)$ is closely related to the well-posedness of the
characteristic equation {\rm \Cref{eq:chac:z}}. Our notion of {\rm W-BSDE}
therefore shares the same basic structure as the classical Markovian BSDE,
with the essential difference that the Markovian variable is now the
conditional law and the role of $\nabla_x u$ is played by
$\nabla_x\delta_m Y$. Thus, the {\rm W-BSDE} provides a probabilistic way to construct both a value
function on the Wasserstein space and its intrinsic gradient.

\medskip Let us emphasize that, compared with the classical Markovian BSDE setting, additional assumptions are needed for the well-posedness of the W-BSDE. These extra requirements come from the dependence on the measure variable, and not from the spatial variable itself. In particular, when $\delta_mG$ and $\Fc$ do not depend on the measure argument, the characteristic equation reduces to the classical one, and we recover the usual well-posedness conditions.

\medskip

$(iii)$ The well-posedness of our notion of {\rm W-BSDE} has interesting
consequences, which will be developed in {\rm \Cref{sec:conseq}}. In
particular, this object provides new results, especially with respect to the
conditions required in the study of subject with mean field components (mean field control, mean field games, $\cdots$), and it also gives a new,
and sometimes simpler, probabilistic viewpoint on known results.
\end{remark}

We next present a stability result which is a straight application of the Arzel\`a-Ascoli-Theorem applied to the sequence of gradients.
Let $\bigl(T^k,\Fc^k,G^k,p^k\bigr)_{k\ge 1}$ be a sequence such that, for every $k\ge 1$, the triple $(T^k,\Fc^k,G^k)$ satisfies {\rm \Cref{assum:gradient}}, with associated characteristic map
\[
    p^k:[0,T^k)\times \R^d\times \Pc(\R^d)\longrightarrow \R^d.
\]
Assume that there exist a time horizon $T>0$, a generator
\[
    \Fc:[0,T]\times\R^d\times\Pc(\R^d)\times\R^d\to\R,
\]
and a terminal condition
\[
    G:\Pc(\R^d)\to\R
\]
such that
\[
    T^k\longrightarrow T,
    \qquad
    \underline{{\rm Lip}}_{T^k}(\Fc^k,G^k)
    \longrightarrow
    \underline{{\rm Lip}}_{T}(\Fc,G).
\]

\medskip

Without loss of generality, we extend the maps $\Fc^k$ and $p^k$ to a common time interval $[0,\overline T]$, where
\[
    \overline T>\sup_{k\ge1}T^k\vee T,
\]
so that all convergence statements below are understood on a common domain.

\medskip

We further assume that the coefficients converge along every admissible flow. More precisely, for every $m \in \Pc(\R^d)$, one has
\[
    G^k(m)\longrightarrow G(m),
\]
and, for $m(\mathrm{d}x) \otimes\mathrm{d}t$-a.e.,
\[
    \lim_{k\to\infty}
    \sup_{z\in\R^d}
    \left|
        \Fc^k(t,x,m,z)-\Fc(t,x,m,z)
    \right|
    =
    0.
\]
\begin{corollary}
    [Stability]\label{coro:stability}
    Under the assumptions above, there exists a map
    $p:[0,T)\times\R^d\times\Pc(\R^d)\longrightarrow\R^d$ satisfying
    \eqref{eq:regularity_p_assumption} and \eqref{eq:chac:z} with respect to
    $(T,\Fc,G)$, and such that, up to a sub-sequence, for any $\mu \in \Lc$
    associated with initial time $t_0$,
    \begin{align*}
        \lim_{k \to \infty} \E \left[ \int_{t_0}^{T} |p(t,X_t,\mu_t)-p^k(t,X_t,\mu_t)|\,\mathrm{d}t \right] = 0.
    \end{align*}
    Moreover, there exists a {\rm W-BSDE} solution $(Y,Z)$ associated with the
    limiting coefficients $(\Fc,G)$ and maturity $T$ verifying, for each
    $(t,m) \in [0,T) \x \Pc(\R^d)$,
    \begin{align*}
        \nabla_x \delta_m Y(t,m)(x) =p(t,x,m),\quad m(\mathrm{d}x)\text{-a.e.}
    \end{align*}
    If, in addition, \eqref{eq:chac:z} admits at most one solution in the class
    defined by \eqref{eq:regularity_p_assumption}, then $(T,\Fc,G)$ satisfies
    {\rm \Cref{assum:gradient}} and the whole sequence converges.
\end{corollary}
\section{Consequences and applications of W-BSDEs} \label{sec:conseq}

\subsection{W-BSDE and optimal control of McKean-Vlasov equations with common noise} \label{sec:optimal_control}

We now present one of the main motivations for introducing the notion of
{\rm W-BSDE}. Namely, we show that it provides a verification principle for
McKean-Vlasov control problems in the presence of common noise. In this
interpretation, the first component $Y$ plays the role of a value function
on the Wasserstein space, while the second component $Z$ plays the role of
its intrinsic gradient. The optimal feedback is then obtained by maximizing
the corresponding Hamiltonian.

\smallskip
This connection should be understood as the probabilistic counterpart of the
usual link between Markovian BSDEs and stochastic control in finite
dimension. Here, however, the relevant state variable is not the position of
a representative particle, but its conditional law.

\medskip
We consider a McKean-Vlasov control problem with common noise. Let $A$ be
a compact set. We denote by $\Ac$ the set of bounded Lipschitz feedback
controls $\alpha:[0,T]\times \R^d\times \Pc(\R^d)\longrightarrow A,$
and by $\Ac^\circ$  the set of bounded Lipschitz feedback
controls $\alpha^\circ:[0,T]\times \Pc(\R^d)\longrightarrow A$.
Let
\[
    b,L:[0,T]\times \R^d\times \Pc(\R^d)\times A
    \longrightarrow \R^d\times \R
\]
be bounded Borel maps, with $b$ Lipschitz in $(x,m,a)$, uniformly in
$t$. We fix an initial distribution $\nu\in\Pc(\R^d)$.
For any feedback control $\alpha\in\Ac$, we define the controlled
McKean-Vlasov dynamics by $   \Lc(X^{\nu,\alpha}_0)=\nu$,
and
\begin{align}
    \mathrm d X^{\nu,\alpha}_s
    &=
    b\Big(
        s,
        X^{\nu,\alpha}_s,
        \mu^\alpha_s,
        \alpha(s,X^{\nu,\alpha}_s,\mu^\alpha_s)
    \Big)\mathrm ds
    +\mathrm dW_s+\sigma_\circ\,\mathrm dW^\circ_s,
    \qquad
    \mu^\alpha_s
    :=
    \Lc(X^{\nu,\alpha}_s\mid\Gc_s) = \Lc(X^{\nu,\alpha}_s\mid\Gc_T).
    \label{eq:controlled_MKV_feedback}
\end{align}
The corresponding payoff is
\begin{align}
    J(\nu,\alpha)
    :=
    \E\left[
        \int_0^T
        L\Big(
            s,
            X^{\nu,\alpha}_s,
            \mu^\alpha_s,
            \alpha(s,X^{\nu,\alpha}_s,\mu^\alpha_s)
        \Big)\mathrm ds
        +
        G(\mu^\alpha_T)
    \right],
    \label{eq:payoff_MFC_feedback}
\end{align}
and the value function is
\[
    V_{\rm MFC}(\nu)
    :=
    \sup_{\alpha\in\Ac}J(\nu,\alpha).
    \label{eq:value_MFC_feedback}
\]

We shall also consider a second class of controls, where the controller uses
only the common noise. For $\alpha^\circ\in\Ac^\circ$, we define
$X^{\circ,\nu,\alpha^\circ}$ by $\Lc(X^{\circ,\nu,\alpha^\circ}_0)=\nu,$
and
\begin{align}
    \mathrm d X^{\circ,\nu,\alpha^\circ}_s
    &=
    b\Big(
        s,
        X^{\circ,\nu,\alpha^\circ}_s,
        \mu^{\circ,\alpha^\circ}_s,
        \alpha^\circ(s,\mu^{\circ,\alpha^\circ}_s)
    \Big)\mathrm ds
    +\mathrm dW_s+\sigma_\circ\,\mathrm dW^\circ_s,
    \quad
    \mu^{\circ,\alpha^\circ}_s
    :=
    \Lc(X^{\circ,\nu,\alpha^\circ}_s\mid\Gc_s) = \Lc(X^{\circ,\nu,\alpha^\circ}_s\mid\Gc_T).
    \label{eq:controlled_MKV_open_common}
\end{align}
The associated payoff and value are
\begin{align}
    J^\circ(\nu,\alpha^\circ)
    &:=
    \E\left[
        \int_0^T
        L\Big(
            s,
            X^{\circ,\nu,\alpha^\circ}_s,
            \mu^{\circ,\alpha^\circ}_s,
            \alpha^\circ(s,\mu^{\circ,\alpha^\circ}_s)
        \Big)\mathrm ds
        +
        G(\mu^{\circ,\alpha^\circ}_T)
    \right],
    \qquad
    V^\circ_{\rm MFC}(\nu)
    &:=
    \sup_{\alpha^\circ\in\Ac^\circ}J^\circ(\nu,\alpha^\circ).
    \label{eq:value_MFC_common_control}
\end{align}

We introduce the generator
\[
    \Fc(s,x,m,z)
    :=
    \sup_{a\in A}h(s,x,m,z,a),\quad h(s,x,m,z,a)
    :=
    L(s,x,m,a)+b(s,x,m,a)\cdot z.
\]
We assume that a measurable maximizer exists, i.e.\ that there is a measurable
map
$$\widehat a(s,x,m,z)\in\arg\max_{a\in A}h(s,x,m,z,a).$$
For the common-noise-adapted control problem, we assume that the drift has the
form $$b(t,x,m,a)=b_1(t,x,m)+b^\circ(t,m,a),$$ and we set, for $\bar z\in\R^d$,
\[
    h^\circ(s,m,\bar z,a)
    :=
    \int_{\R^d}L(s,x,m,a)\,m(\mathrm dx)+b^\circ(s,m,a)\cdot\bar z .
\]

Let $(Y,Z)$ be a W-BSDE solution associated with $(\Fc,G)$.
For $\alpha^\circ\in\Ac^\circ$, let $(Y^{\alpha^\circ},Z^{\alpha^\circ})$ be a
W-BSDE solution associated with $G$ and the generator
\[
    \Fc^{\alpha^\circ}(s,x,m,z)
    := L\bigl(s,x,m,\alpha^\circ(s,m)\bigr)+b\bigl(s,x,m,\alpha^\circ(s,m)\bigr)\cdot z ,
\]
and set $\bar Z^{\alpha^\circ}(s,m):=\int_{\R^d}Z^{\alpha^\circ}(s,x,m)\,m(\mathrm dx)$.
We now define the candidate optimal dynamics. Assume that there exists a
process $R$, with $\Lc(R_0)=\nu$, satisfying
\begin{align}
     \mathrm dR_s
     &=
     b\Big(
        s,
        R_s,
        \nu_s,
        \widehat a(s,R_s,\nu_s,Z(s,R_s,\nu_s))
     \Big)\mathrm ds
     +\mathrm dW_s+\sigma_\circ\,\mathrm dW^\circ_s,
     \qquad
     \nu_s
     :=
     \Lc(R_s\mid\Gc_s) = \Lc(R_s\mid\Gc_T).
     \label{eq:resolvant_McK}
\end{align}
The corresponding candidate optimal controls is
\[
    \widehat\alpha(s,x,m)
    :=
    \widehat a(s,x,m,Z(s,x,m)).
\]

\begin{proposition}
\label{prop:McK-fromBSDE}
Assume that the resolvent equation
\eqref{eq:resolvant_McK} is well posed.
Then the feedback control $\widehat\alpha$ is optimal for
$V_{\rm MFC}(\nu)$, and
\[
    Y(0,\nu)=V_{\rm MFC}(\nu).
\]
Assume moreover that the drift has the form
$b(t,x,m,a)=b_1(t,x,m)+b^\circ(t,m,a)$, and that $\alpha^\circ\in\Ac^\circ$ satisfies, for every $(s,m)$,
\[
    \alpha^\circ(s,m)\in\arg\max_{a\in A}h^\circ\bigl(s,m,\bar Z^{\alpha^\circ}(s,m),a\bigr).
\]
Then $\alpha^\circ$ is optimal for $V^\circ_{\rm MFC}(\nu)$, and
$Y^{\alpha^\circ}(0,\nu)=V^\circ_{\rm MFC}(\nu)$.

\end{proposition}

\begin{remark}
The first part is a verification result for feedback controls depending on the
state and on the conditional law. The {\rm W-BSDE} identifies the relevant
Wasserstein gradient field $Z$, and an optimal feedback is obtained by maximizing
the Hamiltonian pointwise. The resolvent equation \eqref{eq:resolvant_McK} is then
the closed-loop McKean-Vlasov equation obtained after inserting this feedback
into the dynamics.
As mentioned in the introduction, this gives a direct alternative to maximum
principle methods; see, for instance,
{\rm \cite{andersson2011maximum, buckdahn2011general, Carmona2013ForwardBackwardSD}}. Once the {\rm W-BSDE} is constructed, the
optimal control follows from the Hamiltonian maximization, without requiring
convexity assumptions on the control problem. Note that the Hamiltonian $\Fc$
grows linearly in $z$, and is therefore not bounded in $z$ as required in
$\underline{{\rm Lip}}_T(\Fc,G)$: the existence of $(Y,Z)$ is not a direct
application of {\rm \Cref{thm:existence_WBSDE_gradient}}, but the extension of
the latter to generators with linear growth in $z$ is immediate.

\smallskip

The second part concerns controls adapted only to the common noise: the
controller uses a single action for all states, so the maximization is performed
after averaging, and the optimal control is determined by the conditional mean
$\bar Z^{\alpha^\circ}$ of the gradient field. Since $\Fc^{\alpha^\circ}$ is
affine in $z$, the existence of $(Y^{\alpha^\circ},Z^{\alpha^\circ})$ for a given
$\alpha^\circ$ follows from the extension of {\rm \Cref{thm:existence_WBSDE_gradient}}
to linear growth; the consistency condition on $\alpha^\circ$ is a fixed point,
which we do not address here. A direct characterization of $V^\circ_{\rm MFC}$, without the fixed-point
condition, would require a theory of W-BSDEs whose generator depends on the
conditional law of $Z$. We leave this extension for future work.

\end{remark}

\subsection{W-BSDE and mean field games} \label{sec:mfg}

We have just seen that W-BSDEs provide a verification principle for
McKean-Vlasov control problems with common noise. We now explain that the
same point of view can also be used for mean field games. The situation is
slightly different. In mean field control, the flow of conditional laws is
chosen by the controller through the controlled dynamics. In a mean field
game, the flow of measures is instead an equilibrium object: it must be
consistent with the distribution of a representative player who optimally
responds to this very flow.

\smallskip
The purpose of this section is to show that this fixed-point structure can be
encoded in a W-BSDE. The idea is to enlarge the state space and to consider
a two-component system. The first component represents the state of a
representative player who solves an optimal control problem against a frozen
flow of measures. The second component carries the candidate equilibrium
population dynamics. The W-BSDE then couples these two objects: its first
gradient component gives the optimal feedback of the representative player,
while the second component forces consistency of the equilibrium flow.

\smallskip
This provides a probabilistic verification argument for mean field games with
common noise. In particular, once the characteristic equation for the MFG
gradient admits a solution, the associated feedback
generates an equilibrium flow. The W-BSDE on the enlarged Wasserstein space
then gives the corresponding value representation.

\medskip

Let $A$ be a compact metric space, let $\ell\ge 1$, and set $d:=2\ell.$ The $\R^d$-valued Brownian motion $W$ will be denoted $W=(W^1,W^2)=(W^{1,1},\cdots,W^{1,\ell}, W^{2,1},\cdots,W^{2,\ell})$.
We are given bounded measurable maps
\[
    b:[0,T]\times \R^\ell\times \Pc(\R^\ell)\times A\to \R^\ell,
    \qquad
    L:[0,T]\times \R^\ell\times \Pc(\R^\ell)\times A\to \R,
\]
with $b$ Lipschitz in $(x,m,a)$, uniformly in
$t$, and a terminal reward
\[
    g:\R^\ell\times \Pc(\R^\ell)\to \R.
\]
The common noise coefficient is denoted by $\sigma_\circ\in\R^\ell$, and $\G=(\Gc_t)_{t\in[0,T]}$ is the filtration generated by the common noise.

\medskip

We first recall the notion of mean field game solution used in this section. Let $\alpha:[0,T]\times \R^\ell\times \Pc(\R^\ell)\to A$
be a feedback control which is Lipschitz in $(x,m)$ uniformly in $t \in [0,\overline{T}]$ for each $\overline{T}$. We say that a $\Pc(\R^\ell)$-valued, $\G$-adapted, continuous process
\[
    \mu=(\mu_t)_{t\in[0,T]}
\]
is a mean field game solution associated with the feedback $\alpha$ if
\[
    \mu_t
    =
    \Lc(X^{\alpha,\mu}_t\mid \Gc_t)
    =
    \Lc(X^{\alpha,\mu}_t\mid \Gc_T),
    \qquad t\in[0,T],
\]
where, for any admissible feedback $\beta:[0,T]\times\R^\ell\times\Pc(\R^\ell)\to A$, the state process $X^{\beta,\mu}$ is defined by
\begin{align*}
    \mathrm{d}X^{\beta,\mu}_t
    =
    b\bigl(t,X^{\beta,\mu}_t,\mu_t,\beta(t,X^{\beta,\mu}_t,\mu_t)\bigr)\,\mathrm{d}t
    +
    \mathrm{d}W^1_t
    +
    \sigma_\circ\,\mathrm{d}W^\circ_t,
    \qquad
    \Lc(X^{\beta,\mu}_0)=\mu_0.
\end{align*}
The optimality condition reads
\begin{align*}
    J_\mu(\alpha)
    =
    \sup_{\beta} J_\mu(\beta), \quad J_\mu(\beta)
    :=
    \E \left[
        g(X^{\beta,\mu}_T,\mu_T)
        +
        \int_0^T
        L\bigl(t,X^{\beta,\mu}_t,\mu_t,\beta(t,X^{\beta,\mu}_t,\mu_t)\bigr)\,\mathrm{d}t
    \right].
\end{align*}

\medskip

For $(t,x,\pi,z)\in [0,T]\times\R^\ell\times\Pc(\R^\ell)\times\R^\ell$, define the Hamiltonian
\begin{align*}
    H(t,x,\pi,z)
    :=
    \sup_{a\in A} h(t,x,\pi,z,a),
    \qquad
    h(t,x,\pi,z,a)
    :=
    b(t,x,\pi,a)\cdot z+L(t,x,\pi,a).
\end{align*}
We assume that there exists a measurable maximizer
\[
    \widehat\alpha:[0,T]\times\R^\ell\times\Pc(\R^\ell)\times\R^\ell\to A
\]
such that
\[
    H(t,x,\pi,z)
    =
    h\bigl(t,x,\pi,z,\widehat\alpha(t,x,\pi,z)\bigr),
\]
and we set
\[
    \widehat b(t,x,\pi,z)
    :=
    b\bigl(t,x,\pi,\widehat\alpha(t,x,\pi,z)\bigr).
\]

\medskip

The equilibrium feedback will be obtained through a characteristic equation. Let
\[
    p=(p^1,\dots,p^\ell):[0,T)\times\R^\ell\times\Pc(\R^\ell)\to\R^\ell
\]
be a Borel map. For $(t,m)\in[0,T]\times\Pc(\R^\ell)$, define the conditional law flow
\[
    \Rr^{p,t,m}_s
    :=
    \Lc(Y^{t,m}_s\mid \Gc_T)
    =
    \Lc(Y^{t,m}_s\mid \Gc_s),
    \qquad s\in[t,T],
\]
where $Y^{t,m}$ satisfies $\Lc(Y^{t,m}_t)=m$ and
\begin{align}
    \mathrm{d}Y^{t,m}_s
    =
    \widehat b
    \left(
        s,Y^{t,m}_s,\Rr^{p,t,m}_s,
        p(s,Y^{t,m}_s,\Rr^{p,t,m}_s)
    \right)\mathrm{d}s
    +
    \mathrm{d}W^1_s
    +
    \sigma_\circ\,\mathrm{d}W^\circ_s.
    \label{eq:MFG_equilibrium_flow}
\end{align}
We then consider the fixed-point relation
\begin{align} \label{eq:gradient_MFG}
        p^i(t,x,m)
        &=
        \E \left[
            g(N^{t,x}_T,\Rr^{p,t,m}_T)
            \frac{W^{1,i}_T-W^{1,i}_t}{T-t}
        \right]
        \nonumber
        \\
        &\quad+
        \E \left[
            \int_t^T
            H\left(
                s,N^{t,x}_s,\Rr^{p,t,m}_s,
                p(s,N^{t,x}_s,\Rr^{p,t,m}_s)
            \right)
            \frac{W^{1,i}_s-W^{1,i}_t}{s-t}\,\mathrm{d}s
        \right],
        \qquad i=1,\dots,\ell,
\end{align}
where
\[
    N^{t,x}_s
    :=
    x+W^1_{s\vee t}-W^1_t
    +
    \sigma_\circ\bigl(W^\circ_{s\vee t}-W^\circ_t\bigr).
\]
Equation \eqref{eq:gradient_MFG} is the probabilistic representation of the gradient of the MFG value function along the equilibrium flow. In addition, $p$ satisfies the regularity \Cref{eq:regularity_p_assumption} with the dependence $\underline{{\rm Lip}}_T(\Fc,G)$ being replaced by ${\rm Lip}_T\bigl(g,H,\widehat{b}\bigr)$.

\medskip

We now explain how this MFG structure can be encoded into a W-BSDE on the enlarged space $\R^d=\R^\ell\times\R^\ell$. For $m\in\Pc(\R^d)$, we denote by $m_1$ and $m_2$ its first and second marginals:
\[
    m_1(\mathrm{d}x)
    :=
    m(\mathrm{d}x\times\R^\ell),
    \qquad
    m_2(\mathrm{d}x)
    :=
    m(\R^\ell\times \mathrm{d}x).
\]
Define the terminal condition
\[
    G(m)
    :=
    \int_{\R^\ell}
    g(x,m_2)\,m_1(\mathrm{d}x),
    \qquad m\in\Pc(\R^d).
\]
Given a solution $p$ of \eqref{eq:gradient_MFG}, define the generator
\[
    \Fc_p(t,x,m,z)
    :=
    H(t,x_1,m_2,z_1)
    +
    z_2\cdot
    \widehat b\bigl(t,x_2,m_2,p(t,x_2,m_2)\bigr),
\]
where
\[
    x=(x_1,x_2)\in\R^\ell\times\R^\ell,
    \qquad
    z=(z_1,z_2)\in\R^\ell\times\R^\ell.
\]

\medskip

The construction above separates two roles which are intertwined in the MFG
equilibrium condition. The variable $x_1$ is used to describe the
representative player's optimization problem, while $x_2$ is used to
transport the candidate equilibrium distribution. This is why the terminal
condition depends on the first marginal through the payoff variable and on
the second marginal through the equilibrium population: $G(m)=\int_{\R^\ell} g(x,m_2)\,m_1(\mathrm{d}x).$
Similarly, the generator $\Fc_p$ has two parts. The first one,
\[
    H(t,x_1,m_2,z_1),
\]
is the Hamiltonian of the representative player's control problem when the
population flow is $m_2$. The second one,
\[
    z_2\cdot \widehat b(t,x_2,m_2,p(t,x_2,m_2)),
\]
propagates the equilibrium flow under the feedback obtained from the
characteristic field $p$. Thus, if the first component of the W-BSDE
gradient satisfies
\[
    Z^1_p(t,(x_1,x_2),m)=p(t,x_1,m_2),
\]
then the W-BSDE contains both the optimality condition and the consistency
condition of the MFG. The enlarged state is driven by the idiosyncratic noise $W=(W^1,W^2)$, with
$W^1,W^2$ independent and $\R^\ell$-valued, and by the same scalar common noise
$W^\circ$ acting on both components, so that the common noise coefficient on
$\R^d$ is $\boldsymbol\sigma_\circ:=(\sigma_\circ,\sigma_\circ)\in\R^d .$
Component wise, both coordinates are therefore driven by $\sigma_\circ\,\mathrm{d}W^\circ$,
as in \eqref{eq:MFG_equilibrium_flow}.

\begin{proposition}[MFG equilibrium and W-BSDE]
    Assume that $(T,g,H,\widehat{b})$ is such that {\rm \Cref{eq:gradient_MFG}} admits a unique solution
    \[
        p:[0,T)\times\R^\ell\times\Pc(\R^\ell)\to\R^\ell
    \]
    satisfying the regularity of {\rm \Cref{eq:regularity_p_assumption}}.
    Then the feedback
    \[
        \alpha^\star(t,x,\pi)
        :=
        \widehat\alpha\bigl(t,x,\pi,p(t,x,\pi)\bigr)
    \]
    generates a mean field game equilibrium.
    More precisely, let $\mu$ be the conditional law flow of the state process
    \[
        \mathrm{d}X_t^{\star}
        =
        b\bigl(t,X_t^{\star},\mu_t,\alpha^\star(t,X_t^{\star},\mu_t)\bigr)\,\mathrm{d}t
        +
        \mathrm{d}W^1_t
        +
        \sigma_\circ\,\mathrm{d}W^\circ_t,
        \qquad
        \mu_t=\Lc(X_t^\star\mid\Gc_t).
    \]
    Then $\mu$ is a mean field game solution associated with $\alpha^\star$.  Moreover, suppose that there exists a solution
    \[
        (Y_p,Z_p)=(Y_p,(Z^1_p,Z^2_p))
    \]
    of the W-BSDE on $\Pc(\R^d)$ associated with the coefficients $(\Fc_p,G)$ such that
    \[
        Z^1_p\bigl(t,(x_1,x_2),m\bigr)
        =
        p(t,x_1,m_2),
    \]
    for a.e.\ $(t,x_1,x_2,m)$. Then the W-BSDE provides the corresponding MFG value representation. In particular,
    \[
        Y_p\bigl(0,\mu_0\otimes\mu_0\bigr)
    \]
    coincides with the value of the mean field game starting from the initial distribution $\mu_0=\Lc(X_0)$.
\end{proposition}

\begin{remark}
The proposition should be read as a verification result for {\rm MFGs}. The
characteristic equation \eqref{eq:gradient_MFG} provides the candidate
gradient of the value function of the representative player, evaluated along
the candidate equilibrium flow. Once such a field $p$ is constructed, the
optimal feedback is given by
\[
    \alpha^\star(t,x,\pi)
    =
    \widehat\alpha(t,x,\pi,p(t,x,\pi)).
\]
The proposition then shows that this feedback generates a {\rm MFG} solution.

\smallskip
In this sense, the search for a {\rm MFG} equilibrium is reduced to the
resolution of the fixed-point relation \eqref{eq:gradient_MFG}. This gives a
compact and probabilistic way of constructing $($closed--loop$)$ equilibria, under relatively
mild assumptions, even in the presence of common noise. The common noise case
is usually delicate because the equilibrium flow is random and must be
consistent with the conditional law of the optimally controlled state. Here,
this consistency is encoded directly in the flow
\[
    \Rr^{p,t,m}_s
    =
    \Lc(Y^{t,m}_s\mid\Gc_s),
\]
which is generated by the feedback obtained from $p$.

\smallskip
Compared with the usual approaches based on the master equation $($see  {\rm \citeauthor*{cardaliaguet2015master}~\cite{cardaliaguet2015master}} or {\rm \citeauthor*{GangboMean2022}~\cite{GangboMean2022}}$)$ or on
fixed-point arguments for stochastic flows of measures $(${\rm \citeauthor*{carmona2018probabilisticII}~\cite{carmona2018probabilisticII}}$)$, this representation
appears to offer a different and rather direct viewpoint. It also suggests a
possible route toward numerical methods for {\rm MFGs}: instead of solving
simultaneously for a value function and an equilibrium flow, one may try to
solve the characteristic equation for the gradient field $p$, and then
recover both the equilibrium feedback and the value representation.

\smallskip
This formulation is made possible by the use of the {\rm W-BSDE}. Indeed,
the W-BSDE provides, at the same time, the value of the game and the gradient
field needed to define the optimal response of the representative player.
Thus the pair $(Y_p,Z_p)$ encodes both sides of the {\rm MFG} equilibrium:
the optimality condition through the first component of the gradient, and the
consistency of the equilibrium flow through the enlarged state dynamics.
\end{remark}

\subsection{W-BSDEs and McKean-Vlasov FBSDEs} \label{sec:FBSDEs}

The previous two sections showed that the {\rm W-BSDE} provides a
verification tool for McKean-Vlasov control problems and for mean field
games. In both cases, the central object is the field $Z$, which plays the
role of a gradient of the value function on the Wasserstein space and
determines the optimal feedback.
The purpose of this section is to show that the field $Z$ can be recovered dynamically as
the backward component of a McKean-Vlasov FBSDE. This provides a
characteristic representation of the W-BSDE, complementary to the control
and game interpretations developed above.

\smallskip
Let $(e,m)\in[0,T]\times\Pc(\R^d)$. We consider a forward-backward system
whose conditional law enters both the forward and the backward equations. More
precisely, we look for a triple $(X^{e,m},\Yc^{e,m},\Zc^{e,m})$
such that
\[
    \E\left[
        \sup_{t\in[e,T]}|\Yc^{e,m}_t|^2
        +
        \int_e^T|\Zc^{e,m}_t|^2\,\mathrm{d}t
    \right]<\infty,
\]
and, for $t\in[e,T]$,
\begin{align}
    X^{e,m}_t
    &=
    N^{e,X^{e,m}_e}_t
    +
    \int_e^t
    \nabla_z\Fc
    \bigl(
        s,
        X^{e,m}_s,
        \mu^{e,m}_s,
        \Zc^{e,m}_s
    \bigr)
    \,\mathrm{d}s,
    \label{eq:MKV_FBSDE_forward}
\end{align}
with
\[
    \Lc(X^{e,m}_e\mid\Gc_T)=m,
    \qquad
    \mu^{e,m}_t:=\Lc(X^{e,m}_t\mid\Gc_T),
    \qquad
    \overline\mu^{e,m}_t
    :=
    \Lc(X^{e,m}_t,\Zc^{e,m}_t\mid\Gc_T),
\]
 and $N^{e,x}$ defined by \eqref{eq:defN}.
The backward equation is given by
\begin{equation}\label{eq:MKV_FBSDE_backward}
\begin{aligned}
    \Yc^{e,m}_t
    &=
    \delta_mG(\mu^{e,m}_T)(X^{e,m}_T)
    \\
    &\quad+
    \int_t^T
    \bigg[
        \Fc
        \bigl(
            s,
            X^{e,m}_s,
            \mu^{e,m}_s,
            \Zc^{e,m}_s
        \bigr)
        +
        \int_{\R^d}\int_{\R^d}
        \delta_m\Fc
        \bigl(
            s,
            x,
            \mu^{e,m}_s,
            z
        \bigr)(X^{e,m}_s)
        \,
        \overline\mu^{e,m}_s(\mathrm{d}x,\mathrm{d}z)
    \bigg]
    \,\mathrm{d}s
    \\
    &\quad
    -
    \int_t^T
    \Zc^{e,m}_s\cdot \mathrm{d}X^{e,m}_s .
\end{aligned}
\end{equation}
Equations \eqref{eq:MKV_FBSDE_forward}-\eqref{eq:MKV_FBSDE_backward}
form a McKean-Vlasov forward-backward SDE. The forward dynamics is driven by
the derivative of the generator with respect to the $z$-variable, while the
backward equation contains the linear functional derivative of both the
terminal condition and the generator. The measure
$\overline\mu^{e,m}_t$ keeps track of the joint conditional law of the
forward state and of the backward control variable.

The interpretation is the following. If the W-BSDE solution is sufficiently
regular, then the field
\[
    Z(t,x,m)=\nabla_x\delta_mY(t,m)(x)
\]
should propagate along the characteristic dynamics associated with
$\nabla_z\Fc$. The above McKean-Vlasov FBSDE is precisely the system
describing this propagation. In particular, its backward component
$\Zc^{e,m}$ is expected to coincide with the W-BSDE field evaluated along
the forward process.

\begin{proposition} \label{prop:WBSDE_MKV_FBSDE}
Let $(Y,Z)$ be a {\rm W-BSDE} solution associated with $(\Fc,G)$, and $Z$ satisfying
the characteristic field {\rm \Cref{eq:chac:z}} with the associated regularity {\rm \Cref{eq:regularity_p_assumption}}.
For each $(t,m)\in[0,T]\times\Pc(\R^d)$, let
\[
    (X^{t,m},\Yc^{t,m},\Zc^{t,m})
\]
be the unique solution of the McKean-Vlasov FBSDE
\eqref{eq:MKV_FBSDE_forward}-\eqref{eq:MKV_FBSDE_backward} starting from $(t,m)$. Then, for
$\mathrm{d}t\otimes\mathrm{d}\P$-a.e. $(t,\omega)$,
\[
    \Zc^{t,m}_t
    =
    Z\bigl(t,X^{t,m}_t,m\bigr).
\]
More generally, for $s\in[t,T]$, one has
\[
    \Zc^{t,m}_s
    =
    Z\bigl(s,X^{t,m}_s,\mu^{t,m}_s\bigr),
    \qquad
    \mathrm{d}s\otimes\mathrm{d}\P\text{-a.e.},
\]
where $\mu^{t,m}_s=\Lc(X^{t,m}_s\mid\Gc_T)$.
\end{proposition}

\begin{remark}
$(i)$ The proposition gives a characteristic interpretation of the field $Z$.
The W-BSDE defines $Z$ globally as a function of $(t,x,m)$, while the
McKean-Vlasov FBSDE recovers the same object in a different dynamically way along the
characteristic flow generated by $\nabla_z\Fc$. This is the analogue, in the
Wasserstein setting, of the classical relation between a Markovian BSDE and
the FBSDE associated with the gradient of the value function.

\medskip
$(ii)$ Taking $t=0$ in the previous result gives, in the Hamiltonian case considered in
{\rm \Cref{sec:optimal_control}}, a forward-backward characterization of the optimal
feedback. This FBSDE is different from the adjoint system obtained through the
maximum principle: it is the characteristic system associated with the W-BSDE and
therefore propagates the Wasserstein gradient $Z$ along the optimal
McKean-Vlasov flow. In this formulation, the forward component $X$ is precisely
the optimally controlled state process. In the setting without common noise, a related FBSDE structure is derived in
{\rm \citeauthor*{carmona2025conditionalmckeanvlasovcontrol} \cite{carmona2025conditionalmckeanvlasovcontrol}} by modifying the variational
argument used in the proof of the maximum principle.
\end{remark}

\subsection{W-BSDEs and PDEs on the set of probability measures} \label{sec:PDEs}

We have so far interpreted the {\rm W-BSDE} from three probabilistic
viewpoints: as a verification tool for McKean-Vlasov control, as a
verification tool for mean field games, and as the object associated with a
McKean-Vlasov forward-backward characteristic system. We now turn to its
analytic interpretation.
As in the classical finite-dimensional theory, where Markovian BSDEs are
closely related to semilinear PDEs, the notion of {\rm W-BSDE} has a natural
connection with PDEs on the set of probability measures. In the present
framework, the role of the finite-dimensional state variable is played by
the conditional law of the state process. Therefore, the associated PDE is no
longer written on $[0,T]\times\R^d$, but on
$[0,T]\times\Pc(\R^d)$.
In particular, under additional regularity assumptions, the solution
constructed in {\rm \Cref{thm:existence_WBSDE_gradient}} provides a classical
solution to a PDE on the Wasserstein space. The following statement makes
this connection precise.

\medskip

We recall that the set $\Pc(\R^d)$ is equiped with the weak topology, metrized for instance by
\[
\mathrm \Wc_{\circ}(m,m')
:=
\inf_{\pi\in\Pi(m,m')}
\int_{\R^d}\int_{\R^d}
\bigl(|x-y|\wedge1\bigr)\,\pi(\mathrm dx,\mathrm dy),
\]
where $\Pi(m,m')$ denotes the set of couplings of $m$ and $m'$.
For a smooth map
\[
    \varphi:[0,T]\times\Pc(\R^d)\longrightarrow\R,
\]
we define
\[
\begin{aligned}
\Hc[\varphi](t,m)
&:=
\int_{\R^d}
\Fc\left(t,x,m,\nabla_x\delta_m\varphi(t,m)(x)\right)
\,m(\mathrm dx)
\\
&\quad+
\frac12
\int_{\R^d}
\Tr\left[
\bigl(\Ir_d+\sigma_\circ\sigma_\circ^\top\bigr)
\nabla_x^2\delta_m\varphi(t,m)(x)
\right]m(\mathrm dx)
\\
&\quad+
\frac12
\int_{\R^d}\int_{\R^d}
\Tr\left[
\sigma_\circ\sigma_\circ^\top
\nabla_y\nabla_x\delta_m^2\varphi(t,m)(x,y)
\right]m(\mathrm dx)m(\mathrm dy).
\end{aligned}
\]

Here and below, a test function $\varphi$ is said to be smooth if it is
continuously differentiable in time, admits first and second linear functional
derivatives in the measure variable, and if
\[
\partial_t\varphi,\qquad
\nabla_x\delta_m\varphi,\qquad
\nabla_x^2\delta_m\varphi,\qquad
\nabla_y\nabla_x\delta_m^2\varphi
\]
are continuous and satisfy the boundedness conditions needed to apply the Itô
formula along conditional McKean-Vlasov flows (see, e.g.~\cite[Theorem 4.14]{carmona2018probabilisticII}).
We consider the equation
\begin{equation}
\label{eq:PDE_WBSDE}
\left\{
\begin{aligned}
\partial_t\Uc(t,m)+\Hc[\Uc](t,m)&=0,
&& (t,m)\in[0,T)\times\Pc(\R^d),
\\
\Uc(T,m)&=G(m),
&&m\in\Pc(\R^d).
\end{aligned}
\right.
\end{equation}

A continuous map
\[
    \Uc:[0,T]\times\Pc(\R^d)\longrightarrow\R
\]
is called a viscosity subsolution of
\eqref{eq:PDE_WBSDE} if
\[
    \Uc(T,m)\le G(m),
    \qquad m\in\Pc(\R^d),
\]
and, whenever $\Uc-\varphi$ has a local maximum at
$(t_0,m_0)\in[0,T)\times\Pc(\R^d)$, one has
\[
    \partial_t\varphi(t_0,m_0)+\Hc[\varphi](t_0,m_0)\ge0.
\]
A viscosity supersolution is defined similarly, replacing the terminal inequality
by $\Uc(T,\cdot)\ge G$, local maxima by local minima, and the preceding
inequality by
\[
    \partial_t\varphi(t_0,m_0)+\Hc[\varphi](t_0,m_0)\le0.
\]
A viscosity solution is both a viscosity subsolution and a viscosity supersolution.

\begin{proposition}
\label{prop:viscosity_WBSDE}
Let us work under the assumptions of
{\rm \Cref{thm:existence_WBSDE_gradient}}, and let $(Y,Z)$ be the corresponding
solution of the {\rm W-BSDE}. Then
\[
    \Uc(t,m):=Y(t,m)
\]
is a viscosity solution of \eqref{eq:PDE_WBSDE}.
\end{proposition}

\begin{proof}
We first observe that the terminal condition follows directly from the definition
of the {\rm W-BSDE}. Taking $t=T$ in the backward identity gives
\[
    Y(T,m)=G(m),
    \qquad m\in\Pc(\R^d).
\]
It remains to prove the viscosity inequalities. We only prove the subsolution
property. The supersolution property is obtained by the same argument with the
opposite inequalities.
Let $\varphi$ be a smooth test function such that $Y-\varphi$ has a local
maximum at $(t_0,m_0)\in[0,T)\times\Pc(\R^d)$. Up to adding a constant to
$\varphi$, we may assume that
\[
    Y(t_0,m_0)=\varphi(t_0,m_0),
\]
and, for some $\eta>0$ with $t_0+\eta<T$,
\[
    Y(t,m)-\varphi(t,m)\le0
\]
whenever $t\in[t_0,t_0+\eta]$ and $\mathrm \Wc_\circ(m,m_0)\le \eta$.
We want to prove
\[
    \partial_t\varphi(t_0,m_0)+\Hc[\varphi](t_0,m_0)\ge0.
\]

Assume by contradiction that this is false. By continuity of the derivatives of
$\varphi$, and reducing $\eta>0$ if necessary, there exists $\delta>0$ such
that
\begin{equation}
\label{eq:strict_viscosity_contradiction}
    \partial_t\varphi(t,m)+\Hc[\varphi](t,m)
    \le -\delta
\end{equation}
for all $t\in[t_0,t_0+\eta]$ and $\mathrm \Wc_\circ(m,m_0)\le \eta$.

We now associate with the test function a linear generator. For
$(t,x,m)\in[0,T)\times\R^d\times\Pc(\R^d)$, set
\[
\begin{aligned}
\Lc^\varphi(t,x,m)
&:=
\frac12
\Tr\left[
\bigl(\Ir_d+\sigma_\circ\sigma_\circ^\top\bigr)
\nabla_x^2\delta_m\varphi(t,m)(x)
\right]
\\
&\quad+
\frac12
\int_{\R^d}
\Tr\left[
\sigma_\circ\sigma_\circ^\top
\nabla_y\nabla_x\delta_m^2\varphi(t,m)(x,y)
\right]m(\mathrm dy),
\end{aligned}
\]
and define
\[
    h^\varphi(t,x,m)
    :=
    -\partial_t\varphi(t,m)-\Lc^\varphi(t,x,m).
\]
Then, by the Itô formula on the Wasserstein space, the pair $\bigl(\varphi,\nabla_x\delta_m\varphi\bigr)$ satisfies the W-BSDE identity on $[t_0,t_0+\eta]$, with terminal condition
$\varphi(t_0+\eta,\cdot)$ and generator $h^\varphi$. More precisely, along
any admissible conditional law flow $\mu$, generated by a drift $B$ and a state
process $X$, one has, for $t\le r\le t_0+\eta$,
\[
\begin{aligned}
\varphi(t,\mu_t)
=
\varphi(r,\mu_r)
&+
\E\left[
    \int_t^r h^\varphi(s,X_s,\mu_s)\,\mathrm ds
    \,\middle|\,\Gc_T
\right]
\\
&-
\E\left[
    \int_t^r
    \nabla_x\delta_m\varphi(s,\mu_s)(X_s)\cdot B(s,X_s,\mu_s)
    \,\mathrm ds
    \,\middle|\,\Gc_T
\right]
\\
&-
\int_t^r
\E\left[
    \nabla_x\delta_m\varphi(s,\mu_s)(X_s)
    \,\middle|\,\Gc_T
\right]\cdot\sigma_\circ\,\mathrm dW^\circ_s .
\end{aligned}
\]

We now compare this identity with the W-BSDE identity satisfied by $(Y,Z)$.
Following the proof of the comparison principle, let $c,\varepsilon>0$, and
define the truncated drift
\[
\begin{aligned}
B^{c,\varepsilon}(t,x,m)
:=
\chi_c\bigl(|\Delta(t,x,m)|\bigr)
\mathbf 1_{\{t\le t_0+\eta-\varepsilon\}}
\frac{
\Fc(t,x,m,Z(t,x,m))
-
\Fc(t,x,m,\nabla_x\delta_m\varphi(t,m)(x))
}{
|\Delta(t,x,m)|^2
}
\Delta(t,x,m),
\end{aligned}
\]
where
\[
    \Delta(t,x,m)
    :=
    Z(t,x,m)-\nabla_x\delta_m\varphi(t,m)(x),
\]
and $\chi_c$ is a smooth cut-off function satisfying
\[
    \chi_c(r)=0\quad\text{for }r\le c/2,
    \qquad
    \chi_c(r)=1\quad\text{for }r\ge c.
\]
Since $t_0+\eta<T$, the field $Z=p$ is Lipschitz continuous
in $m$ on $[t_0,t_0+\eta]$ by {\rm \Cref{assum:gradient}}. Together with the
bounds encoded in $\underline{{\rm Lip}}_T(\Fc,G)$ and the smoothness of
$\varphi$, this shows that $B^{c,\varepsilon}$ is bounded and Lipschitz in $m$,
uniformly in $x$, hence defines an admissible drift. Let
$(\mu_s)_{s\in[t_0,t_0+\eta]}$ be the associated conditional law flow starting
from
\[
    \mu_{t_0}=m_0.
\]
We localize this flow inside the neighborhood where the test function touches
$Y$. Define the $\Gc$-stopping time
\[
    \tau
    :=
    \inf\left\{
        s\ge t_0:
        \Wc_\circ(\mu_s,m_0)\ge \eta
    \right\}
    \wedge (t_0+\eta).
\]
Since $\mu_{t_0}=m_0$, the conditional joint law of
$(X_s,X_{t_0})$ given $\Gc_T$ is a coupling of $(\mu_s,m_0)$, and
$|B^{c,\varepsilon}|\le\underline{{\rm Lip}}_T(\Fc,G)=:\Lambda$ uniformly in
$(c,\varepsilon)$. Hence, for $s\in[t_0,t_0+\eta]$,
\[
    \Wc_\circ(\mu_s,m_0)
    \le
    \E\bigl[|X_s-X_{t_0}|\wedge 1\,\big|\,\Gc_T\bigr]
    \le
    \Theta_s
    :=
    \Lambda(s-t_0)+C\sqrt{s-t_0}
    +|\sigma_\circ|\sup_{t_0\le r\le s}\bigl|W^\circ_r-W^\circ_{t_0}\bigr|,
\]
with $C$ depending only on $d$. The process $\Theta$ is continuous,
nondecreasing, $\Gc$-adapted, vanishes at $t_0$, and does not depend on
$(c,\varepsilon)$. Therefore
$\tau\ge\tau^\ast:=\inf\{s\ge t_0:\Theta_s\ge\eta\}\wedge(t_0+\eta)>t_0$
a.s., so that
\[
    \E[\tau-t_0]\ \ge\ \rho:=\E[\tau^\ast-t_0]>0,
\]
where $\rho$ depends only on $(\eta,\Lambda,|\sigma_\circ|)$, and in particular
not on $(c,\varepsilon)$.

Subtracting the W-BSDE identity for
$(\varphi,\nabla_x\delta_m\varphi)$ from the one for $(Y,Z)$, and using the
definition of $B^{c,\varepsilon}$, gives
\[
\begin{aligned}
Y(t_0,m_0)-\varphi(t_0,m_0)
&=
\E\bigg[
Y(\tau,\mu_\tau)-\varphi(\tau,\mu_\tau)
\\
&\quad+
\int_{t_0}^{\tau}
\Big(
\Fc(s,X_s,\mu_s,\nabla_x\delta_m\varphi(s,\mu_s)(X_s))
-
h^\varphi(s,X_s,\mu_s)
\Big)
\,\mathrm ds
\,\bigg|\,\Gc_T
\bigg]
\\
&\quad
-
\int_{t_0}^{\tau}
\E\big[
Z(s,X_s,\mu_s)
-
\nabla_x\delta_m\varphi(s,\mu_s)(X_s)
\,\big|\,\Gc_T
\big]\cdot\sigma_\circ\,\mathrm dW^\circ_s
+
R^{c,\varepsilon}_\tau,
\end{aligned}
\]
where the remainder comes only from the truncation region
$\{|\Delta|\le c\}$ and from the terminal layer
$[t_0+\eta-\varepsilon,t_0+\eta]$. Indeed, by definition of
$B^{c,\varepsilon}$, the integrand of the remainder is
$\bigl(1-\chi_c(|\Delta|)\mathbf 1_{\{s\le t_0+\eta-\varepsilon\}}\bigr)$ times
the difference $\Fc(\cdot,Z)-\Fc(\cdot,\nabla_x\delta_m\varphi)$, which is
bounded by $Lc$ on the truncation region and by $2\|\Fc\|_\infty$ on the
terminal layer. Hence, for a constant $C$
independent of $c,\varepsilon$,
\[
    |R^{c,\varepsilon}_\tau|
    \le C(c+\varepsilon).
\]

Taking expectation removes the common-noise martingale term. Since
$Y-\varphi\le0$ on the localized neighborhood, we have
\[
    Y(\tau,\mu_\tau)-\varphi(\tau,\mu_\tau)\le0.
\]
Moreover, by the definition of $h^\varphi$,
\[
\begin{aligned}
\int_{\R^d}
\Big(
\Fc(s,x,\mu_s,\nabla_x\delta_m\varphi(s,\mu_s)(x))
-
h^\varphi(s,x,\mu_s)
\Big)\mu_s(\mathrm dx)
=
\partial_t\varphi(s,\mu_s)+\Hc[\varphi](s,\mu_s).
\end{aligned}
\]
By \eqref{eq:strict_viscosity_contradiction}, this last quantity is bounded above
by $-\delta$ on $[t_0,\tau]$. Therefore,
\[
\begin{aligned}
0
=
Y(t_0,m_0)-\varphi(t_0,m_0)
&\le
-\delta\,\rho
+
C(c+\varepsilon).
\end{aligned}
\]
Letting $c\downarrow0$ and $\varepsilon\downarrow0$, and
using that $\rho$ does not depend on $(c,\varepsilon)$, we obtain
\[
    0\le -\delta\,\rho<0,
\]
which is impossible. Hence the contradiction assumption was false, and
\[
    \partial_t\varphi(t_0,m_0)+\Hc[\varphi](t_0,m_0)\ge0.
\]
Thus $Y$ is a viscosity subsolution.
The proof of the supersolution property is identical. If $Y-\varphi$ has a local
minimum at $(t_0,m_0)$ and the viscosity inequality fails, then, after reversing
the inequalities in the preceding argument, one obtains
\[
    0\ge \delta\,\rho>0,
\]
again a contradiction. Therefore $Y$ is also a viscosity supersolution, and the
proof is complete.
\end{proof}

\begin{remark} This result shows that, under the general assumptions of {\rm \Cref{thm:existence_WBSDE_gradient}}, the W-BSDE construction provides a viscosity solution of the Wasserstein PDE. Under the additional smoothness assumptions of {\rm \Cref{prop:WBSDE_PDE}} below, the same function becomes a classical solution. Let us stress that our assumptions may look stronger than those usually imposed for viscosity solutions on the Wasserstein space; see, for instance, {\rm\cite{Daudin01022025, meteViscosity2024, Bayraktar03042025}}. This is due to the specific nature of the W-BSDE approach. The construction does not only produce the value function $Y$, but also identifies its intrinsic gradient. Thus the result contains more information than the mere viscosity property of the value function. The additional regularity requirements are precisely the price to pay for this simultaneous construction of the function and its Wasserstein gradient. In the absence of genuine measure dependence, the characteristic equation reduces to the classical Markovian one, and the usual viscosity correspondence is recovered. \end{remark}

\begin{proposition}\label{prop:WBSDE_PDE}
Let $(Y,Z)$ be a solution of the {\rm W--BSDE} associated with $(\Fc,G)$,
with $Y$ continuous on $[0,T]\times\Pc(\R^d)$ and
\[
    Z(t,x,m)=\nabla_x\delta_mY(t,m)(x).
\]
Assume that, for every $t<T$, the map $Y(t,\cdot)$ admits first and second
continuous linear functional derivatives, and that
\[
    \nabla_x\delta_mY(t,m)(x),\qquad
    \nabla_x^2\delta_mY(t,m)(x),\qquad
    \nabla_y\nabla_x\delta_m^2Y(t,m)(x,y)
\]
are jointly continuous and, for every $\overline T<T$, bounded uniformly
in their arguments for $t\in[0,\overline T]$. Assume also that
\[
    \overline{\Fc}(t,m)
    :=\int_{\R^d}\Fc\bigl(t,x,m,Z(t,x,m)\bigr)\,m(\mathrm dx)
\]
is continuous on $[0,T)\times\Pc(\R^d)$ and bounded on
$[0,\overline T]\times\Pc(\R^d)$ for every $\overline T<T$.
Then $t\mapsto Y(t,m)$ is continuously differentiable on $[0,T)$ for
every $m$, with a right derivative at $t=0$, and $\partial_tY$ is jointly
continuous. Moreover, $\Uc:=Y$ is a classical solution of
\eqref{eq:PDE_WBSDE}.
\end{proposition}

\begin{remark}
$(i)$ In view of the identification obtained in
{\rm \Cref{thm:existence_WBSDE_gradient}},
\[
    \nabla_x\delta_m Y(t,m)(x)=Z(t,x,m)=p(t,x,m),
    \qquad m(\mathrm{d}x)\text{-a.e.},
\]
the additional regularity assumption appearing in
{\rm \Cref{prop:WBSDE_PDE}} is essentially a regularity assumption on the
solution $p$ of the characteristic equation {\rm \Cref{eq:chac:z}}. In
other words, the proposition does not require an independent regularity
theory for $Y$. Once the field $p$ is sufficiently regular in the
variables $(x,m)$, the regularity of the Wasserstein derivative of $Y$
follows through the identification above. An example in which such regularity
is available is given in {\rm \cite{fabricemattiaDecoupling}}.

\medskip

$(ii)$ To the best of our knowledge, this provides one of the first
well-posedness results for classical solutions of a PDE of the form
{\rm \Cref{eq:PDE_WBSDE}}. The point is particularly relevant in the presence
of common noise, where the PDE contains not only the usual second-order term
in the state variable, but also a second-order term in the measure variable.
Moreover, the result is obtained without imposing a separability structure
between the state variable $x$ and the measure variable $m$ in the
nonlinearity $\Fc$ (see for instance {\rm \cite{cardaliaguet2015master}} where such a condition is imposed to obtain a solution without common noise in the torus when  $\Fc$ plays the role of a Hamiltonian ).
\end{remark}

\subsection{W-BSDEs as the mean-field limit of classical BSDEs} \label{sec:n_player}

We now explain how the {\rm W-BSDE} can be obtained as the mean-field limit
of classical finite-dimensional BSDEs. More precisely, we derive a
quantitative estimate comparing the solution of an $n$-particle BSDE with
the solution of the limiting {\rm W-BSDE} evaluated along the empirical
measure of the particle system.

Let us work under the assumptions of {\rm \Cref{thm:existence_WBSDE_gradient}},
and let $(Y,Z)$ be the corresponding solution of the {\rm W-BSDE}
associated with $(\Fc,G)$. As in {\rm \Cref{prop:WBSDE_PDE}}, we assume in
addition that the map
\[
    (t,m,x)\longmapsto \nabla_x\delta_mY(t,m)(x)
\]
admits the derivatives
\[
    \nabla_x^2\delta_mY(t,m)(x),
    \qquad
    \nabla_y\nabla_x\delta_m^2Y(t,m)(x,y),
\]
and that these derivatives are continuous and bounded in their variables.
In particular, by the identification obtained in
{\rm \Cref{thm:existence_WBSDE_gradient}}, we have
\[
    Z(t,x,m)=\nabla_x\delta_mY(t,m)(x).
\]

Let $(N^i)_{1\le i\le n}$ be the system of independent particles with
common noise, defined by
\[
    N^i_t
    =
    X^i_0+W^i_t+\sigma_\circ W^\circ_t,
    \qquad i=1,\dots,n,
\]
and let
\[
    \mu^n_t
    :=
    \frac1n\sum_{i=1}^n\delta_{N^i_t}
\]
be the associated empirical measure. We also denote by
\[
    \Hc^n_t
    :=
    \sigma\left(N^1_s,\dots,N^n_s:\,0\le s\le t\right)
\]
the natural filtration generated by the particle system.
For each $n\ge1$, let
\[
    (Y^n,\Zbb^n)
    =
    \bigl(Y^n,(Z^{1,n},\dots,Z^{n,n})\bigr)
\]
be the $\Hc^n$-adapted solution of the classical $n$-dimensional BSDE
\begin{align}
    Y^n_t
    &=
    G(\mu^n_T)
    +
    \frac1n\sum_{i=1}^n
    \int_t^T
    \Fc\bigl(s,N^i_s,\mu^n_s,nZ^{i,n}_s\bigr)\,\mathrm{d}s
    -
    \sum_{i=1}^n
    \int_t^T
    Z^{i,n}_s\cdot
    \bigl(\mathrm{d}W^i_s+\sigma_\circ\,\mathrm{d}W^\circ_s\bigr),
    \qquad t\in[0,T].
    \label{eq:n_particle_BSDE_statement}
\end{align}
We finally introduce the probability measure under which the estimate on the
martingale integrands is stated. For $i=1,\dots,n$ and $s\in[0,T]$, set
\[
    a^i_s
    :=
    \int_0^1
    \nabla_z\Fc\Bigl(
        s,N^i_s,\mu^n_s,
        (1-\lambda)\,Z(s,N^i_s,\mu^n_s)+\lambda\,nZ^{i,n}_s
    \Bigr)\,\mathrm{d}\lambda,
\]
and define $\Q^n$ by
\[
    \frac{\mathrm{d}\Q^n}{\mathrm{d}\P}
    :=
    \mathcal E^n_T,
    \qquad
    \mathrm{d}\mathcal E^n_t
    =
    \mathcal E^n_t\sum_{i=1}^n a^i_t\cdot\mathrm{d}W^i_t,
    \qquad
    \mathcal E^n_0=1.
\]
Since $|a^i_s|\le\underline{{\rm Lip}}_T(\Fc,G)$, Novikov's condition holds and
$\Q^n$ is equivalent to $\P$. Under $\Q^n$, the processes
$W^{i,\Q}_t:=W^i_t-\int_0^ta^i_s\,\mathrm{d}s$, $i=1,\dots,n$, and $W^\circ$ are
independent Brownian motions, and the particles satisfy
\[
    \mathrm{d}N^i_t
    =
    a^i_t\,\mathrm{d}t
    +
    \mathrm{d}W^{i,\Q}_t
    +
    \sigma_\circ\,\mathrm{d}W^\circ_t .
\]
Since $a^i_t$ is an average of $\nabla_z\Fc$ along the segment between
$Z(t,N^i_t,\mu^n_t)$ and $nZ^{i,n}_t$, under $\Q^n$ the particles are driven by
an approximation of the drift of the characteristic flow \eqref{eq:max_dynamic}.

\begin{proposition} \label{prop:WBSDE_mean_field_limit}
Assume the previous conditions. Suppose moreover that the coefficients are
regular enough so that the classical BSDE
{\rm \Cref{eq:n_particle_BSDE_statement}} is well posed for every $n\ge1$,
and that the preceding derivatives of $Y$ are bounded uniformly in their
variables.
Then there exists a constant $C>0$, independent of $n$, such that, for every
$n\ge1$ and every $t\in[0,T]$,
\begin{align} \label{eq:WBSDE_particle_estimate}
    \left|Y^n_t-Y(t,\mu^n_t)\right|^2
    +
    \E^{\Q^n}\left[
        \int_t^T
        \sum_{i=1}^n
        \left|
            Z^{i,n}_s
            -
            \frac1n Z(s,N^i_s,\mu^n_s)
        \right|^2
        \,\mathrm{d}s
        \,\middle|\,
        \Hc^n_t
    \right]
    \le
    \frac{C}{n^2},
    \qquad \text{a.s.}
\end{align}
The bound on the first term holds pathwise, hence under any equivalent
probability measure. In terms of the rescaled martingale integrands,
\[
    \E^{\Q^n}\left[
        \int_0^T
        \frac1n\sum_{i=1}^n
        \left|
            nZ^{i,n}_s-Z(s,N^i_s,\mu^n_s)
        \right|^2
        \mathrm{d}s
    \right]
    \le
    \frac{C}{n}.
\]
\end{proposition}

\begin{remark}
    $(i)$ The estimate shows that the classical $n$-particle {\rm BSDE} is not only
formally related to the {\rm W-BSDE}; it converges to it with an explicit
rate. The first component $Y^n_t$ is close to the limiting value
$Y(t,\mu^n_t)$, while the rescaled martingale integrand $nZ^{i,n}_t$
is close to the Wasserstein gradient field
$Z(t,N^i_t,\mu^n_t)$. This is the exact analogue of the classical relation
between Markovian BSDEs and gradients of value functions, but now at the
level of the empirical measure. In this sense, the {\rm W-BSDE} can be seen
as the intrinsic mean-field limit of the family of finite-dimensional
{\rm BSDEs} \eqref{eq:n_particle_BSDE_statement}.

\medskip
$(ii)$ It is worth emphasizing that the previous result does not require
$\Fc$ to have the particular structure of a Hamiltonian. In particular, we
do not impose convexity, concavity, or any separability condition between the
state variable $x$ and the measure variable $m$. The argument only relies
on the existence of a sufficiently smooth solution $p$ to the characteristic
{\rm \Cref{eq:chac:z}}.

\smallskip
This is an important point. In many approaches to PDEs on the Wasserstein
space, especially those motivated by mean field control, the nonlinearity is
often assumed to come from an optimization problem and therefore has a
Hamiltonian form. Such a structure is useful, but it also imposes restrictions
on the class of equations that can be treated. Here, by contrast, the central
object is the equation for the field $p$, which is a straight fixed
point problem. Thus, once this fixed point problem admits a smooth solution,
the construction of the {\rm W-BSDE}, the identification
\[
    Z(t,x,m)=p(t,x,m)=\nabla_x\delta_mY(t,m)(x),
\]
and the convergence from the $n$-particle BSDE follow without relying on an
underlying control interpretation.

\smallskip
An example of conditions on $(T,\Fc,G)$ ensuring the existence of such a
smooth solution $p$ is given in {\rm \cite{fabricemattiaDecoupling}}. This makes the result rather flexible:
it applies not only to Hamiltonian nonlinearities arising in mean field
control, but also to more general generators $\Fc$. In this sense, the
framework developed here appears to be more general than several existing
results in the literature, where convexity, separability, or control-based
structures are often used in an essential way; see for instance {\rm \cite{CardaliaguetAnAlgebraic2023, DAUDIN2024110660,cardaliaguet2023sharpconvergenceratesmean}}.
\end{remark}

\section{Proofs of the main results} \label{sec:proofs}

\subsection{The proof of \texorpdfstring{\rm \Cref{thm:existence_WBSDE_gradient}}{Theorem} with smooth coefficients} \label{sec:existence_smooth}

We begin with the proof in the smooth case. The purpose of this step is to
exhibit clearly how a regular solution $p$ of the characteristic equation
\eqref{eq:chac:z} generates a solution of the {\rm W-BSDE}. The construction
proceeds by recovering a functional $Y$ whose Wasserstein gradient is $p$,
namely
\[
    \nabla_x\delta_mY(t,m)(x)=p(t,x,m).
\]
The symmetry condition imposed on $\delta_m p$ in \Cref{eq:p_symmetry} is precisely the
condition which allows one to reconstruct such a potential $Y$.

Once $Y$ is constructed, the rest of the proof consists in applying the Itô
formula on the Wasserstein space along an arbitrary admissible conditional law
flow. This gives the dynamics of $Y(t,\mu_t)$, and the characteristic
equation for $p$ ensures that the drift and martingale terms match exactly
the defining identity of the {\rm W-BSDE}. Thus the smooth case contains the
main idea of the proof; the nonsmooth case will be treated afterwards by
approximation.

\medskip
A map $\Vc:[0,T)\times\R^d\times\Pc(\R^d)\longrightarrow \R$
will be called smooth if it admits bounded continuous derivatives
\[
    \nabla_x\Vc(t,x,m),\qquad
    \delta_m\Vc(t,x,m)(y),\qquad
    \delta_m^2\Vc(t,x,m)(y)(z),
\]
and if, for each $(t,x,m)$, the maps
\[
    y\longmapsto \delta_m\Vc(t,x,m)(y),
    \qquad
    (y,z)\longmapsto \delta_m^2\Vc(t,x,m)(y)(z)
\]
are bounded and twice continuously differentiable, with bounded continuous
derivatives.
Let us assume, that we have obtained a sufficiently smooth map
\[
    p=(p^1,\dots,p^d):[0,T)\times \R^d \times \Pc(\R^d)\longrightarrow \R^d
\]
satisfying \eqref{eq:chac:z}. We also assume that $p$ satisfies the symmetry condition (the proof will be done in the next section)
\begin{align} \label{eq:p_symmetry}
    \nabla_{y_j}\delta_m p^i(t,x,m)(y)
    =
    \nabla_{y_i}\delta_m p^j(t,y,m)(x),
    \qquad
    1\le i,j\le d,
\end{align}
for all $(t,m,x,y)\in[0,T)\times\R^d\times\Pc(\R^d)\times\R^d$, where we recall that
$\nabla_y$ denotes the derivative in the second spatial argument, taken before
the arguments are evaluated; on the right-hand side, that argument is $x$. Equivalently, in matrix notation,
\[
    D_y\delta_m p(t,x,m)(y)
    =
    \bigl(D_y\delta_m p(t,y,m)(x)\bigr)^\top,\quad (t,m,x,y)\in[0,T)\times\Pc(\R^d)\times(\R^d)^2.
\]
We explain how such a map gives rise to a solution of the {\rm W-BSDE}.

\medskip

For $(t,x,m)\in[0,T]\times\R^d\times\Pc(\R^d)$, define
\begin{equation}\label{eq:def_q_d}
\begin{aligned}
    q(t,x,m)
    &:=
    \E \left[
        \delta_mG(\ell^{p,t,m}_T)(N^{t,x}_T)
    \right]
    \\
    &\quad+
    \E \left[
        \int_t^T
        \Fc\left(
            s,N^{t,x}_s,\ell^{p,t,m}_s,
            p(s,N^{t,x}_s,\ell^{p,t,m}_s)
        \right)\,\mathrm{d}s
    \right]
    \\
    &\quad+
    \E \left[
        \int_t^T
        \delta_m \Fc\left(
            s,S^{t,m}_s,\ell^{p,t,m}_s,
            p(s,S^{t,m}_s,\ell^{p,t,m}_s)
        \right)(\widehat N^{t,x}_s)
        \,\mathrm{d}s
    \right].
\end{aligned}
\end{equation}
Here $U$ is independent of the Brownian noises with $\Lc(U)=m$, while $\widehat N$ denotes an independent copy of $N$, conditionally on the common noise. By differentiating under the expectation sign and using the characterization \eqref{eq:chac:z}, we obtain
\[
    \nabla_x q(t,x,m)=p(t,x,m).
\]

\medskip

Let now $\mu\in\Lc$ be associated with an initial time $e$ and an admissible drift $B:[0,T]\times\R^d\times\Pc(\R^d)\to\R^d,$
and let $X$ be a state process satisfying
\[
    \mathrm{d}X_t
    =
    B(t,X_t,\mu_t)\,\mathrm{d}t
    +
    \mathrm{d}W_t
    +
    \sigma_\circ\,\mathrm{d}W^\circ_t,
    \qquad
    \mu_t=\Lc(X_t\mid\Gc_t)=\Lc(X_t\mid\Gc_T).
\]
Here $W$ is $\R^d$-valued, $W^\circ$ is one-dimensional, and $\sigma_\circ\in\R^d$. By applying ~\cite[Theorem 4.1.]{mattiaAMcKean2026} (Backward flow formula) through the semigroup to the map $(t,x,m)\longmapsto q(t,x,m)$, we obtain, for $t \in [e,T]$,
\begin{equation} \label{eq:q_dynamics_before_F_d}
\begin{aligned}
    \mathrm{d}q(t,X_t,\mu_t)
    =
    &-
    \left[
        \Fc(t,X_t,\mu_t,p(t,X_t,\mu_t))
        +
        \int_{\R^d}
        \delta_m \Fc(t,y,\mu_t,p(t,y,\mu_t))(X_t)
        \,\mu_t(\mathrm{d}y)
    \right]\mathrm{d}t
    \\
    &+
    \int_{\R^d}
    D_y\delta_m q(t,X_t,\mu_t)(y)
    \cdot
    \left[
        B(t,y,\mu_t)
        -
        \nabla_z\Fc(t,y,\mu_t,p(t,y,\mu_t))
    \right]
    \,\mu_t(\mathrm{d}y)\,\mathrm{d}t
    \\
    &+
    \int_{\R^d}
    D_y\delta_m q(t,X_t,\mu_t)(y)
    \,\mu_t(\mathrm{d}y)
    \cdot
    \sigma_\circ\,\mathrm{d}W^\circ_t
    \\
    &+
    \nabla_x q(t,X_t,\mu_t)\cdot
    \bigl(
        B(t,X_t,\mu_t)\,\mathrm{d}t
        +
        \mathrm{d}W_t
        +
        \sigma_\circ\,\mathrm{d}W^\circ_t
    \bigr).
\end{aligned}
\end{equation}
Here $D_y\delta_m q(t,x,m)(y)\in\R^d$, and $\nabla_z\Fc$ denotes the gradient of $\Fc$ with respect to its last variable.

\medskip

We introduce the averaged generator
\[
    F(t,m)
    :=
    \int_{\R^d}
    \Fc(t,x,m,p(t,x,m))\,m(\mathrm{d}x).
\]
By the chain rule for the linear functional derivative,
\begin{equation}\label{eq:dm_F_first_d}
\begin{aligned}
    \delta_m F(t,m)(x)
    =
    &\Fc(t,x,m,p(t,x,m))
    +
    \int_{\R^d}
    \delta_m \Fc(t,y,m,p(t,y,m))(x)\,m(\mathrm{d}y)
    \\
    &+
    \int_{\R^d}
    \nabla_z\Fc(t,y,m,p(t,y,m))
    \cdot
    \delta_m p(t,y,m)(x)\,m(\mathrm{d}y).
\end{aligned}
\end{equation}
Since $\nabla_x q=p$, we have, for every $y,x\in\R^d$,
\[
    \delta_m p(t,y,m)(x)
    =
    \delta_m\nabla_x q(t,y,m)(x).
\]
Using the symmetry condition \Cref{eq:p_symmetry}, this can be rewritten as
\[
    \delta_m p(t,y,m)(x)
    =
    D_y\delta_m q(t,x,m)(y).
\]
Indeed, by \Cref{eq:p_symmetry} the two sides have the same derivative with
respect to $x$, and they therefore coincide for the normalisation of
$\delta_m p(t,y,m)(\cdot)$, which is defined up to an additive constant in its
evaluation variable.
Therefore \eqref{eq:dm_F_first_d} becomes
\begin{align}
    \delta_m F(t,m)(x)
    =
    &\Fc(t,x,m,p(t,x,m))
    +
    \int_{\R^d}
    \delta_m \Fc(t,y,m,p(t,y,m))(x)\,m(\mathrm{d}y)
    \nonumber
    \\
    &+
    \int_{\R^d}
    D_y\delta_m q(t,x,m)(y)
    \cdot
    \nabla_z\Fc(t,y,m,p(t,y,m))\,m(\mathrm{d}y).
    \label{eq:dm_F_sym_d}
\end{align}
Substituting \eqref{eq:dm_F_sym_d} into \eqref{eq:q_dynamics_before_F_d}, we obtain
\begin{equation}\label{eq:q_linear_equation_d}
\begin{aligned}
    \mathrm{d}q(t,X_t,\mu_t)
    =
    &-
    \delta_m F(t,\mu_t)(X_t)\,\mathrm{d}t
    \\
    &+
    \int_{\R^d}
    D_y\delta_m q(t,X_t,\mu_t)(y)
    \cdot
    B(t,y,\mu_t)\,\mu_t(\mathrm{d}y)\,\mathrm{d}t
    \\
    &+
    \int_{\R^d}
    D_y\delta_m q(t,X_t,\mu_t)(y)
    \,\mu_t(\mathrm{d}y)
    \cdot
    \sigma_\circ\,\mathrm{d}W^\circ_t
    \\
    &+
    \nabla_x q(t,X_t,\mu_t)\cdot
    \bigl(
        B(t,X_t,\mu_t)\,\mathrm{d}t
        +
        \mathrm{d}W_t
        +
        \sigma_\circ\,\mathrm{d}W^\circ_t
    \bigr).
\end{aligned}
\end{equation}

\medskip
Let $m \in \Pc(\R^d)$.
We now freeze the flow $(\mu_t)_{t \in [0,T]}$ at $m$ between $[0,e]$ and choose $B=0$.
Taking conditional expectation with respect to $\Gc_e$ in \eqref{eq:q_linear_equation_d}, we obtain, for $m(\mathrm{d}x)$ a.e.  $x$,
\begin{align}
    q(e,x,m)
    =
    \E\left[
        \delta_mG(M^{e,m}_T)(N^{e,x}_T)
        +
        \int_{e}^T
        \delta_m F(s,M^{e,m}_s)(N^{e,x}_s)
        \,\mathrm{d}s
        \,\middle|\,\Gc_{e}
    \right],
    \label{eq:q_frozen_representation_d}
\end{align}
where, for $m\in\Pc(\R^d)$,
\begin{equation}\label{eq:def_M_flow}
    M^{t,m}_s
    :=
    \Lc(N^{t,U}_s\mid\Gc_T)
    =
    \Lc(N^{t,U}_s\mid\Gc_s),
    \qquad s\in[t,T],
\end{equation}
with $U$ independent of the Brownian noises and $\Lc(U)=m$. This is true for any $e \in [0,T]$ and $m \in \Pc(\R^d)$.

\medskip

We next define the candidate value functional. For $(t,m)\in[0,T]\times\Pc(\R^d)$, set
\begin{align}
    Y(t,m)
    :=
    \E\left[
        G(M^{t,m}_T)
        +
        \int_t^T F(s,M^{t,m}_s)\,\mathrm{d}s
    \right].
    \label{eq:def_Y_smooth_d}
\end{align}
Equivalently, along any admissible flow $\mu\in\Lc$,
\begin{align}
    Y(t,\mu_t)
    =
    \E\left[
        G(M^{t,\mu_t}_T)
        +
        \int_t^T
        F(s,M^{t,\mu_t}_s)
        \,\mathrm{d}s
        \,\middle|\,\Gc_t
    \right].
    \label{eq:Y_along_flow_d}
\end{align}

\medskip

Differentiating \eqref{eq:def_Y_smooth_d} with respect to the measure argument gives
\begin{align}
    \delta_m Y(t,m)(x)
    =
    &\E\left[
        \delta_mG(M^{t,m}_T)(N^{t,x}_T) +
        \int_t^T
        \delta_m F(s,M^{t,m}_s)(N^{t,x}_s)
        \,\mathrm{d}s
    \right].
    \label{eq:dm_Y_representation_d}
\end{align}
Comparing with \eqref{eq:q_frozen_representation_d}, we deduce that
\[
    \delta_mY(t,m)(x)
    =
    q(t,x,m),
    \qquad
    \text{for}\,m(\mathrm{d}x) \text{-a.e. }x.
\]
Consequently, since $\nabla_x q=p$,
\[
    \nabla_x\delta_mY(t,m)(x)
    =
    p(t,x,m),
    \qquad
    \text{for}\,m(\mathrm{d}x) \text{-a.e. }x.
\]

\medskip

It remains to verify that $(Y,p)$ satisfies the {\rm W-BSDE}. Applying the Itô formula through the semigroup to $(t,m)\mapsto Y(t,m)$ along the admissible flow $\mu$ (see \cite[Theorem 4.1.]{mattiaAMcKean2026}), and using the previous identification, gives
\begin{align}
    \mathrm{d}Y(t,\mu_t)
    =
    &-
    F(t,\mu_t)\,\mathrm{d}t +
    \int_{\R^d}
    \nabla_x\delta_mY(t,\mu_t)(x)
    \cdot
    B(t,x,\mu_t)\,\mu_t(\mathrm{d}x)\,\mathrm{d}t
    \nonumber
    \\
    &+
    \int_{\R^d}
    \nabla_x\delta_mY(t,\mu_t)(x)
    \,\mu_t(\mathrm{d}x)
    \cdot
    \sigma_\circ\,\mathrm{d}W^\circ_t
    \nonumber
    \\
    =
    &-
    \int_{\R^d}
    \Fc(t,x,\mu_t,p(t,x,\mu_t))\,\mu_t(\mathrm{d}x)\,\mathrm{d}t
    +
    \int_{\R^d}
    p(t,x,\mu_t)\cdot B(t,x,\mu_t)\,\mu_t(\mathrm{d}x)\,\mathrm{d}t
    \nonumber
    \\
    &+
    \int_{\R^d}
    p(t,x,\mu_t)\,\mu_t(\mathrm{d}x)
    \cdot
    \sigma_\circ\,\mathrm{d}W^\circ_t.
    \label{eq:Y_dynamics_smooth_d}
\end{align}
Integrating from $t$ to $T$ and using $Y(T,\mu_T)=G(\mu_T)$, we get
\begin{equation}\label{eq:Y_integrated_smooth_d}
\begin{aligned}
    Y(t,\mu_t)
    =
    G(\mu_T)
    &+
    \int_t^T
    \int_{\R^d}
    \Fc(r,x,\mu_r,p(r,x,\mu_r))\,\mu_r(\mathrm{d}x)
    \,\mathrm{d}r
    \\
    &-
    \int_t^T
    \int_{\R^d}
    p(r,x,\mu_r)\cdot B(r,x,\mu_r)\,\mu_r(\mathrm{d}x)
    \,\mathrm{d}r
    \\
    &-
    \int_t^T
    \left(
        \int_{\R^d}
        p(r,x,\mu_r)\,\mu_r(\mathrm{d}x)
    \right)
    \cdot
    \sigma_\circ\,\mathrm{d}W^\circ_r.
\end{aligned}
\end{equation}
Since $\mu_r=\Lc(X_r\mid\Gc_T)$, the first two integral terms can be rewritten as
\[
    \int_{\R^d}
    \Fc(r,x,\mu_r,p(r,x,\mu_r))\,\mu_r(\mathrm{d}x)
    =
    \E\left[
        \Fc(r,X_r,\mu_r,p(r,X_r,\mu_r))
        \,\middle|\,\Gc_T
    \right],
\]
and
\[
    \int_{\R^d}
    p(r,x,\mu_r)\cdot B(r,x,\mu_r)\,\mu_r(\mathrm{d}x)
    =
    \E\left[
        p(r,X_r,\mu_r)\cdot B(r,X_r,\mu_r)
        \,\middle|\,\Gc_T
    \right].
\]
Moreover,
\[
    \int_{\R^d}
    p(r,x,\mu_r)\,\mu_r(\mathrm{d}x)
    =
    \E\left[
        p(r,X_r,\mu_r)
        \,\middle|\,\Gc_T
    \right].
\]
Thus \eqref{eq:Y_integrated_smooth_d} is precisely the defining identity of the {\rm W-BSDE} associated with $(\Fc,G)$, with $Z(t,x,\mu_t)=p(t,x,\mu_t)$.
The integrability conditions follow from the smooth regime we are in.
Therefore, in the smooth setting, $(Y,p)$ is a {\rm W-BSDE} solution associated with $(\Fc,G)$.

\subsection{Proof of the symmetry property \texorpdfstring{\rm \Cref{eq:p_symmetry}}{equation}}

We now justify the symmetry property used in the previous step, namely {\rm \Cref{eq:p_symmetry}}. Recall that,
in order to reconstruct a potential $Y$ from the characteristic field $p$,
it is not enough to know that $p$ solves \eqref{eq:chac:z}. One also needs to ensure that $p$
can indeed be written as the gradient of a linear functional derivative,
namely
\[
    p(t,x,m)=\nabla_x\delta_mY(t,m)(x).
\]
The role of the symmetry condition \eqref{eq:p_symmetry} is precisely to
provide this compatibility condition.

\smallskip
The proof below explains why this symmetry is not an additional artificial
constraint in the smooth regime. It follows from the structure of the
characteristic equation itself. To keep the argument readable, we present it
in the simplified case where
\[
    \Fc(t,x,m,z)=f(z),
\]
with $f$ smooth, and in dimension $d=1$. This case already contains the
main idea. The general case is obtained by the same method, with heavier
notation coming from the dependence of $\Fc$ on $(t,x,m)$ and from the
vectorial nature of $p$.
The strategy is the following. Starting from the auxiliary function $q$
constructed in the previous subsection, we study the mixed derivative
\[
    \nabla_y\delta_m q(t,x,m)(y)
\]
and compare it with the same quantity after exchanging the two spatial
variables. The difference between these two terms defines a symmetry defect.
By deriving a representation formula for this defect and using a backward
Gronwall argument, we prove that the defect must vanish identically. Since
$\partial_xq=p$, this gives the desired symmetry property for $p$.

\medskip
Let $p$ be a smooth solution of \eqref{eq:chac:z}, and let $q$ be defined as in \eqref{eq:def_q_d}. We use throughout the notation introduced in the previous steps.

\medskip

Let $\mu\in\Lc$ be an admissible flow associated with an initial time $e \in [0,T]$ and an admissible drift $B$. Let $X$ be a process satisfying
\[
    \mathrm{d}X_t
    =
    \mathrm{d}W_t
    +
    \sigma_\circ\,\mathrm{d}W^\circ_t.
\]
By the Itô formula through the semigroup (  \cite[Theorem 4.1.]{mattiaAMcKean2026}) applied to the map $(t,x,m)\mapsto q(t,x,m)$, and using the identity $\partial_x q=p$, we obtain
\begin{equation}\label{eq:dynamics_q_symmetry}
\begin{aligned}
    \mathrm{d}q(t,X_t,\mu_t)
    =
    &-
    f\bigl(p(t,X_t,\mu_t)\bigr)\,\mathrm{d}t
    \\
    &+
    \int_\R
    \partial_y\delta_m q(t,X_t,\mu_t)(\xi)
    \Bigl(
        B(t,\xi,\mu_t)
        -
        f'\bigl(p(t,\xi,\mu_t)\bigr)
    \Bigr)
    \mu_t(\mathrm{d}\xi)\,\mathrm{d}t
    \\
    &+
    \int_\R
    \partial_y\delta_m q(t,X_t,\mu_t)(\xi)
    \mu_t(\mathrm{d}\xi)\,
    \sigma_\circ\,\mathrm{d}W^\circ_t
    \\
    &+
    \partial_xq(t,X_t,\mu_t)
    \bigl(
        \mathrm{d}W_t+\sigma_\circ\,\mathrm{d}W^\circ_t
    \bigr).
\end{aligned}
\end{equation}

We introduce the auxiliary functional
\begin{align}
    \Jc(t,x,m)
    :=
    f\bigl(p(t,x,m)\bigr)
    +
    \int_\R
    f'\bigl(p(t,\xi,m)\bigr)
    \partial_\xi\delta_m q(t,x,m)(\xi)\,
    m(\mathrm{d}\xi).
    \label{eq:def_J_symmetry}
\end{align}
With this notation, \eqref{eq:dynamics_q_symmetry} may be rewritten in the compact form
\begin{align}
    \mathrm{d}q(t,X_t,\mu_t)
    =
    &-
    \Jc(t,X_t,\mu_t)\,\mathrm{d}t
    +
    \int_\R
    \partial_y\delta_m q(t,X_t,\mu_t)(\xi)
    B(t,\xi,\mu_t)\,\mu_t(\mathrm{d}\xi)\,\mathrm{d}t
    \nonumber
    \\
    &+
    \int_\R
    \partial_y\delta_m q(t,X_t,\mu_t)(\xi)
    \mu_t(\mathrm{d}\xi)\,
    \sigma_\circ\,\mathrm{d}W^\circ_t
    +
    \partial_x q(t,X_t,\mu_t)
    \bigl(
        \mathrm{d}W_t+\sigma_\circ\,\mathrm{d}W^\circ_t
    \bigr).
    \label{eq:dynamics_q_J}
\end{align}

As in the previous step, we now take $m \in \Pc(\R^d)$ freeze the flow at $m$ between $[0,e]$ with $B=0$.
Taking the conditional expectation with respect to $\Gc_e$, the martingale terms vanish and we obtain, for every $m(\mathrm{d}x)$-a.e.,
\begin{align}
    q(e,x,m)
    =
    \E\left[
        \delta_m G\bigl(M^{e,m}_T\bigr)
        \bigl(N^{e,x}_T\bigr)
        +
        \int_e^T
        \Jc\bigl(s,N^{e,x}_s,M^{e,m}_s\bigr)
        \,\mathrm{d}s
        \,\middle|\,\Gc_e
    \right].
    \label{eq:q_representation_symmetry}
\end{align}
This holds for any $(e,m)$, where $M^{e,m}$ is the null-drift flow introduced in \Cref{eq:def_M_flow}.

Let $\widehat N$ be an independent copy of $N$, conditionally on the common noise. Differentiating \eqref{eq:q_representation_symmetry} with respect to the measure argument gives
\begin{align}
    \delta_m q(t,x,m)(y)
    =
    \E\left[
        \delta_m^2G\bigl(M^{t,m}_T\bigr)
        \bigl(N^{t,x}_T\bigr)
        \bigl(\widehat N^{t,y}_T\bigr)
    +
        \int_t^T
        \delta_m\Jc
        \bigl(s,N^{t,x}_s,M^{t,m}_s\bigr)
        \bigl(\widehat N^{t,y}_s\bigr)
        \,\mathrm{d}s
    \right].
    \label{eq:dm_hat_q_symmetry}
\end{align}

We shall also use the derivative of $\Jc$. From \eqref{eq:def_J_symmetry}, a direct differentiation with respect to the measure variable yields
\begin{equation}\label{eq:dm_J_symmetry}
\begin{aligned}
    \delta_m\Jc(t,x,m)(y)
    =
    &f'\bigl(p(t,x,m)\bigr)\,
    \delta_m p(t,x,m)(y)
    \\
    &+
    f'\bigl(p(t,y,m)\bigr)\,
    \partial_y\delta_m q(t,x,m)(y)
    \\
    &+
    \int_\R
    f''\bigl(p(t,\xi,m)\bigr)\,
    \delta_m p(t,\xi,m)(y)\,
    \partial_\xi\delta_m q(t,x,m)(\xi)\,
    m(\mathrm{d}\xi)
    \\
    &+
    \int_\R
    f'\bigl(p(t,\xi,m)\bigr)\,
    \partial_\xi\delta_m^2 q(t,x,m)(\xi)(y)\,
    m(\mathrm{d}\xi).
\end{aligned}
\end{equation}
We now apply again the Itô formula through the semigroup to the map
\[
    (t,x,m,y)\longmapsto \delta_m q(t,x,m)(y)
\]
along two conditionally independent copies $X$ and $\widehat X$, both driven by the same common noise. This gives
\begin{equation}\label{eq:ito_dm_q_symmetry}
\begin{aligned}
    \mathrm{d}\delta_m q(t,X_t,\mu_t)(\widehat X_t)
    =
    &-
    \delta_m\Jc(t,X_t,\mu_t)(\widehat X_t)\,\mathrm{d}t
    \\
    &+
    \int_\R
    \partial_z\delta_m^2q(t,X_t,\mu_t)(\widehat X_t)(z)
    B(t,z,\mu_t)\,\mu_t(\mathrm{d}z)\,\mathrm{d}t
    \\
    &+
    \int_\R
    \partial_z\delta_m^2q(t,X_t,\mu_t)(\widehat X_t)(z)
    \mu_t(\mathrm{d}z)\,
    \sigma_\circ\,\mathrm{d}W^\circ_t
    \\
    &+
    \partial_x\delta_m q(t,X_t,\mu_t)(\widehat X_t)
    \bigl(
        \mathrm{d}W_t+\sigma_\circ\,\mathrm{d}W^\circ_t
    \bigr)
    \\
    &+
    \partial_y\delta_m q(t,X_t,\mu_t)(\widehat X_t)
    \bigl(
        \mathrm{d}\widehat W_t+\sigma_\circ\,\mathrm{d}W^\circ_t
    \bigr).
\end{aligned}
\end{equation}

At this stage, we choose the drift in the preceding identity to be
\[
    B(s,z,m)=f'\bigl(p(s,z,m)\bigr),
    \qquad s\ge e.
\]
Using the symmetry of the second-order linear functional derivative,
\[
    \delta_m^2 q(t,x,m)(y)(z)
    =
    \delta_m^2 q(t,x,m)(z)(y),
\]
the last term in \eqref{eq:dm_J_symmetry} cancels with the drift term in
\eqref{eq:ito_dm_q_symmetry}. Together with $\partial_x q=p$, the representation
\eqref{eq:dm_hat_q_symmetry} then yields, for arbitrary $m\in\Pc(\R)$,
\begin{equation}\label{eq:dm_q_representation_symmetry}
\begin{aligned}
    \delta_m q(e,x,m)(y)
    =
    &\E\bigg[
        \delta_m^2G\bigl(\ell^{e,m}_T\bigr)
        \bigl(N^{e,x}_T\bigr)
        \bigl(\widehat N^{e,y}_T\bigr)
        \\
        &\quad+
        \int_e^T
        f'\bigl(p(s,N^{e,x}_s,\ell^{e,m}_s)\bigr)
        \delta_m p
        \bigl(s,N^{e,x}_s,\ell^{e,m}_s\bigr)
        \bigl(\widehat N^{e,y}_s\bigr)
        \,\mathrm{d}s
        \\
        &\quad+
        \int_e^T
        f'\bigl(p(s,\widehat N^{e,y}_s,\ell^{e,m}_s)\bigr)
        \partial_y\delta_m q
        \bigl(s,N^{e,x}_s,\ell^{e,m}_s\bigr)
        \bigl(\widehat N^{e,y}_s\bigr)
        \,\mathrm{d}s
        \\
        &\quad+
        \int_e^T
        f''\bigl(p(s,\widetilde S^{e,m}_s,\ell^{e,m}_s)\bigr)
        \delta_m p
        \bigl(s,\widetilde S^{e,m}_s,\ell^{e,m}_s\bigr)
        \bigl(\widehat N^{e,y}_s\bigr)
        \partial_y\delta_m q
        \bigl(s,N^{e,x}_s,\ell^{e,m}_s\bigr)
        \bigl(\widetilde S^{e,m}_s\bigr)
        \,\mathrm{d}s
    \bigg].
\end{aligned}
\end{equation}
Here $\widetilde S^{t,m}$ denotes an independent copy with conditional law
$\ell_s^{t,m}$ given the common noise. This is true for any $(e,m)$.
Finally, differentiating \eqref{eq:dm_q_representation_symmetry} with respect to the first spatial variable $x$, and using the integration by parts formula, gives
\begin{equation}\label{eq:px_dm_q_symmetry}
\begin{aligned}
    \partial_x\delta_m q(t,x,m)(y)
    =
    &\E\bigg[
        \delta_m^2G\bigl(\ell^{t,m}_T\bigr)
        \bigl(N^{t,x}_T\bigr)
        \bigl(\widehat N^{t,y}_T\bigr)
        \frac{W_T-W_t}{T-t}
        \\
        &\quad+
        \int_t^T
        f'\bigl(p(s,N^{t,x}_s,\ell^{t,m}_s)\bigr)
        \delta_m p
        \bigl(s,N^{t,x}_s,\ell^{t,m}_s\bigr)
        \bigl(\widehat N^{t,y}_s\bigr)
        \frac{W_s-W_t}{s-t}\,\mathrm{d}s
        \\
        &\quad+
        \int_t^T
        f'\bigl(p(s,\widehat N^{t,y}_s,\ell^{t,m}_s)\bigr)
        \partial_y\delta_m q
        \bigl(s,N^{t,x}_s,\ell^{t,m}_s\bigr)
        \bigl(\widehat N^{t,y}_s\bigr)
        \frac{W_s-W_t}{s-t}\,\mathrm{d}s
        \\
        &\quad+
        \int_t^T
        f''\bigl(p(s,\widetilde S^{t,m}_s,\ell^{t,m}_s)\bigr)
        \delta_m p
        \bigl(s,\widetilde S^{t,m}_s,\ell^{t,m}_s\bigr)
        \bigl(\widehat N^{t,y}_s\bigr)
        \\
        &\hspace{4.8cm}
        \times
        \partial_y\delta_m q
        \bigl(s,N^{t,x}_s,\ell^{t,m}_s\bigr)
        \bigl(\widetilde S^{t,m}_s\bigr)
        \frac{W_s-W_t}{s-t}\,\mathrm{d}s
    \bigg].
\end{aligned}
\end{equation}
Similarly, differentiating \eqref{eq:dm_q_representation_symmetry} with respect to the second spatial variable $y$ gives
\begin{equation}    \label{eq:py_dm_q_symmetry}
\begin{aligned}
    \partial_y\delta_m q(t,x,m)(y)
    =
    &\E\bigg[
        \delta_m^2G\bigl(\ell^{t,m}_T\bigr)
        \bigl(N^{t,x}_T\bigr)
        \bigl(\widehat N^{t,y}_T\bigr)
        \frac{\widehat W_T-\widehat W_t}{T-t}
        \\
        &\quad+
        \int_t^T
        f'\bigl(p(s,N^{t,x}_s,\ell^{t,m}_s)\bigr)
        \delta_m p
        \bigl(s,N^{t,x}_s,\ell^{t,m}_s\bigr)
        \bigl(\widehat N^{t,y}_s\bigr)
        \frac{\widehat W_s-\widehat W_t}{s-t}\,\mathrm{d}s
        \\
        &\quad+
        \int_t^T
        f'\bigl(p(s,\widehat N^{t,y}_s,\ell^{t,m}_s)\bigr)
        \partial_y\delta_m q
        \bigl(s,N^{t,x}_s,\ell^{t,m}_s\bigr)
        \bigl(\widehat N^{t,y}_s\bigr)
        \frac{\widehat W_s-\widehat W_t}{s-t}\,\mathrm{d}s
        \\
        &\quad+
        \int_t^T
        f''\bigl(p(s,\widetilde S^{t,m}_s,\ell^{t,m}_s)\bigr)
        \delta_m p
        \bigl(s,\widetilde S^{t,m}_s,\ell^{t,m}_s\bigr)
        \bigl(\widehat N^{t,y}_s\bigr)
        \\
        &\hspace{4.8cm}
        \times
        \partial_y\delta_m q
        \bigl(s,N^{t,x}_s,\ell^{t,m}_s\bigr)
        \bigl(\widetilde S^{t,m}_s\bigr)
        \frac{\widehat W_s-\widehat W_t}{s-t}\,\mathrm{d}s
    \bigg].
\end{aligned}
\end{equation}
These two formulas are the starting point for the Gronwall argument proving the symmetry.

\medskip

We now prove the desired symmetry by a backward Gronwall argument. Define the symmetry defect
\begin{align*}
    \Delta(t,x,y,m)
    :=
    \partial_y\delta_m q(t,x,m)(y)
    -
    \partial_x\delta_m q(t,y,m)(x).
\end{align*}
Our goal is to prove that $\Delta\equiv0$. Recall that $\partial_x q=p$.

\medskip
Let us compare \eqref{eq:py_dm_q_symmetry} with \eqref{eq:px_dm_q_symmetry} written at $(t,y,m)(x)$. After exchanging the two independent Brownian copies in the latter formula, the terminal contributions cancel by the symmetry
\[
    \delta_m^2G(m)(a)(b)=\delta_m^2G(m)(b)(a).
\]

For notational convenience, set
\[
    R(s,a,b,\nu):=\partial_y\delta_m q(s,a,\nu)(b),
    \qquad
    P(s,a,b,\nu):=\delta_m p(s,a,\nu)(b).
\]
Then
\[
    \Delta(s,a,b,\nu)=R(s,a,b,\nu)-P(s,b,a,\nu).
\]
The integrand appearing in \eqref{eq:dm_q_representation_symmetry} can be written as
\begin{equation}\label{eq:A_corrected_symmetry}
\begin{aligned}
    \Ac(s,a,b,\nu)
    &:=
    f'(p(s,a,\nu))P(s,a,b,\nu)
    +
    f'(p(s,b,\nu))R(s,a,b,\nu)
    \\
    &\quad+
    \int_\R
    f''(p(s,\xi,\nu))
    P(s,\xi,b,\nu)
    R(s,a,\xi,\nu)\,\nu(\mathrm{d}\xi).
\end{aligned}
\end{equation}
Using
\[
    P(s,a,b,\nu)=R(s,b,a,\nu)-\Delta(s,b,a,\nu),
    \qquad
    P(s,\xi,b,\nu)=R(s,b,\xi,\nu)-\Delta(s,b,\xi,\nu),
\]
we get
\begin{equation}\label{eq:A_difference_defect}
\begin{aligned}
    \Ac(s,a,b,\nu)-\Ac(s,b,a,\nu)
    =
    &f'(p(s,b,\nu))\Delta(s,a,b,\nu)
    -
    f'(p(s,a,\nu))\Delta(s,b,a,\nu)
    \\
    &+
    \int_\R
    f''(p(s,\xi,\nu))
    \Big[
        R(s,b,\xi,\nu)\Delta(s,a,\xi,\nu)
        -
        R(s,a,\xi,\nu)\Delta(s,b,\xi,\nu)
    \Big]\nu(\mathrm{d}\xi).
\end{aligned}
\end{equation}
Thus the difference of the two representations is controlled only by the defect $\Delta$.

\medskip
More precisely, using the uniform boundedness of $f'$, $f''$, and of the derivatives of $q$, together with the standard estimate
\[
    \E\left[
        \left|\frac{W_s-W_t}{s-t}\right|
        +
        \left|\frac{\widehat W_s-\widehat W_t}{s-t}\right|
    \right]
    \le
    \frac{C}{\sqrt{s-t}},
    \qquad t<s\le T,
\]
we obtain
\begin{equation}\label{eq:symmetry_defect_estimate}
\begin{aligned}
    |\Delta(t,x,y,m)|
    &\le
    C
    \int_t^T
    \frac{
        \E\left[
            \left|
            \Delta(s,N^{t,x}_s,\widehat N^{t,y}_s,\ell^{t,m}_s)
            \right|_\infty
        \right]
    }{\sqrt{s-t}}
    \,\mathrm{d}s
    \\
    &\quad+
    C
    \int_t^T
    \frac{
        \E\left[
            \int_\R
            \left|
            \Delta(s,N^{t,x}_s,\xi,\ell^{t,m}_s)
            \right|_\infty
            \ell^{t,m}_s(\mathrm{d}\xi)
        \right]
    }{\sqrt{s-t}}
    \,\mathrm{d}s
    \\
    &\quad+
    C
    \int_t^T
    \frac{
        \E\left[
            \int_\R
            \left|
            \Delta(s,\widehat N^{t,y}_s,\xi,\ell^{t,m}_s)
            \right|_\infty
            \ell^{t,m}_s(\mathrm{d}\xi)
        \right]
    }{\sqrt{s-t}}
    \,\mathrm{d}s.
\end{aligned}
\end{equation}
In particular, if
\[
    D(t):=
    \sup_{x,y\in\R,\;m\in\Pc(\R)}
    |\Delta(t,x,y,m)|,
\]
then \eqref{eq:symmetry_defect_estimate} implies
\[
    D(t)
    \le
    C\int_t^T\frac{D(s)}{\sqrt{s-t}}\,\mathrm{d}s,
    \qquad t\in[0,T).
\]

To conclude, define $\overline D(t):=\sup_{r\in[t,T]}D(r)$. Then, for all $t\le r\le T$,
\[
    D(r)
    \le
    C\overline D(t)\int_r^T\frac{\mathrm{d}s}{\sqrt{s-r}}
    =
    2C\sqrt{T-r}\,\overline D(t).
\]
Taking the supremum over $r\in[t,T]$ gives
\[
    \overline D(t)\le 2C\sqrt{T-t}\,\overline D(t).
\]
Choose $\varepsilon>0$ such that $2C\sqrt{\varepsilon}<1$. Then $\overline D(t)=0$ for every $t\in[T-\varepsilon,T]$, and hence
\[
    \Delta(t,x,y,m)=0,
    \qquad t\in[T-\varepsilon,T].
\]
Iterating the argument backward in time, and using the fact that the constant $C$ is uniform, we obtain
\[
    \Delta(t,x,y,m)=0,
    \qquad
    (t,x,y,m)\in[0,T]\times\R^2\times\Pc(\R).
\]
Therefore,
\[
    \partial_y\delta_m q(t,x,m)(y)
    =
    \partial_x\delta_m q(t,y,m)(x).
\]
Since $\partial_xq=p$, this is equivalently
\[
    \partial_y\delta_m q(t,x,m)(y)
    =
    \delta_m p(t,y,m)(x).
\]
Differentiating with respect to the remaining spatial variable gives the desired symmetry property for $p$,
\[
    \partial_y\delta_m p(t,x,m)(y)
    =
    \partial_y\delta_m p(t,y,m)(x),
\]
which is {\rm \Cref{eq:p_symmetry}} in the scalar case $d=1$.

\subsection{Proof of \texorpdfstring{{\rm \Cref{thm:existence_WBSDE_gradient}} under {\rm \Cref{assum:gradient}}}{existence under assumptions}} \label{sec:existence_general}

We now prove {\rm \Cref{thm:existence_WBSDE_gradient}} under the general
assumption {\rm \Cref{assum:gradient}}. The smooth case treated above gives
the construction of a W-BSDE solution for regular coefficients. The goal of
this subsection is to remove this smoothness restriction by a stability
argument.

\smallskip
We approximate $(\Fc,G)$ by smooth coefficients $(\Fc^k,G^k)$. For each
$k$, the previous result gives a smooth characteristic field $p^k$ and a
corresponding W-BSDE solution $(Y^k,Z^k)$. The estimates imposed in
{\rm \Cref{assum:gradient}} are uniform along the approximation and imply
compactness of the sequence $(p^k)_{k\ge1}$. Passing to the limit in the
characteristic equation identifies the limit with the prescribed field $p$.
Finally, the W-BSDE identities pass to the limit by dominated convergence,
yielding a solution $(Y,Z)$ associated with the original coefficients and
satisfying $Z=p$.

\smallskip
Let us briefly recall a standard regularization procedure in the measure variable.
Let $\Ac:\Pc(\R^d)\to\R$ be bounded and continuous. For $n\ge1$ and
$\varepsilon>0$, let $G_\varepsilon$ denote the density of a centered Gaussian
random variable with covariance matrix $\varepsilon\Ir_d$. We define
\[
    \Ac_n(\boldsymbol y)
    :=
    \Ac\left(\frac1n\sum_{i=1}^n\delta_{y_i}\right),
    \qquad
    \boldsymbol y=(y_1,\ldots,y_n),
\]
and set
\begin{align} \label{eq:regularization_measure}
\widehat{\Ac}_{n,\varepsilon}(m)
:=
\int_{(\R^d)^n}\int_{(\R^d)^n}
    \Ac_n(\boldsymbol y)
    G_\varepsilon(\boldsymbol x-\boldsymbol y)
    \,\mathrm d\boldsymbol y\,
    m(\mathrm dx_1)\cdots m(\mathrm dx_n).
\end{align}
For fixed $n$ and $\varepsilon$, the map
$m\mapsto \widehat{\Ac}_{n,\varepsilon}(m)$ admits linear functional derivatives
of all orders, and these derivatives are smooth in the spatial variables. Moreover,
for every $m\in\Pc(\R^d)$,
\[
    \lim_{n\to\infty}\lim_{\varepsilon\downarrow0}
    \widehat{\Ac}_{n,\varepsilon}(m)
    =
    \Ac(m).
\]
If $\Ac$ has additional regularity in the measure variable, one can also obtain
convergence of the corresponding derivatives. We refer, for instance, to
\citeauthor*{AndreaOnsmooth2023}~\cite{AndreaOnsmooth2023} for related smoothing
results on spaces of probability measures.

\medskip

We now prove the existence of a {\rm W-BSDE} solution under
{\rm \Cref{assum:gradient}}. Let

$$
    p:[0,T)\times\R^d\times\Pc(\R^d)\longrightarrow\R^d
$$

be the unique characteristic field associated with $(T,\Fc,G)$ in the sense of
{\rm \Cref{assum:gradient}}.

\medskip
The proof proceeds by approximation. We first approximate the coefficients by
smooth ones and solve the corresponding characteristic equations. In the smooth
setting, the symmetry property established above allows us to construct a
{\rm W-BSDE} solution whose $Z$-component coincides with the corresponding
characteristic field. We then use the stability of the characteristic equation
and of the W-BSDE representation to pass to the limit and recover a solution
associated with the original coefficients $(\Fc,G)$.

\medskip

Let $\bigl(\Fc^k,G^k\bigr)_{k\ge1}$ be a sequence of smooth coefficients
approximating $(\Fc,G)$. For instance, one may use the regularization procedure
introduced in \Cref{eq:regularization_measure}, combined, when necessary, with
a standard mollification in the finite-dimensional variables. We choose the
approximation so that

$$
    \bigl(\Fc^k(t,x,m,z),G^k(m)\bigr)
    \longrightarrow
    \bigl(\Fc(t,x,m,z),G(m)\bigr),
$$

locally uniformly in $(x,m,z)$, for every $t\in[0,T]$, and, similarly,

$$
    \bigl(
        \delta_mG^k,
        \delta_m\Fc^k,
        \nabla_z\Fc^k
    \bigr)
    \longrightarrow
    \bigl(
        \delta_mG,
        \delta_m\Fc,
        \nabla_z\Fc
    \bigr)
$$

locally uniformly in their respective variables. In addition, the
regularization can be chosen so that

$$
    \underline{{\rm Lip}}_T(\Fc^k,G^k)
    \longrightarrow
    \underline{{\rm Lip}}_T(\Fc,G),
$$

that is,

$$
    {\rm Lip}_T
    \bigl(
        \delta_mG^k,
        \Fc^k,
        \delta_m\Fc^k,
        \nabla_z\Fc^k
    \bigr)
    \longrightarrow
    {\rm Lip}_T
    \bigl(
        \delta_mG,
        \Fc,
        \delta_m\Fc,
        \nabla_z\Fc
    \bigr).
$$

\medskip

For every $k$, the coefficients $(\Fc^k,G^k)$ are smooth. We may therefore
apply the smooth well-posedness result for the characteristic equation given in {fabricemattiaDecoupling}. This gives a unique smooth map

$$
    p^k:[0,T)\times\R^d\times\Pc(\R^d)\longrightarrow\R^d
$$

satisfying \eqref{eq:chac:z} with $(\Fc,G)$ replaced by $(\Fc^k,G^k)$.

\medskip
Let us emphasize that the time horizon $T$ can be chosen independently of $k$.
Indeed, in the fixed-point argument used to construct $p^k$, the admissible
time horizon depends only on the first-order Lipschitz bounds entering the
nonlinear characteristic equation, namely on quantities of the form

$$
    \underline{{\rm Lip}}_T(\Fc^k,G^k),
$$

and not on the higher-order derivatives introduced by the smoothing procedure.
The latter derivatives are needed in order to propagate smoothness and justify
the differentiations performed after the fixed point has been constructed, but
they enter only linearly in the corresponding derivative equations. In
particular, they do not affect the contraction condition defining the
characteristic field itself. Since

$$
    \underline{{\rm Lip}}_T(\Fc^k,G^k)
    \longrightarrow
    \underline{{\rm Lip}}_T(\Fc,G),
$$

the contraction constants remain in a fixed bounded set. Hence, after
discarding finitely many indices if necessary, the same time horizon $T$ may be
used for all $k$.

\medskip

For each $k$, the smooth characteristic field $p^k$ satisfies the symmetry
property established in the previous subsection. Consequently, the smooth
construction applies and yields a {\rm W-BSDE} solution

$$
    (Y^k,Z^k)
$$

associated with $(\Fc^k,G^k)$, with

$$
    Z^k(t,x,m)=p^k(t,x,m).
$$

More precisely, if

$$
    F^k(t,m)
    :=
    \int_{\R^d}
    \Fc^k\bigl(t,x,m,p^k(t,x,m)\bigr)\,m(\mathrm{d}x),
$$

then
\begin{equation}\label{eq:Yk_rep_existence}
\begin{aligned}
Y^k(t,m)
&=
\E\left[
G^k(M^{t,m}_T)
+
\int_t^T
F^k(s,M^{t,m}_s)\,\mathrm{d}s
\right]
\\
&=
\E\left[
G^k(M^{t,m}_T)
+
\int_t^T
\Fc^k\bigl(
s,N^{t,U}_s,M^{t,m}_s,
p^k(s,N^{t,U}_s,M^{t,m}_s)
\bigr)\,\mathrm{d}s
\right],
\end{aligned}
\end{equation}
where $\Lc(U)=m$ and

$$
    M^{t,m}_s
    :=
    \Lc(N^{t,U}_s\mid\Gc_T)
    =
    \Lc(N^{t,U}_s\mid\Gc_s),
    \qquad s\in[t,T].
$$

\medskip

We next derive estimates which are uniform in $k$. By the a priori estimates
for the characteristic equation, together with the convergence of
$\underline{{\rm Lip}}_T(\Fc^k,G^k)$, for every $\overline T<T$ there exist constants $\Kr>0$, $\alpha\in(0,1)$,
depending only on $\overline T$, $T$ and $\underline{{\rm Lip}}_T(\Fc,G)$, such
that
\begin{align}
|p^k(t,x,m)-p^k(t',x',m')|
\le
\Kr
\left(
|t-t'|^\alpha
+
|x-x'|
+
\Wc_\circ(m,m')
\right),
\label{eq:uniform_pk_estimate}
\end{align}
for all $(t,t',x,x',m,m')\in[0,\overline T]^2\times(\R^d)^2\times\Pc(\R^d)^2.$
Moreover, for every compact

$$
    K\subset [0,\overline T]\times\R^d\times\Pc(\R^d),
$$
one has
$$
    \sup_{k\ge1}\sup_{(t,x,m)\in K}
    |p^k(t,x,m)|<\infty.
$$

Since $\R^d\times\Pc(\R^d)$ is separable for the topology under consideration,
the uniform equicontinuity estimate \eqref{eq:uniform_pk_estimate}, together
with a diagonal argument on a countable dense subset, implies that
$(p^k)_{k\ge1}$ is relatively compact for the topology of local uniform
convergence. Hence, up to extracting a subsequence, there exists a Borel map

$$
    \overline p:
    [0,T)\times\R^d\times\Pc(\R^d)
    \longrightarrow\R^d
$$

such that, for every $\overline T<T$,
$$
    p^k\longrightarrow\overline p
$$
locally uniformly on $[0,\overline T]\times\R^d\times\Pc(\R^d).$

Passing to the limit in \eqref{eq:uniform_pk_estimate} shows that $\overline p$
belongs to the same H\"older-Lipschitz regularity class as the characteristic
field in {\rm \Cref{assum:gradient}}.

\medskip

We now identify $\overline p$. For every $k$, every $i=1,\dots,d$, and every
$(t,x,m)\in[0,T)\times\R^d\times\Pc(\R^d)$, one has
\begin{align}
(p^k)^i(t,x,m)
&=
\E \left[
\delta_mG^k(\ell^{p^k,t,m}_T)(N^{t,x}_T)
\frac{W^i_T-W^i_t}{T-t}
\right]
\nonumber
\\
&\quad+
\E \left[
\int_t^T
\Fc^k\left(
s,N^{t,x}_s,\ell^{p^k,t,m}_s,
p^k(s,N^{t,x}_s,\ell^{p^k,t,m}_s)
\right)
\frac{W^i_s-W^i_t}{s-t}\,\mathrm{d}s
\right]
\nonumber
\\
&\quad+
\E \left[
\int_t^T
\delta_m\Fc^k\left(
s,N^{t,U}_s,\ell^{p^k,t,m}_s,
p^k(s,N^{t,U}_s,\ell^{p^k,t,m}_s)
\right)(\Nh^{t,x}_s)
\frac{\Wh^i_s-\Wh^i_t}{s-t}\,\mathrm{d}s
\right].
\label{eq:char_pk_limit}
\end{align}

The stability of the conditional McKean-Vlasov equation defining
$\ell^{p^k,t,m}$ gives

$$
    \ell^{p^k,t,m}_s
    \longrightarrow
    \ell^{\overline p,t,m}_s,
    \qquad s\in[t,T],
$$

in $\Wc_\circ$, locally uniformly in the initial data.
We may consequently pass to the limit in \eqref{eq:char_pk_limit}. For the
integral terms, we use the standard Bismut estimate

$$
    \E\left[
        \left|
            \frac{W^i_s-W^i_t}{s-t}
        \right|
    \right]
    +
    \E\left[
        \left|
            \frac{\Wh^i_s-\Wh^i_t}{s-t}
        \right|
    \right]
    \le
    \frac{C}{\sqrt{s-t}},
    \qquad t<s\le T,
$$

whose singularity is integrable on $[t,T]$. Together with the uniform bounds
on the coefficients and their first-order derivatives, this provides a
dominating integrable function. Hence dominated convergence gives that
$\overline p$ satisfies the characteristic equation \eqref{eq:chac:z}
associated with $(T,\Fc,G)$.
By the uniqueness part of {\rm \Cref{assum:gradient}}, we obtain

$$
    \overline p=p.
$$

In particular, every convergent subsequence of $(p^k)_{k\ge1}$ has the same
limit. Hence the whole sequence converges:
\begin{align}
p^k\longrightarrow p
\quad\text{locally uniformly on}\quad
[0,T)\times\R^d\times\Pc(\R^d).
\label{eq:pk_to_p}
\end{align}

\medskip

We now pass to the limit in the value functions. From
\eqref{eq:Yk_rep_existence}, the convergence of $(\Fc^k,G^k)$ and
\eqref{eq:pk_to_p}, together with dominated convergence, yield, for every
$(t,m)\in[0,T]\times\Pc(\R^d)$,

$$
    Y^k(t,m)\longrightarrow Y(t,m),
$$

where
\begin{align}
Y(t,m)
&:=
\E\left[
G(M^{t,m}_T)
+
\int_t^T
\Fc\bigl(
s,N^{t,U}_s,M^{t,m}_s,
p(s,N^{t,U}_s,M^{t,m}_s)
\bigr)\,\mathrm{d}s
\right].
\label{eq:def_Y_limit}
\end{align}
We set

$$
    Z(t,x,m):=p(t,x,m).
$$

\medskip

It remains to verify that $(Y,Z)$ satisfies the W-BSDE identity. Let
$\mu\in\Lc$ be an arbitrary admissible conditional law flow, associated with a
drift $B$ and a state process $X$ satisfying

$$
    \mathrm{d}X_t
    =
    B(t,X_t,\mu_t)\,\mathrm{d}t
    +
    \mathrm{d}W_t
    +
    \sigma_\circ\,\mathrm{d}W^\circ_t,
    \qquad
    \mu_t
    =
    \Lc(X_t\mid\Gc_t)
    =
    \Lc(X_t\mid\Gc_T).
$$

For every $k$, the smooth W-BSDE identity gives
\begin{align}
Y^k(t,\mu_t)
=
G^k(\mu_T)
&+
\E\left[
\int_t^T
\Fc^k\bigl(
r,X_r,\mu_r,
p^k(r,X_r,\mu_r)
\bigr)\,\mathrm{d}r
\,\middle|\,\Gc_T
\right]
\nonumber
\\
&-
\E\left[
\int_t^T
p^k(r,X_r,\mu_r)
\cdot B(r,X_r,\mu_r)\,\mathrm{d}r
\,\middle|\,\Gc_T
\right]
\nonumber
\\
&-
\int_t^T
\E\left[
p^k(r,X_r,\mu_r)\mid\Gc_T
\right]
\cdot\sigma_\circ\,\mathrm{d}W^\circ_r.
\label{eq:WBSDE_k_limit}
\end{align}

The convergence \eqref{eq:pk_to_p} implies, for a.e.\ $(r,x,\omega)$,

$$
    p^k(r,x,\mu_r)
    \longrightarrow
    p(r,x,\mu_r).
$$

The uniform estimates and the standing integrability assumptions therefore give
\begin{align}
\E\left[
\int_0^T
\left|
p^k(r,X_r,\mu_r)-p(r,X_r,\mu_r)
\right|
\mathrm{d}r
\right]
\longrightarrow0.
\label{eq:pk_L2_flow}
\end{align}

The It\^o isometry then gives convergence of the common-noise stochastic
integrals in \eqref{eq:WBSDE_k_limit}. The two finite-variation terms converge
by dominated convergence. Hence, letting $k\to\infty$ in
\eqref{eq:WBSDE_k_limit}, we obtain
\begin{align}
Y(t,\mu_t)
=
G(\mu_T)
&+
\E\left[
\int_t^T
\Fc\bigl(
r,X_r,\mu_r,
Z(r,X_r,\mu_r)
\bigr)\,\mathrm{d}r
\,\middle|\,\Gc_T
\right]
\nonumber
\\
&-
\E\left[
\int_t^T
Z(r,X_r,\mu_r)\cdot B(r,X_r,\mu_r)
\,\mathrm{d}r
\,\middle|\,\Gc_T
\right]
\nonumber
\\
&-
\int_t^T
\E\left[
Z(r,X_r,\mu_r)\mid\Gc_T
\right]
\cdot\sigma_\circ\,\mathrm{d}W^\circ_r.
\label{eq:WBSDE_limit_final}
\end{align}
Moreover,

$$
    \E\left[
        \int_0^T
        \left(
            \left|
                \Fc(r,X_r,\mu_r,Z(r,X_r,\mu_r))
            \right|
            +
            |Z(r,X_r,\mu_r)|
            +
            \left|
                \E[Z(r,X_r,\mu_r)\mid\Gc_T]
            \right|^2
        \right)\mathrm{d}r
    \right]
    <\infty.
$$

Since $\mu\in\Lc$ was arbitrary, \eqref{eq:WBSDE_limit_final} proves that
$(Y,Z)$ is a {\rm W-BSDE} solution associated with $(\Fc,G)$ on $[0,T]$.

\medskip

We finally identify its Wasserstein gradient. In the smooth case, we have

$$
    \delta_mY^k(t,m)(x)=q^k(t,x,m),
    \qquad
    \nabla_xq^k(t,x,m)=p^k(t,x,m),
$$

where
\begin{align}
q^k(t,x,m)
=
\E\left[
\delta_mG^k(M^{t,m}_T)(N^{t,x}_T)
+
\int_t^T
\delta_mF^k(s,M^{t,m}_s)(N^{t,x}_s)
\,\mathrm{d}s
\right].
\label{eq:qk_rep_limit}
\end{align}
The convergence of the coefficients and of $p^k$ implies, again by dominated
convergence, that

$$
    q^k(t,x,m)\longrightarrow q(t,x,m)
$$

locally uniformly, where

$$
    q(t,x,m)
    :=
    \E\left[
        \delta_mG(M^{t,m}_T)(N^{t,x}_T)
        +
        \int_t^T
        \delta_mF(s,M^{t,m}_s)(N^{t,x}_s)
        \,\mathrm{d}s
    \right],
$$

with

$$
    F(t,m)
    :=
    \int_{\R^d}
    \Fc(t,x,m,p(t,x,m))\,m(\mathrm{d}x).
$$

To see that $q$ is indeed a linear functional derivative of $Y$, fix
$m,m'\in\Pc(\R^d)$ and set

$$
    m_\lambda:=(1-\lambda)m+\lambda m',
    \qquad \lambda\in[0,1].
$$

For every $k$,

$$
    Y^k(t,m')-Y^k(t,m)
    =
    \int_0^1
    \int_{\R^d}
    q^k(t,x,m_\lambda)
    \,(m'-m)(\mathrm{d}x)
    \,\mathrm{d}\lambda.
$$

Passing to the limit yields

$$
    Y(t,m')-Y(t,m)
    =
    \int_0^1
    \int_{\R^d}
    q(t,x,m_\lambda)
    \,(m'-m)(\mathrm{d}x)
    \,\mathrm{d}\lambda.
$$

Thus $q$ is a version of the linear functional derivative of $Y$:

$$
    \delta_mY(t,m)(x)=q(t,x,m).
$$

Finally, since

$$
    \nabla_xq^k=p^k,
$$

and both $q^k\to q$ and $p^k\to p$ locally uniformly, the standard stability
property for spatial derivatives gives

$$
    \nabla_xq(t,x,m)=p(t,x,m).
$$

Consequently,

$$
    \nabla_x\delta_mY(t,m)(x)
    =
    p(t,x,m)
    =
    Z(t,x,m),
$$

for the relevant $(t,x,m)$. This completes the construction and the
identification of the {\rm W-BSDE} solution.

\subsection{Proof of \texorpdfstring{\rm \Cref{prop:comparison}}{proposition}}

Let $(Y^i,Z^i)_{i=1,2}$ be two W-BSDE solutions associated with
$(\Fc^i,G^i)_{i=1,2}$. Fix $(e,m)\in[0,T]\times\Pc(\R^d)$. For
$c,\varepsilon>0$, define the truncated drift
\begin{align*}
B^{c,\varepsilon}(t,x,m)
&:=
\frac{
\Fc^1(t,x,m,Z^1(t,x,m))
-
\Fc^1(t,x,m,Z^2(t,x,m))
}{
\big|Z^1(t,x,m)-Z^2(t,x,m)\big|^2
}
\big(Z^1(t,x,m)-Z^2(t,x,m)\big)\\
&\qquad\times
\chi_c\bigl(|Z^1(t,x,m)-Z^2(t,x,m)|\bigr)
\mathbf 1_{\{t\le T-\varepsilon\}},
\end{align*}
where $\chi_c$ is a smooth cut-off function satisfying
$\chi_c(r)=0$ for $r\le c/2$ and $\chi_c(r)=1$ for $r\ge c$. Since
$B^{c,\varepsilon}$ vanishes for $t>T-\varepsilon$, the assumptions above on
$\Fc^1$ and on $Z^1,Z^2$ show that $B^{c,\varepsilon}$ is bounded and Lipschitz
in the measure variable, uniformly in $(t,x)$.
Therefore the McKean-Vlasov equation
\[
    \mathrm dX_t
    =
    B^{c,\varepsilon}(t,X_t,\mu_t)\,\mathrm dt
    +\mathrm dW_t+\sigma_\circ\,\mathrm dW^\circ_t,
    \qquad
    \mu_t=\Lc(X_t\mid\Gc_t),
\]
started from $\mu_e=m$, generates an admissible conditional law flow.
Applying the W-BSDE identity to both solutions along this same flow and taking the
difference, we obtain
\[
\begin{aligned}
Y^1(e,m)-Y^2(e,m)
&=
\E\bigg[
G^1(\mu_T)-G^2(\mu_T)
\\
&\quad+
\int_e^T
\Big(
\Fc^1(s,X_s,\mu_s,Z^1(s,X_s,\mu_s))
-
\Fc^2(s,X_s,\mu_s,Z^2(s,X_s,\mu_s))
\Big)\,\mathrm ds
\,\bigg|\,\Gc_T
\bigg]
\\
&\quad
-
\E\bigg[
\int_e^T
B^{c,\varepsilon}(s,X_s,\mu_s)\cdot
\big(Z^1(s,X_s,\mu_s)-Z^2(s,X_s,\mu_s)\big)
\,\mathrm ds
\,\bigg|\,\Gc_T
\bigg]
\\
&\quad
-
\int_e^T
\E\big[
Z^1(s,X_s,\mu_s)-Z^2(s,X_s,\mu_s)
\,\big|\,\Gc_T
\big]\cdot\sigma_\circ\,\mathrm dW^\circ_s .
\end{aligned}
\]
Taking expectation removes the common-noise martingale term. We then use the
definition of $B^{c,\varepsilon}$ to cancel the part of the difference of the
generators coming from the $z$-variable. More precisely,
\[
\begin{aligned}
Y^1(e,m)-Y^2(e,m)
&=
\E\bigg[
G^1(\mu_T)-G^2(\mu_T)
\\
&\quad+
\int_e^T
\Big(
\Fc^1(s,X_s,\mu_s,Z^2(s,X_s,\mu_s))
-
\Fc^2(s,X_s,\mu_s,Z^2(s,X_s,\mu_s))
\Big)\,\mathrm ds
\bigg]
+
R_{c,\varepsilon},
\end{aligned}
\]
where the remainder $R_{c,\varepsilon}$ comes only from the truncation on the set
$\{|Z^1-Z^2|\le c\}$ and from the terminal layer $[T-\varepsilon,T]$.
Indeed, by definition of $B^{c,\varepsilon}$, the integrand
of the remainder is $\bigl(1-\chi_c(|Z^1-Z^2|)\mathbf 1_{\{s\le T-\varepsilon\}}\bigr)$ times the difference $\Fc^1(\cdot,Z^1)-\Fc^1(\cdot,Z^2)$, which is bounded by $Lc$ on the truncation set, $L$ being the Lipschitz constant of $\Fc^1$ in $z$, and by $2\|\Fc^1\|_\infty$ on the terminal layer. Hence there is a constant $C$, independent of $c,\varepsilon$, such that
\[
    |R_{c,\varepsilon}|
    \le C(c+\varepsilon).
\]
Hence $R_{c,\varepsilon}\to0$ as $c\downarrow0$ and
$\varepsilon\downarrow0$.
If $G^1\le G^2$ and $\Fc^1\le \Fc^2$, then the main term in the previous
identity is nonpositive. Letting first $c\downarrow0$ and then
$\varepsilon\downarrow0$, we obtain
\[
        Y^1(e,m)-Y^2(e,m)\le0.
\]
Since $(e,m)$ was arbitrary, this proves
\[
        Y^1(e,m)\le Y^2(e,m),
        \qquad (e,m)\in[0,T]\times\Pc(\R^d).
\]

\subsection{Proof of W--BSDEs and PDEs}

We derive time differentiability from the W--BSDE identity and the spatial
regularity in {\rm \Cref{prop:WBSDE_PDE}}. The key is to apply the It\^o
formula with the time argument of $Y$ held fixed.

Set $\Uc:=Y$ and define the operator containing only the spatial and
measure derivatives by
\[
\begin{aligned}
    \mathcal A\Uc(r,m)
    &:=\frac12\int_{\R^d}
    \Tr\left[
        \bigl(\Ir_d+\sigma_\circ\sigma_\circ^\top\bigr)
        \nabla_x^2\delta_m\Uc(r,m)(x)
    \right]m(\mathrm dx)
    \\
    &\quad+\frac12\int_{\R^d\times\R^d}
    \Tr\left[
        \sigma_\circ\sigma_\circ^\top
        \nabla_y\nabla_x\delta_m^2\Uc(r,m)(x,y)
    \right]m(\mathrm dx)m(\mathrm dy).
\end{aligned}
\]
The regularity and boundedness assumptions imply that
$(r,m)\mapsto\mathcal A\Uc(r,m)$ is continuous and bounded on
$[0,\overline T]\times\Pc(\R^d)$ for every $\overline T<T$.

\medskip
Fix $0\le a<b<T$ and $m\in\Pc(\R^d)$. Consider the admissible flow
with zero drift, starting from $m$ at time $a$:
\[
    X_s=\xi+W_s-W_a+\sigma_\circ(W_s^\circ-W_a^\circ),
    \qquad
    \mu_s^{a,m}=\Lc(X_s\mid\Gc_s),
    \qquad s\in[a,b],
\]
where $\xi$ has law $m$ and is independent of both Brownian motions.
In particular, $\mu_a^{a,m}=m$.

Subtracting the W--BSDE identities at $a$ and $b$ and taking expectation
gives the dynamic programming identity
\begin{equation}\label{eq:WBSDE_short_time_DPP}
    \Uc(a,m)
    =\E\left[
        \Uc(b,\mu_b^{a,m})
        +\int_a^b\overline{\Fc}(s,\mu_s^{a,m})\,\mathrm ds
    \right].
\end{equation}
The common-noise stochastic integral has zero expectation by the assumed
bounds on $Z$.

We now apply the It\^o formula to the fixed functional
\[
    m'\longmapsto\Uc(b,m')
\]
along $(\mu_s^{a,m})_{s\in[a,b]}$. This uses only derivatives in the
measure and spatial variables. It gives
\[
\begin{aligned}
    \Uc(b,\mu_b^{a,m})-\Uc(b,m)
    =\int_a^b\mathcal A\Uc(b,\mu_s^{a,m})\,\mathrm ds
    +\int_a^b
    \left(
        \int_{\R^d}
        \nabla_x\delta_m\Uc(b,\mu_s^{a,m})(x)
        \,\mu_s^{a,m}(\mathrm dx)
    \right)\cdot\sigma_\circ\,\mathrm dW_s^\circ.
\end{aligned}
\]
Taking expectation and combining with
\eqref{eq:WBSDE_short_time_DPP}, we obtain
\begin{equation}\label{eq:WBSDE_time_increment}
    \Uc(b,m)-\Uc(a,m)
    =-\E\left[\int_a^b
        \Bigl(
            \overline{\Fc}(s,\mu_s^{a,m})
            +\mathcal A\Uc(b,\mu_s^{a,m})
        \Bigr)\,\mathrm ds
    \right].
\end{equation}
In particular, the bounds on the integrands imply
\[
    |\Uc(b,m)-\Uc(a,m)|
    \le C_{\overline T}(b-a),
    \qquad 0\le a<b\le\overline T<T,
\]
so $\Uc$ is locally Lipschitz in time, uniformly in $m$.

\medskip
Fix $(t,m)\in[0,T)\times\Pc(\R^d)$ and take $a=t$, $b=t+h$ in
\eqref{eq:WBSDE_time_increment}. Then
\[
    \frac{\Uc(t+h,m)-\Uc(t,m)}{h}
    =-\frac1h\E\left[\int_t^{t+h}
        \Bigl(
            \overline{\Fc}(s,\mu_s^{t,m})
            +\mathcal A\Uc(t+h,\mu_s^{t,m})
        \Bigr)\,\mathrm ds
    \right].
\]
Since $\mu_t^{t,m}=m$ and the conditional law flow is continuous,
joint continuity and boundedness of the integrands yield
\begin{equation}\label{eq:WBSDE_right_time_derivative}
    \lim_{h\downarrow0}
    \frac{\Uc(t+h,m)-\Uc(t,m)}{h}
    =-\overline{\Fc}(t,m)-\mathcal A\Uc(t,m).
\end{equation}

For $t>0$, applying \eqref{eq:WBSDE_time_increment} on $[t-h,t]$ to
a flow starting from $m$ at time $t-h$ similarly gives
\[
    \frac{\Uc(t,m)-\Uc(t-h,m)}{h}
    =-\frac1h\E\left[\int_{t-h}^{t}
        \Bigl(
            \overline{\Fc}(s,\mu_s^{t-h,m})
            +\mathcal A\Uc(t,\mu_s^{t-h,m})
        \Bigr)\,\mathrm ds
    \right].
\]
The Brownian motion regularity implies that
$\mu_s^{t-h,m}\to m$ in probability, uniformly for $s\in[t-h,t]$,
as $h\downarrow0$. Thus the left difference quotient has the same limit
as in \eqref{eq:WBSDE_right_time_derivative}.
Consequently, $\partial_t\Uc(t,m)$ exists for every $t<T$, with a
right derivative at $t=0$, and
\begin{equation}\label{eq:PDE_WBSDE_proof}
    \partial_t\Uc(t,m)
    +\overline{\Fc}(t,m)
    +\mathcal A\Uc(t,m)=0.
\end{equation}
The right-hand side of
$\partial_t\Uc=-\overline{\Fc}-\mathcal A\Uc$ is jointly continuous,
so $\partial_t\Uc$ is jointly continuous as well. Using
$Z=\nabla_x\delta_m\Uc$ and the definitions of $\overline{\Fc}$ and
$\mathcal A$, \eqref{eq:PDE_WBSDE_proof} is precisely
\eqref{eq:PDE_WBSDE}.
Finally, the terminal condition in the W--BSDE gives
\[
    \Uc(T,m)=Y(T,m)=G(m),\qquad m\in\Pc(\R^d).
\]
Hence $\Uc$ is a classical solution of
\eqref{eq:PDE_WBSDE}.
\subsection{Proof of W-BSDEs as the mean-field limit of classical BSDEs}

We now prove that the {\rm W--BSDE} arises as the mean--field limit of the
classical $n$--dimensional BSDEs \eqref{eq:n_particle_BSDE_statement}. The idea
is to evaluate the Wasserstein solution $Y$ along the empirical measure
$\mu^n_t$ of the particle system, through a finite--dimensional lift of $Y$ to
$(\R^d)^n$. It\^o's formula shows that the lifted process solves the
$n$--particle BSDE up to a remainder of order $1/n$. A comparison argument then
yields the convergence of the first component, while a linearisation of the
difference of the generators, combined with a Girsanov change of measure,
yields that of the rescaled martingale integrands, i.e.
\[
    nZ^{i,n}_t
    \longrightarrow
    Z(t,N^i_t,\mu^n_t).
\]

\medskip
Throughout the proof, we write
\[
    \Uc(t,m):=Y(t,m),
    \qquad
    p(t,x,m):=Z(t,x,m)=\nabla_x\delta_m\Uc(t,m)(x).
\]
By the assumptions of the proposition and by {\rm \Cref{prop:WBSDE_PDE}}, the
function $\Uc$ is a classical solution of \Cref{eq:PDE_WBSDE_proof}, with
terminal condition $\Uc(T,\cdot)=G$.

For $\xbb=(x_1,\ldots,x_n)\in(\R^d)^n$, we set
\[
    m^n_{\xbb}:=\frac1n\sum_{i=1}^n\delta_{x_i},
    \qquad
    U^n(t,\xbb):=\Uc(t,m^n_{\xbb}).
\]
We first compute the derivatives of this finite-dimensional lift. By the
chain rule for linear functional derivatives, for every $i=1,\ldots,n$,
\[
    \nabla_{x_i}U^n(t,\xbb)
    =
    \frac1n
    \nabla_x\delta_m\Uc(t,m^n_{\xbb})(x_i)
    =
    \frac1n p(t,x_i,m^n_{\xbb}).
\]
Moreover, for $i\neq j$,
\[
    \nabla^2_{x_i x_j}U^n(t,\xbb)
    =
    \frac1{n^2}
    \nabla_y\nabla_x\delta_m^2\Uc(t,m^n_{\xbb})(x_i,x_j),
\]
while, for $i=j$,
\[
    \nabla^2_{x_i x_i}U^n(t,\xbb)
    =
    \frac1n
    \nabla_x^2\delta_m\Uc(t,m^n_{\xbb})(x_i)
    +
    \frac1{n^2}
    \nabla_y\nabla_x\delta_m^2\Uc(t,m^n_{\xbb})(x_i,x_i).
\]
The first term in the diagonal derivative comes from the differentiation of
the spatial variable $x_i$, while the second term comes from the
dependence of the empirical measure $m^n_{\xbb}$ on $x_i$.

Along the particle system $(N^1,\ldots,N^n)$ and its empirical measure $\mu^n_t$
introduced in \Cref{sec:n_player}, we define
\[
    \overline Y^n_t
    :=
    U^n(t,N^1_t,\ldots,N^n_t)
    =
    \Uc(t,\mu^n_t),
\]
and
\[
    \overline Z^{i,n}_t
    :=
    \frac1n
    \nabla_x\delta_m\Uc(t,\mu^n_t)(N^i_t)
    =
    \frac1n p(t,N^i_t,\mu^n_t),
    \qquad i=1,\ldots,n.
\]
Equivalently,
\[
    n\overline Z^{i,n}_t
    =
    p(t,N^i_t,\mu^n_t)
    =
    Z(t,N^i_t,\mu^n_t).
\]

We apply Itô's formula to
\[
    \overline Y^n_t=U^n(t,N^1_t,\ldots,N^n_t).
\]
Since
\[
    \mathrm{d}N^i_t
    =
    \mathrm{d}W^i_t+\sigma_\circ\,\mathrm{d}W^\circ_t,
\]
we have
\[
    \mathrm{d}\langle N^i,N^j\rangle_t
    =
    \left(
        \mathbf 1_{\{i=j\}}\Ir_d
        +
        \sigma_\circ\sigma_\circ^\top
    \right)\mathrm{d}t.
\]
Therefore,
\begin{align*}
    \mathrm{d}\overline Y^n_t
    &=
    \partial_t U^n(t,N_t)\,\mathrm{d}t
    +
    \sum_{i=1}^n
    \nabla_{x_i}U^n(t,N_t)\cdot
    \bigl(\mathrm{d}W^i_t+\sigma_\circ\,\mathrm{d}W^\circ_t\bigr)
    \\
    &\quad
    +
    \frac12
    \sum_{i,j=1}^n
    \Tr\left[
        \left(
            \mathbf 1_{\{i=j\}}\Ir_d
            +
            \sigma_\circ\sigma_\circ^\top
        \right)
        \nabla^2_{x_i x_j}U^n(t,N_t)
    \right]\mathrm{d}t,
\end{align*}
where $N_t=(N^1_t,\ldots,N^n_t)$. Using the expression of
$\nabla_{x_i}U^n$, the martingale part is
\[
    \sum_{i=1}^n
    \overline Z^{i,n}_t\cdot
    \bigl(\mathrm{d}W^i_t+\sigma_\circ\,\mathrm{d}W^\circ_t\bigr).
\]

It remains to identify the drift. Using the second-order derivatives computed
above, we obtain
\begin{align*}
    \frac12
    \sum_{i,j=1}^n
    \Tr\left[
        \left(
            \mathbf 1_{\{i=j\}}\Ir_d
            +
            \sigma_\circ\sigma_\circ^\top
        \right)
        \nabla^2_{x_i x_j}U^n(t,N_t)
    \right]
    &=
    \frac12
    \frac1n\sum_{i=1}^n
    \Tr\left[
        \left(\Ir_d+\sigma_\circ\sigma_\circ^\top\right)
        \nabla_x^2\delta_m\Uc(t,\mu^n_t)(N^i_t)
    \right]
    \\
    &\quad
    +
    \frac12
    \frac1{n^2}
    \sum_{i,j=1}^n
    \Tr\left[
        \sigma_\circ\sigma_\circ^\top
        \nabla_y\nabla_x\delta_m^2\Uc(t,\mu^n_t)(N^i_t,N^j_t)
    \right]
    \\
    &\quad
    +
    \frac12
    \frac1{n^2}
    \sum_{i=1}^n
    \Tr\left[
        \nabla_y\nabla_x\delta_m^2\Uc(t,\mu^n_t)(N^i_t,N^i_t)
    \right].
\end{align*}
In terms of the empirical measure $\mu^n_t$, this can be rewritten as
\begin{align*}
    &\frac12
    \int_{\R^d}
    \Tr\left[
        \left(\Ir_d+\sigma_\circ\sigma_\circ^\top\right)
        \nabla_x^2\delta_m\Uc(t,\mu^n_t)(x)
    \right]\mu^n_t(\mathrm{d}x)
    \\
    &\quad
    +
    \frac12
    \int_{\R^d}\int_{\R^d}
    \Tr\left[
        \sigma_\circ\sigma_\circ^\top
        \nabla_y\nabla_x\delta_m^2\Uc(t,\mu^n_t)(x,y)
    \right]\mu^n_t(\mathrm{d}x)\mu^n_t(\mathrm{d}y)
    \\
    &\quad
    +
    \frac1{2n}
    \int_{\R^d}
    \Tr\left[
        \nabla_y\nabla_x\delta_m^2\Uc(t,\mu^n_t)(x,x)
    \right]\mu^n_t(\mathrm{d}x).
\end{align*}
The first two terms are exactly the second-order terms appearing in the PDE
\eqref{eq:PDE_WBSDE_proof}, evaluated at $m=\mu^n_t$. Hence, using the PDE,
we find
\begin{align*}
    \mathrm{d}\overline Y^n_t
    &=
    \left(
        -
        \frac1n\sum_{i=1}^n
        \Fc\bigl(t,N^i_t,\mu^n_t,p(t,N^i_t,\mu^n_t)\bigr)
        +
        r^n_t
    \right)\mathrm{d}t
    +
    \sum_{i=1}^n
    \overline Z^{i,n}_t\cdot
    \bigl(\mathrm{d}W^i_t+\sigma_\circ\,\mathrm{d}W^\circ_t\bigr),
\end{align*}
where
\[
    r^n_t
    :=
    \frac1{2n}
    \int_{\R^d}
    \Tr\left[
        \nabla_y\nabla_x\delta_m^2\Uc(t,\mu^n_t)(x,x)
    \right]\mu^n_t(\mathrm{d}x).
\]
Rewriting $p(t,N^i_t,\mu^n_t)$ as $n\overline Z^{i,n}_t$, so that the generator
takes the same form as in \eqref{eq:n_particle_BSDE_statement}, setting
$\Rc^n_s:=-r^n_s$, and integrating from $t$ to $T$, we obtain
\begin{align} \label{eq:bar_BSDE_with_R}
    \overline Y^n_t
    &=
    G(\mu^n_T)
    +
    \frac1n\sum_{i=1}^n
    \int_t^T
    \Fc\bigl(s,N^i_s,\mu^n_s,n\overline Z^{i,n}_s\bigr)
    \,\mathrm{d}s
    +
    \int_t^T \Rc^n_s\,\mathrm{d}s
    -
    \sum_{i=1}^n
    \int_t^T
    \overline Z^{i,n}_s\cdot
    \bigl(\mathrm{d}W^i_s+\sigma_\circ\,\mathrm{d}W^\circ_s\bigr).
\end{align}

By the assumptions of the proposition, $\nabla_y\nabla_x\delta_m^2\Uc$ is
bounded, and therefore
\begin{align} \label{eq:remainder_bound}
    \sup_{s\in[0,T]}\left|\Rc^n_s\right|
    \le
    \frac{M_2}{2n},
    \qquad
    M_2
    :=
    \sup_{t,m,x}
    \left|
        \Tr\left[
            \nabla_y\nabla_x\delta_m^2\Uc(t,m)(x,x)
        \right]
    \right|
    <\infty .
\end{align}

We now compare $(\overline Y^n,\overline{\Zbb}^n)$ with the solution
$(Y^n,\Zbb^n)$ of the $n$-particle BSDE
\begin{align} \label{eq:true_n_BSDE_proof}
    Y^n_t
    &=
    G(\mu^n_T)
    +
    \frac1n\sum_{i=1}^n
    \int_t^T
    \Fc\bigl(s,N^i_s,\mu^n_s,nZ^{i,n}_s\bigr)
    \,\mathrm{d}s
    -
    \sum_{i=1}^n
    \int_t^T
    Z^{i,n}_s\cdot
    \bigl(\mathrm{d}W^i_s+\sigma_\circ\,\mathrm{d}W^\circ_s\bigr).
\end{align}
Since the two generators differ by the additive term $\Rc^n$ only, the
standard comparison theorem for BSDEs gives
\begin{align} \label{eq:inequality_Y}
    \left|Y^n_t-\overline Y^n_t\right|
    \le
    (T-t)\sup_{s\in[0,T]}\left|\Rc^n_s\right|
    \le
    \frac{C}{n},
    \qquad t\in[0,T],
\end{align}
for a constant $C$ independent of $n$, where we used \eqref{eq:remainder_bound}.
Since $\overline Y^n_t=Y(t,\mu^n_t)$, in particular we have the pathwise bound
\[\sup_{t\in[0,T]}
    \left|Y^n_t-Y(t,\mu^n_t)\right|
    =
    \sup_{t\in[0,T]}
    \left|Y^n_t-\overline Y^n_t\right|
    \le
    \frac{C}{n}.
    \]

We now prove the estimate on the martingale integrands. Let
\[
    \Delta Y_t:=Y^n_t-\overline Y^n_t,
    \qquad
    \Delta Z^i_t:=Z^{i,n}_t-\overline Z^{i,n}_t,
    \qquad i=1,\ldots,n.
\]
Subtracting \eqref{eq:bar_BSDE_with_R} from \eqref{eq:true_n_BSDE_proof}, we obtain
\begin{align} \label{eq:delta_BSDE}
    \Delta Y_t
    &=
    \int_t^T
    \Delta \Fc_s\,\mathrm{d}s
    -
    \sum_{i=1}^n
    \int_t^T
    \Delta Z^i_s\cdot
    \bigl(\mathrm{d}W^i_s+\sigma_\circ\,\mathrm{d}W^\circ_s\bigr),
\end{align}
where, since $n\overline Z^{i,n}_s=Z(s,N^i_s,\mu^n_s)$ and by definition of $a^i$,
\begin{align} \label{eq:Delta_F_def}
    \Delta \Fc_s
    &:=
    \frac1n\sum_{i=1}^n
    \left[
        \Fc\bigl(s,N^i_s,\mu^n_s,nZ^{i,n}_s\bigr)
        -
        \Fc\bigl(s,N^i_s,\mu^n_s,n\overline Z^{i,n}_s\bigr)
    \right]
    -
    \Rc^n_s
    =
    \sum_{i=1}^n a^i_s\cdot\Delta Z^i_s-\Rc^n_s .
\end{align}
A direct energy estimate on \eqref{eq:delta_BSDE} would only use the bound
$|\Delta\Fc_s|\le L\sum_i|\Delta Z^i_s|+|\Rc^n_s|$, where
$L\le\underline{{\rm Lip}}_T(\Fc,G)$ is the Lipschitz constant of $\Fc$ in $z$,
and passing from $\sum_i|\Delta Z^i_s|$ to $(\sum_i|\Delta Z^i_s|^2)^{1/2}$
costs a factor $\sqrt n$, which cannot be avoided since the particles are
exchangeable. Instead, we work under the measure $\Q^n$ of
{\rm \Cref{sec:n_player}}, under which the linear part of $\Delta\Fc$
is exactly compensated, and

\[
    \Delta Y_t
    =
    -\int_t^T\Rc^n_s\,\mathrm{d}s
    -
    \sum_{i=1}^n\int_t^T
    \Delta Z^i_s\cdot
    \bigl(\mathrm{d}W^{i,\Q}_s+\sigma_\circ\,\mathrm{d}W^\circ_s\bigr).
\]

Set $M_t:=\Delta Y_t-\int_0^t\Rc^n_s\,\mathrm{d}s$. By the previous identity,
$M$ is a continuous $\Q^n$--local martingale, with
\[
    M_T-M_t
    =
    \sum_{i=1}^n\int_t^T
    \Delta Z^i_s\cdot
    \bigl(\mathrm{d}W^{i,\Q}_s+\sigma_\circ\,\mathrm{d}W^\circ_s\bigr),
    \qquad
    M_T=-\int_0^T\Rc^n_s\,\mathrm{d}s .
\]
Moreover $M$ is bounded, by \eqref{eq:inequality_Y} and
\eqref{eq:remainder_bound}. It is therefore a square--integrable
$\Q^n$--martingale, also with respect to $\Hc^n$, to which it is adapted.
Computing its bracket,
\[
    \E^{\Q^n}\left[
        \int_t^T
        \left(
            \sum_{i=1}^n|\Delta Z^i_s|^2
            +
            \Bigl|\sum_{i=1}^n\sigma_\circ^\top\Delta Z^i_s\Bigr|^2
        \right)\mathrm{d}s
        \,\middle|\,\Hc^n_t
    \right]
    =
    \E^{\Q^n}\!\left[(M_T-M_t)^2\,\middle|\,\Hc^n_t\right]
    \le
    \E^{\Q^n}\!\left[M_T^2\,\middle|\,\Hc^n_t\right],
\]
the last inequality because $M_t=\E^{\Q^n}[M_T\mid\Hc^n_t]$. Since
$|M_T|\le TM_2/(2n)$ by \eqref{eq:remainder_bound}, we obtain
\begin{align} \label{eq:Z_sharp_estimate}
    \E^{\Q^n}\left[
        \int_t^T\sum_{i=1}^n|\Delta Z^i_s|^2\,\mathrm{d}s
        \,\middle|\,\Hc^n_t
    \right]
    \le
    \frac{C}{n^2}.
\end{align}
Since $\Delta Z^i_s=Z^{i,n}_s-\frac1n Z(s,N^i_s,\mu^n_s)$ and
$\Delta Y_t=Y^n_t-Y(t,\mu^n_t)$, squaring \eqref{eq:inequality_Y} and combining
it with \eqref{eq:Z_sharp_estimate} gives \eqref{eq:WBSDE_particle_estimate}.
Finally, since
$\frac1n\sum_i|nZ^{i,n}_s-Z(s,N^i_s,\mu^n_s)|^2=n\sum_i|\Delta Z^i_s|^2$, taking
$t=0$ and expectations in \eqref{eq:Z_sharp_estimate} yields the estimate on the
rescaled integrands, and the proof is complete.

\subsection{Proof of the connection with mean field control}

We now prove the verification result connecting the {\rm W-BSDE} with the
McKean-Vlasov control problem. The argument follows the classical BSDE
verification principle, but with the conditional law $\mu_t$ playing the
role of the Markovian state variable. The proof consists in showing that the W-BSDE gives both sides of the
verification argument: the upper bound for all controls and the optimality of
the feedback obtained from the Hamiltonian maximizer.

\medskip

We first consider the W-BSDE $(Y,Z)$ associated with $(\Fc,G)$. The argument is the same as in the absence of common noise, for which we refer to \cite{djete2025notionbsdewassersteinspace}. Indeed, using the equation satisfied by $(Y,Z)$ and applying the verification identity along the controlled conditional law flow
\[
    \mu^\alpha_s:=\Lc(X^{\nu,\alpha}_s\mid \Gc_s),
\]
we obtain, for every $\alpha\in\Ac$,
\begin{align*}
    J(\nu,\alpha)
    =
    Y(0,\nu)
    + \E\left[
        \int_0^T
        \Big(
            h\big(s,X^{\nu,\alpha}_s,\mu^\alpha_s,
                Z(s,X^{\nu,\alpha}_s,\mu^\alpha_s),
                \alpha(s,X^{\nu,\alpha}_s,\mu^\alpha_s)\big)
            -
            \Fc\big(s,X^{\nu,\alpha}_s,\mu^\alpha_s,
                Z(s,X^{\nu,\alpha}_s,\mu^\alpha_s)\big)
        \Big)\mathrm ds
    \right].
\end{align*}
Since
\[
    \Fc(s,x,m,z)=\sup_{a\in A}h(s,x,m,z,a),
\]
the integrand is nonpositive. Hence
\[
    J(\nu,\alpha)\le Y(0,\nu),
    \qquad \alpha\in\Ac.
\]
Moreover, for the feedback control
\[
    \widehat\alpha(s,x,m)
    :=
    \widehat a(s,x,m,Z(s,x,m)),
\]
the maximum is attained pointwise, and therefore equality holds:
\[
    J(\nu,\widehat\alpha)=Y(0,\nu).
\]
It follows that $\widehat\alpha$ is optimal and
\[
    V_{\rm MFC}(\nu)=Y(0,\nu).
\]

Let $\beta\in\Ac^\circ$, and let $\mu$ be the conditional law flow controlled by
$\beta$, with drift $b_1+b^\circ(\cdot,\beta)$. Writing the W-BSDE identity for
$(Y^{\alpha^\circ},Z^{\alpha^\circ})$ along $\mu$ at $t=0$ and taking
expectations, and using that $b^\circ$, $\alpha^\circ$ and $\beta$ do not depend
on $x$, the terms involving $b_1$ cancel and, with
$\bar Z_s:=\bar Z^{\alpha^\circ}(s,\mu_s)$,
\[
    Y^{\alpha^\circ}(0,\nu)
    =
    \E\Big[G(\mu_T)+\int_0^T\Big(
        h^\circ\bigl(s,\mu_s,\bar Z_s,\alpha^\circ(s,\mu_s)\bigr)
        -\bar Z_s\cdot b^\circ\bigl(s,\mu_s,\beta(s,\mu_s)\bigr)
    \Big)\mathrm ds\Big].
\]
Since
$J^\circ(\nu,\beta)=\E\bigl[G(\mu_T)+\int_0^T\int_{\R^d}L(s,x,\mu_s,\beta(s,\mu_s))\,\mu_s(\mathrm dx)\,\mathrm ds\bigr]$,
we obtain
\[
    J^\circ(\nu,\beta)-Y^{\alpha^\circ}(0,\nu)
    =
    \E\int_0^T\Big(
        h^\circ\bigl(s,\mu_s,\bar Z_s,\beta(s,\mu_s)\bigr)
        -
        h^\circ\bigl(s,\mu_s,\bar Z_s,\alpha^\circ(s,\mu_s)\bigr)
    \Big)\mathrm ds
    \le0,
\]
by the maximality of $\alpha^\circ$, with equality for $\beta=\alpha^\circ$.
Hence $\alpha^\circ$ is optimal and $V^\circ_{\rm MFC}(\nu)=Y^{\alpha^\circ}(0,\nu)$.

\subsection{Proof of the connection with mean field games}

We now prove the verification result connecting the {\rm W-BSDE} with mean
field games. The argument follows the same philosophy as in the mean field
control case, but with one additional fixed-point requirement: the flow of
measures used by the representative player must coincide with the conditional
law generated when the optimal feedback is applied.

\medskip
We prove the result in the one-dimensional case $\ell=1$ in order to simplify
notation, so that the enlarged state space is $\R^2$ and the common noise
coefficient on it is $\boldsymbol\sigma_\circ=(\sigma_\circ,\sigma_\circ)$. The
general case follows by the same argument, replacing products by Euclidean
scalar products.

\medskip
We first assume that all coefficients are smooth and that \eqref{eq:gradient_MFG} admits a smooth solution $p$. Recall that
\[
    \alpha^\star(t,x,\pi)
    :=
    \widehat{\alpha}\bigl(t,x,\pi,p(t,x,\pi)\bigr),
\]
and set
\[
    \overline b(t,x,\pi)
    :=
    b\bigl(t,x,\pi,\alpha^\star(t,x,\pi)\bigr)
    =
    \widehat b\bigl(t,x,\pi,p(t,x,\pi)\bigr).
\]
With this notation, the generator on the enlarged space $\R^2$ reads
\[
    \Fc_p(t,x,m,z)
    =
    H(t,x_1,m_2,z_1)
    +
    z_2\,\overline b(t,x_2,m_2),
\]
where $x=(x_1,x_2)$, $z=(z_1,z_2)$, and $m_2$ denotes the second marginal of $m\in\Pc(\R^2)$.

\medskip

We first record the derivatives of the terminal condition and of the generator. For $m\in\Pc(\R^2)$, one has
\begin{align*}
    \delta_m G(m)(x)
    =
    g(x_1,m_2)
    +
    \int_{\R}
    \delta_{m_2}g(y_1,m_2)(x_2)\,m_1(\mathrm{d}y_1).
\end{align*}
Similarly,
\begin{align*}
    \delta_m \Fc_p(t,x,m,z)(y)
    =
    \delta_{m_2}H(t,x_1,m_2,z_1)(y_2)
    +
    z_2\,\delta_{m_2}\overline b(t,x_2,m_2)(y_2),
\end{align*}
and
\begin{align*}
    \nabla_z \Fc_p(t,x,m,z)
    =
    \bigl(
        \nabla_{z_1}H(t,x_1,m_2,z_1),
        \overline b(t,x_2,m_2)
    \bigr).
\end{align*}

\medskip

We now embed the MFG characteristic equation into the characteristic equation associated with the enlarged W-BSDE. Define
\[
    \Pr^1(t,x,m)
    :=
    p(t,x_1,m_2),
    \qquad (t,x,m)\in[0,T]\times\R^2\times\Pc(\R^2).
\]
We then choose $\Pr^2$ as a smooth solution of the second component of the characteristic system associated with $(\Fc_p,G)$. More precisely, writing
\[
    \Pr=(\Pr^1,\Pr^2),
\]
let $U^{t,m}=(U^{1,t,m},U^{2,t,m})$ be such that $\Lc(U^{t,m}_t)=m$ and
\begin{align*}
    \mathrm{d}U^{t,m}_s
    =
    \nabla_z \Fc_p
    \bigl(
        s,U^{t,m}_s,\Ec^{t,m}_s,
        \Pr(s,U^{t,m}_s,\Ec^{t,m}_s)
    \bigr)\,\mathrm{d}s
    +
    \mathrm{d}W_s
    +
    \boldsymbol\sigma_\circ\,\mathrm{d}W^\circ_s,
    \qquad
    \Ec^{t,m}_s:=\Lc(U^{t,m}_s\mid\Gc_s).
\end{align*}
The second component $\Pr^2$ is required to satisfy, for $(t,x,m)\in[0,T)\times\R^2\times\Pc(\R^2)$,
\begin{align*}
    \Pr^2(t,x,m)
    &=
    \E\left[
        \delta_m G(\Ec^{t,m}_T)(N^{t,x}_T)
        \frac{W^2_T-W^2_t}{T-t}
    \right]
    \\
    &\quad+
    \E\left[
        \int_t^T
        \Fc_p\bigl(
            s,N^{t,x}_s,\Ec^{t,m}_s,
            \Pr(s,N^{t,x}_s,\Ec^{t,m}_s)
        \bigr)
        \frac{W^2_s-W^2_t}{s-t}\,\mathrm{d}s
    \right]
    \\
    &\quad+
    \E\left[
        \int_t^T
        \delta_m\Fc_p\bigl(
            s,N^{t,U}_s,\Ec^{t,m}_s,
            \Pr(s,N^{t,U}_s,\Ec^{t,m}_s)
        \bigr)(\widehat N^{t,x}_s)
        \frac{\widehat W^2_s-\widehat W^2_t}{s-t}\,\mathrm{d}s
    \right].
\end{align*}
Here $\widehat N$ and $\widehat W$ denote independent copies of $N$ and $W$ given the common noise, as in the general construction.

\medskip
By the specific form of $\Fc_p$, the conditional law flow $\Ec^{t,m}$ does not depend on the second component $\Pr^2$. Indeed, the dynamics of the second coordinate is
\[
    \mathrm{d}U^{2,t,m}_s
    =
    \overline b(s,U^{2,t,m}_s,\Ec^{t,m}_{s,2})\,\mathrm{d}s
    +
    \mathrm{d}W^2_s
    +
    \sigma_\circ\,\mathrm{d}W^\circ_s,
\]
and therefore
\[
    \Ec^{t,m}_{s,2}
    =
    \Lc(U^{2,t,m}_s\mid \Gc_s)
    =
    \Rr^{p,t,m_2}_s.
\]
Consequently, the first component of the characteristic system reduces exactly to \eqref{eq:gradient_MFG}. Since $\Pr^1(t,x,m)=p(t,x_1,m_2)$ depends on $x$ only through $x_1$, and similarly $\Pr^2$ only through $x_2$, and since $p$ solves
\eqref{eq:gradient_MFG}, a direct substitution shows that $\Pr=(\Pr^1,\Pr^2)$ solves the characteristic equation \eqref{eq:chac:z} associated with the enlarged generator $\Fc_p$ and the terminal condition $G$.

\medskip

By the existence theorem for the W-BSDE associated with a solution of the characteristic equation, there exists a W-BSDE solution
\[
    (Y_p,Z_p)=(Y_p,(Z^1_p,Z^2_p))
\]
associated with $(\Fc_p,G)$ such that, for every admissible flow $\nu$ on $\Pc(\R^2)$,
\[
    Z_p(t,x,\nu_t)
    =
    \Pr(t,x,\nu_t),
    \qquad
    \text{for a.e. }(t,x,\omega).
\]
In particular,
\[
    Z^1_p(t,(x_1,x_2),\nu_t)
    =
    p(t,x_1,\nu_{t,2}),
\]
where $\nu_{t,2}$ denotes the second marginal of $\nu_t$. As in {\rm \Cref{sec:optimal_control}}, we use here
the immediate extension of {\rm \Cref{thm:existence_WBSDE_gradient}} to
generators with linear growth in $z$, since $\Fc_p$ is affine in $z_2$.

\medskip

We now perform the verification argument. Let $\beta$ be an admissible Lipschitz feedback, and let $\nu=(\nu_t)_{t\in[0,T]}$ be the conditional law flow associated with the drift
\[
    (t,x,m)
    \longmapsto
    \bigl(
        b(t,x_1,m_2,\beta(t,x_1,m_2)),
        \overline b(t,x_2,m_2)
    \bigr),
\]
with initial condition
\[
    \nu_0=\Lc(X_0)\otimes\Lc(X_0).
\]
Let
\[
    \nu_t^i:=\Lc(X^i_t\mid\Gc_t),
    \qquad i=1,2,
\]
denote the two marginals of $\nu_t$. Applying the W-BSDE identity to this flow and using the expression of $\Fc_p$ and $Z^1_p$, we obtain
\begin{align*}
    Y_p(t,\nu_t)
    =
    &\E\bigg[
        g(X^1_T,\nu^2_T)
        +
        \int_t^T
        H\bigl(s,X^1_s,\nu^2_s,p(s,X^1_s,\nu^2_s)\bigr)\,\mathrm{d}s
        \\
        &\quad
        -
        \int_t^T
        p(s,X^1_s,\nu^2_s)\,
        b\bigl(s,X^1_s,\nu^2_s,\beta(s,X^1_s,\nu^2_s)\bigr)
        \,\mathrm{d}s
        \\
        &\quad
        -
        \int_t^T
        (p(s,X^1_s,\nu^2_s) + \Pr^2(s,X^2_s,\nu_s))\,
        \sigma_\circ\,\mathrm{d}W^\circ_s
        \,\bigg|\,\Gc_T
    \bigg].
\end{align*}
Taking expectations at $t=0$ and using the definition of the payoff gives
\begin{align*}
    J_{\nu^2}(\beta)
    =
    Y_p(0,\nu_0)
    +
    \E\left[
        \int_0^T
        \Big(
            h\bigl(s,X^1_s,\nu^2_s,p(s,X^1_s,\nu^2_s),
            \beta(s,X^1_s,\nu^2_s)\bigr)
            -
            H\bigl(s,X^1_s,\nu^2_s,p(s,X^1_s,\nu^2_s)\bigr)
        \Big)\,\mathrm{d}s
    \right].
\end{align*}
Since $H$ is the supremum of $h$ over $A$, we have
\[
    h\bigl(s,x,\pi,p(s,x,\pi),a\bigr)
    \le
    H\bigl(s,x,\pi,p(s,x,\pi)\bigr),
    \qquad a\in A.
\]
Therefore,
\[
    J_{\nu^2}(\beta)\le Y_p(0,\nu_0).
\]

Now choose $\beta=\alpha^\star$. By the definition of $\alpha^\star$ as a maximizer of the Hamiltonian,
\[
    h\bigl(s,x,\pi,p(s,x,\pi),\alpha^\star(s,x,\pi)\bigr)
    =
    H\bigl(s,x,\pi,p(s,x,\pi)\bigr).
\]
Hence $J_{\nu^2}(\alpha^\star)=Y_p(0,\nu_0)$, and so
\[
    J_{\nu^2}(\alpha^\star)
    =
    \sup_{\beta}J_{\nu^2}(\beta).
\]

It remains to identify the consistency condition. When $\beta=\alpha^\star$, the first coordinate dynamics becomes
\[
    \mathrm{d}X^1_t
    =
    \overline b(t,X^1_t,\nu^2_t)\,\mathrm{d}t
    +
    \mathrm{d}W^1_t
    +
    \sigma_\circ\,\mathrm{d}W^\circ_t,
\]
while the second coordinate satisfies
\[
    \mathrm{d}X^2_t
    =
    \overline b(t,X^2_t,\nu^2_t)\,\mathrm{d}t
    +
    \mathrm{d}W^2_t
    +
    \sigma_\circ\,\mathrm{d}W^\circ_t.
\]
By uniqueness in law for this McKean-Vlasov dynamics, and since the two coordinates have the same initial law, we obtain
\[
    \nu^1_t=\nu^2_t,
    \qquad t\in[0,T].
\]
Thus $\nu^2$ is consistent with the law of the optimally controlled state and is therefore a mean field game solution associated with $\alpha^\star$.

\medskip
Finally, if the coefficients are only in the regularity class stated in the proposition, the result follows by approximation. One first regularizes the coefficients and applies the previous smooth argument. The stability of the characteristic equation and of the associated W-BSDE then
allows one to pass to the limit, exactly as in the proof of {\rm~\Cref{thm:existence_WBSDE_gradient}} under {\rm \Cref{assum:gradient}}.

\subsection{Proof of the connection with McKean-Vlasov FBSDEs}

Let $p$ be the characteristic field associated with the W-BSDE solution,
so that
\[
    Z(t,x,m)=p(t,x,m).
\]
We fix $(e,m)\in[0,T]\times\Pc(\R^d)$. Let $S^{e,m}$ be the solution of
the characteristic McKean-Vlasov dynamics
\[
    S^{e,m}_t
    =
    N^{e,U}_t
    +
    \int_e^t
    \nabla_z\Fc
    \bigl(
        s,
        S^{e,m}_s,
        \ell^{p,e,m}_s,
        p(s,S^{e,m}_s,\ell^{p,e,m}_s)
    \bigr)\,\mathrm{d}s,
    \qquad \Lc(U)=m,
\]
where $\ell^{p,e,m}_t := \Lc(S^{e,m}_t\mid\Gc_T) = \Lc(S^{e,m}_t\mid\Gc_t)$. We also set
\[
    \overline \ell^{p,e,m}_t
    :=
    \Lc
    \bigl(
        S^{e,m}_t,
        p(t,S^{e,m}_t,\ell^{p,e,m}_t)
        \mid \Gc_T
    \bigr).
\]

Recall the auxiliary function
\[
\begin{aligned}
    q(t,x,m)
    &:=
    \E \left[
        \delta_mG(\ell^{p,t,m}_T)(N^{t,x}_T)
    \right]
    \\
    &\quad+
    \E \left[
        \int_t^T
        \Fc\left(
            s,N^{t,x}_s,\ell^{p,t,m}_s,
            p(s,N^{t,x}_s,\ell^{p,t,m}_s)
        \right)\,\mathrm{d}s
    \right]
    \\
    &\quad+
    \E \left[
        \int_t^T
        \delta_m \Fc\left(
            s,S^{t,m}_s,\ell^{p,t,m}_s,
            p(s,S^{t,m}_s,\ell^{p,t,m}_s)
        \right)(\widehat N^{t,x}_s)
        \,\mathrm{d}s
    \right],
\end{aligned}
\]
where $\widehat N^{t,x}$ denotes an independent copy of the heat process
starting from $x$, conditionally on the common noise. By the construction
used in the proof of {\rm \Cref{thm:existence_WBSDE_gradient}}, we have
\[
    \nabla_x q(t,x,m)=p(t,x,m).
\]

We claim that the triple
\[
    \left(
        S^{e,m}_t,\,
        q(t,S^{e,m}_t,\ell^{p,e,m}_t),\,
        p(t,S^{e,m}_t,\ell^{p,e,m}_t)
    \right)_{t\in[e,T]}
\]
is a solution of the McKean-Vlasov FBSDE
\eqref{eq:MKV_FBSDE_forward}-\eqref{eq:MKV_FBSDE_backward}.
The forward equation is immediate from the definition of $S^{e,m}$. Indeed,
if we define
\[
    \Yc_t
    :=
    q(t,S^{e,m}_t,\ell^{p,e,m}_t),
    \qquad
    \Zc_t
    :=
    p(t,S^{e,m}_t,\ell^{p,e,m}_t),
\]
then $S^{e,m}$ satisfies
\[
    S^{e,m}_t
    =
    N^{e,U}_t
    +
    \int_e^t
    \nabla_z\Fc
    \bigl(
        s,
        S^{e,m}_s,
        \ell^{p,e,m}_s,
        \Zc_s
    \bigr)\,\mathrm{d}s.
\]
Moreover,
\[
    \ell^{p,e,m}_t
    =
    \Lc(S^{e,m}_t\mid\Gc_T),
    \qquad
    \overline \ell^{p,e,m}_t
    =
    \Lc(S^{e,m}_t,\Zc_t\mid\Gc_T).
\]
Thus the conditional laws appearing in the FBSDE are exactly those generated
by the candidate solution.

It remains to verify the backward equation. Applying the Itô formula on the
Wasserstein space to the map
\[
    (t,x,m)\longmapsto q(t,x,m)
\]
along the flow $(S^{e,m}_t,\ell^{p,e,m}_t)$, and using the identity $\nabla_x q(t,x,m)=p(t,x,m)$, gives
\begin{equation}
\begin{aligned}
    \mathrm{d}\Yc_t
    &=
    -
    \bigg[
        \Fc
        \bigl(
            t,
            S^{e,m}_t,
            \ell^{p,e,m}_t,
            \Zc_t
        \bigr)
        +
        \int_{\R^d}\int_{\R^d}
        \delta_m\Fc
        \bigl(
            t,
            x,
            \ell^{p,e,m}_t,
            z
        \bigr)(S^{e,m}_t)
        \,
        \overline \ell^{p,e,m}_t(\mathrm{d}x,\mathrm{d}z)
    \bigg]\mathrm{d}t
    +
    \Zc_t\cdot \mathrm{d}S^{e,m}_t .
\end{aligned}
\end{equation}
Equivalently, integrating between $t$ and $T$,
\[
\begin{aligned}
    \Yc_t
    &=
    \Yc_T
    \\
    &\quad+
    \int_t^T
    \bigg[
        \Fc
        \bigl(
            s,
            S^{e,m}_s,
            \ell^{p,e,m}_s,
            \Zc_s
        \bigr)
        +
        \int_{\R^d}\int_{\R^d}
        \delta_m\Fc
        \bigl(
            s,
            x,
            \ell^{p,e,m}_s,
            z
        \bigr)(S^{e,m}_s)
        \,
        \overline \ell^{p,e,m}_s(\mathrm{d}x,\mathrm{d}z)
    \bigg]\mathrm{d}s
    -
    \int_t^T
    \Zc_s\cdot \mathrm{d}S^{e,m}_s .
\end{aligned}
\]
By the definition of $q$, or equivalently by taking $t=T$ in its
representation, we have
\[
    q(T,x,m)=\delta_mG(m)(x).
\]
Therefore,
\[
    \Yc_T
    =
    q(T,S^{e,m}_T,\ell^{p,e,m}_T)
    =
    \delta_mG(\ell^{p,e,m}_T)(S^{e,m}_T).
\]
Consequently,
\[
\begin{aligned}
    \Yc_t
    &=
    \delta_mG(\ell^{p,e,m}_T)(S^{e,m}_T)
    \\
    &\quad+
    \int_t^T
    \bigg[
        \Fc
        \bigl(
            s,
            S^{e,m}_s,
            \ell^{p,e,m}_s,
            \Zc_s
        \bigr)
        +
        \int_{\R^d}\int_{\R^d}
        \delta_m\Fc
        \bigl(
            s,
            x,
            \ell^{p,e,m}_s,
            z
        \bigr)(S^{e,m}_s)
        \,
        \overline \ell^{p,e,m}_s(\mathrm{d}x,\mathrm{d}z)
    \bigg]\mathrm{d}s
    -
    \int_t^T
    \Zc_s\cdot \mathrm{d}S^{e,m}_s .
\end{aligned}
\]
This is precisely the backward equation
\eqref{eq:MKV_FBSDE_backward}.
We have therefore shown that
\[
    \bigl(
        S^{e,m},
        q(\cdot,S^{e,m}_\cdot,\ell^{p,e,m}_\cdot),
        p(\cdot,S^{e,m}_\cdot,\ell^{p,e,m}_\cdot)
    \bigr)
\]
is a solution of the McKean-Vlasov FBSDE. By uniqueness of the solution to
\eqref{eq:MKV_FBSDE_forward}-\eqref{eq:MKV_FBSDE_backward}, it follows that
\[
    X^{e,m}_t=S^{e,m}_t,
    \qquad
    \Yc^{e,m}_t=q(t,S^{e,m}_t,\ell^{p,e,m}_t),
    \qquad
    \Zc^{e,m}_t=p(t,S^{e,m}_t,\ell^{p,e,m}_t),
\]
for $\mathrm{d}t\otimes\mathrm{d}\P$-a.e. $(t,\omega)$. In particular, since $p=Z$, we obtain
\[
    \Zc^{e,m}_t
    =
    p(t,X^{e,m}_t,\mu^{e,m}_t)
    =
    Z(t,X^{e,m}_t,\mu^{e,m}_t),
    \qquad
    \mathrm{d}t\otimes\mathrm{d}\P\text{-a.e.}
\]
Taking $t=e$, and using $\mu^{e,m}_e=m$, gives in particular  $\Zc^{e,m}_e  =  Z(e,X^{e,m}_e,m)$,
which is the desired identification. This completes the proof.

\bibliography{BSDE_McKVlasov_arxivVersion}

\end{document}